\documentclass[a4paper,11pt]{amsart}
\usepackage[T1]{fontenc}
\usepackage[latin1]{inputenc}
\usepackage{graphicx}
\usepackage{geometry}
\usepackage{amssymb, amsmath, amsfonts, amsthm}
\usepackage{url}
\usepackage{euscript}
\usepackage{verbatim} 
\usepackage[all]{xy}
\newcommand{\dott}{\, \cdot\,}
\numberwithin{equation}{section}

\DeclareMathOperator{\id}{Id}
\DeclareMathOperator{\meas}{meas}
\DeclareMathOperator{\supp}{supp}
\DeclareMathOperator*{\esssup}{esssup}
\newcommand{\act}{\cdot}
\newcommand{\tnorm}[1]{\vert\hspace*{-1pt}\vert\hspace*{-1pt}\vert#1\vert\hspace*{-1pt}\vert\hspace*{-1pt}\vert}

\newcommand{\Linf}{{L^\infty}}
\newcommand{\Winf}{{W^{1,\infty}}}
\newcommand{\Winfloc}{{W^{1,\infty}_{\rm loc}}}
\newcommand{\epsi}{\varepsilon}
\newcommand{\mus}{\mu_{\text{\rm s}}}
\newcommand{\Nat}{\mathbb N}
\newcommand{\smin}{\ensuremath{s_{\text{min}}}}
\newcommand{\smax}{\ensuremath{s_{\text{max}}}}
\newcommand{\inv}{\ensuremath{^{-1}}}
\newcommand{\WX}{\ensuremath{W_X^{1,\infty}}}
\newcommand{\WY}{\ensuremath{W_Y^{1,\infty}}}
\newcommand{\LX}{\ensuremath{L_X^\infty}}
\newcommand{\LY}{\ensuremath{L_Y^\infty}}

\renewcommand{\P}{\ensuremath{\mathcal{P}}}
\renewcommand{\H}{\ensuremath{\mathcal{H}}}
\newcommand{\N}{\ensuremath{\mathcal{N}}}

\newcommand{\X}{\ensuremath{\mathcal{X}}}
\newcommand{\Y}{\ensuremath{\mathcal{Y}}}

\newcommand{\C}{\ensuremath{\mathcal{C}}}
\newcommand{\B}{\ensuremath{\mathcal{B}}}
\newcommand{\Gr}{\ensuremath{G}}
\newcommand{\V}{\ensuremath{\mathcal{V}}}

\newcommand{\W}{\ensuremath{\mathcal{W}}}
\newcommand{\fracpar}[2]{\frac{\partial
    #1}{\partial #2}}

\newcommand{\ldX}{\fracpar{}{X}}
\newcommand{\ldY}{\fracpar{}{Y}}

\newcommand{\dXY}[1]{{#1}_{XY}}

\newcommand{\abs}[1]{\left\vert#1\right\vert}
\newcommand{\Real}{\mathbb R}

\newcommand{\norm}[1]{\left\Vert#1\right\Vert}
\newcommand{\norms}[1]{\Vert#1\Vert}
\newcommand{\Z}{\ensuremath{\mathcal{Z}}}
\newcommand{\E}{\ensuremath{\mathcal{E}}}

\newcommand{\D}{\ensuremath{\mathcal{D}}}
\newcommand{\G}{\ensuremath{\mathcal{G}}}
\newcommand{\muac}{\mu_{\text{\rm ac}}}
\newcommand{\nuac}{\nu_{\text{\rm ac}}}

\newcommand{\F}{\ensuremath{\mathcal{F}}}
\newcommand{\Fquot}{\ensuremath{\mathcal{F}/\Gr^2}}

\newcommand{\Mb}{\ensuremath{\mathbf{M}}}
\newcommand{\Pib}{\ensuremath{\mathbf{\Pi}}}
\newcommand{\Eb}{\ensuremath{\mathbf{E}}}
\newcommand{\Sb}{\ensuremath{\mathbf{S}}}
\newcommand{\Db}{\ensuremath{\mathbf{D}}}
\newcommand{\Lb}{\ensuremath{\mathbf{L}}}
\newcommand{\Cb}{\ensuremath{\mathbf{C}}}
\newcommand{\tb}{\ensuremath{\mathbf{t}}}
\newcommand{\xb}{\ensuremath{\mathbf{x}}}

\newcommand{\boldn}{\mathbf n}

\allowdisplaybreaks
\newtheorem{theorem}{Theorem}[section]

\newtheorem{lemma}[theorem]{Lemma}
\newtheorem{definition}[theorem]{Definition}

\begin{document}

\title[The nonlinear variational wave
equation]{Global semigroup of conservative solutions of the
  nonlinear variational wave equation}
\author[Holden]{Helge Holden}
\address[Holden]{\newline Department of
  Mathematical Sciences, Norwegian University of
  Science and Technology, NO--7491 Trondheim,
  Norway,\newline {\rm and} \newline Centre of
  Mathematics for Applications,
University of Oslo,
  P.O.\ Box 1053, Blindern,
  NO--0316 Oslo, Norway }
\email[]{holden@math.ntnu.no}
\urladdr{www.math.ntnu.no/\~{}holden/}

\author[Raynaud]{Xavier Raynaud}
\address[Raynaud]{\newline
  Centre of Mathematics for Applications, 
University of Oslo,
  P.O.\ Box 1053, Blindern,
  NO--0316 Oslo, Norway }
\email[]{raynaud@cma.uio.no}

\date{\today}



\thanks{Supported in part by the Research Council of Norway. This paper was written as part of  the international research program on Nonlinear Partial Differential Equations at the Centre for Advanced Study at the Norwegian Academy of Science and Letters in Oslo during the academic year 2008--09.}

\begin{abstract} 
We prove the existence of a global semigroup for conservative solutions of the nonlinear variational wave
equation $u_{tt}-c(u)(c(u)u_x)_x=0$. We allow for initial data $u|_{t=0}$ and
$u_t|_{t=0}$ that contain measures. We assume that $0<\kappa^{-1}\le c(u) \le
\kappa$. Solutions of this equation may experience concentration of the energy density 
$(u_t^2+c(u)^2u_x^2)dx$ into sets of measure zero.  The solution is constructed by introducing new variables related to the characteristics, whereby singularities in the energy density become  manageable. Furthermore, we prove that the energy may only focus on a set of times of zero measure or at points where $c'(u)$ vanishes.  A new numerical method to
construct conservative solutions is provided and illustrated on examples. 
\end{abstract}

\maketitle
\tableofcontents

\section{Introduction}

The nonlinear variational wave equation (NVW), which was first
introduced by Saxton in
\cite{Saxt:89}, is given
by the following nonlinear partial differential equation on the line
\begin{equation}
  \label{eq:nvw}
  u_{tt}-c(u)(c(u)u_x)_x=0
\end{equation}
with initial data
\begin{equation}
  \label{eq:initdat}
  u|_{t=0}=u_0,\ u_t|_{t=0}=u_1.
\end{equation}
The equation can be derived from the variational
principle applied to the functional
\begin{equation*}
  \iint\left(u_t^2-c^2(u)u_x^2\right)\,dxdt.
\end{equation*}
We are interested in the analysis of conservative solutions of this
initial value problem for $u_0, u_1\in L^2(\Real)$.   It is well known
that solutions of this equation develop singularities in finite time,
even for smooth initial data, see, e.g., \cite{GlaHunZh:96}. The continuation past
singularities is highly nontrivial, and allows for several distinct
solutions. Thus additional information or requirements are needed to
select a unique solution, and stability of solutions becomes a particularly
delicate issue. 
We here study the conservative case where one in addition to the
solution $u$ itself, requires that the energy is conserved. For
smooth solutions the energy is given by
$\E(t)=\int_{\Real}(u_t^2+c^2u_x^2)(t,x)\,dx$. However, as energy may focus
in isolated points, one has to look at energy density in the sense of
measures such that the absolutely continuous part of the measure
corresponds to the usual energy density.  The analysis resembles to a
large extent recent work done on the Camassa--Holm equation and the
Hunter--Saxton equation (see, e.g., \cite{HolRay:07,BreCons:06,ZhaZhen:00,HS:91} and references
therein). Our main result is the proof of the existence of a global semigroup for conservative solutions of the NVW equation, allowing for concentration of the energy density on sets of zero measure.  

The NVW equation has been extensively studied by Zhang and Zheng \cite{ZhaZhen:98,ZhaZhen:00,ZhaZhen:01,ZhaZhen:01a,ZhaZhen:03,ZhaZhen:05a,ZhaZhen:05}. 
However, our approach is closely related to the approach by Bressan and Zheng
\cite{BreZhe:06}, in that we introduce new variables based on the
characteristics, thereby, loosely speaking, separating waves going in positive and
negative direction. 

It is difficult to illustrate the ideas in this paper as there are no
elementary and explicit solutions available, except for the trivial
case where $c$ is constant, which yields the classical linear wave
equation. Thus one is forced to illustrate ideas numerically. Traditional finite
difference schemes will not yield conservative solutions, but rather
dissipative solutions due to the intrinsic numerical diffusion in these
methods. Hence it is a challenge of separate interest to compute numerically conservative
solutions of this equation to display some of the intricacies. This question is addressed and analyzed in
Section \ref{sec:num}.

Let us now turn to a more precise description of the content of this paper.
We consider the variables $R$ and $S$ defined as
\begin{equation}
  \label{eq:defRS}
  \left\{
    \begin{aligned}
      R&=u_t+c(u)u_x,\\
      S&=u_t-c(u)u_x.
    \end{aligned}
  \right.
\end{equation}
By \eqref{eq:nvw}, we have
\begin{equation}
  \label{eq:evolRS}
  \left\{
    \begin{aligned}
      R_t-cR_x&=\frac{c'}{4c}(R^2-S^2),\\
      S_t+cS_x&=\frac{c'}{4c}(S^2-R^2),
    \end{aligned}
  \right.
\end{equation}
or, on conservative form, 
\begin{equation}
  \label{eq:evolRSsq}
  \left\{
    \begin{aligned}
      (R^2+S^2)_t-(c(R^2-S^2))_x&=0,\\
      (\frac{1}{c}(R^2-S^2))_t-(R^2+S^2)_x&=0.
    \end{aligned}
  \right.
\end{equation}
Let $\E(t)$ denote the total energy of the system
at time $t$, i.e.,
\begin{equation}
  \label{eq:energdef}
  \E(t)=\int_{\Real}(u_t^2+c^2u_x^2)(t,x)\,dx=\int_\Real(R^2+S^2)\,dx.
\end{equation}
We assume that the initial total energy that we
denote $\E_0$ is finite and that $u$ is bounded in
$L^\infty$. For smooth solutions of \eqref{eq:nvw}
we have $\frac{d\E}{dt}=0$. We also assume that
$c\in C^1(\Real)$ and
$c\colon\Real\to[\kappa^{-1},\kappa]$ for some
constant $\kappa>0$. 

From \eqref{eq:energdef} we see that we need that
the functions $R$ and $S$ belong to
$L^2(\Real)$. It turns out that, as time evolves,
the functions $R^2$ and $S^2$ can concentrate on
sets of measure zero. The example presented in
Figure \ref{fig:difftimes}, see Section \ref{sec:nonlin}, illustrates
this phenomenon. In this case, we have a
nontrivial solution $u$ for $t$ nonzero, which is
however is identically equal to one at
$t=0$ and $u_t(0,x)=0$. However, when we analyze this example
closer, we see that the energy concentrates at the
origin, indeed
\begin{equation*}
  \lim_{t\to0}R^2(t,x)\,dx=\delta\quad\text{ and }\quad\lim_{t\to0}S^2(t,x)\,dx=2\delta
\end{equation*}
\begin{figure}[h]
  \centering
  \includegraphics[width=3cm]{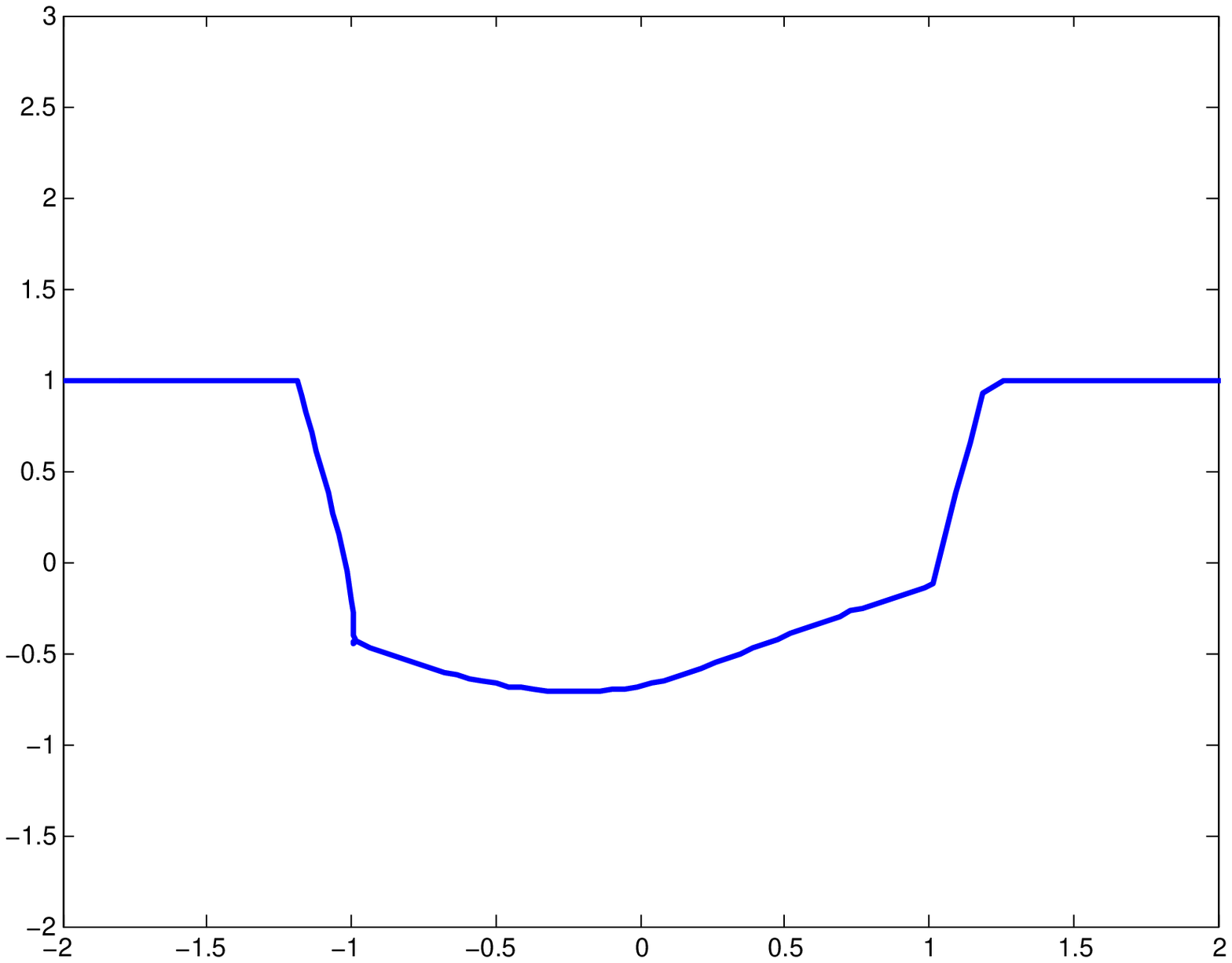}
  \includegraphics[width=3cm]{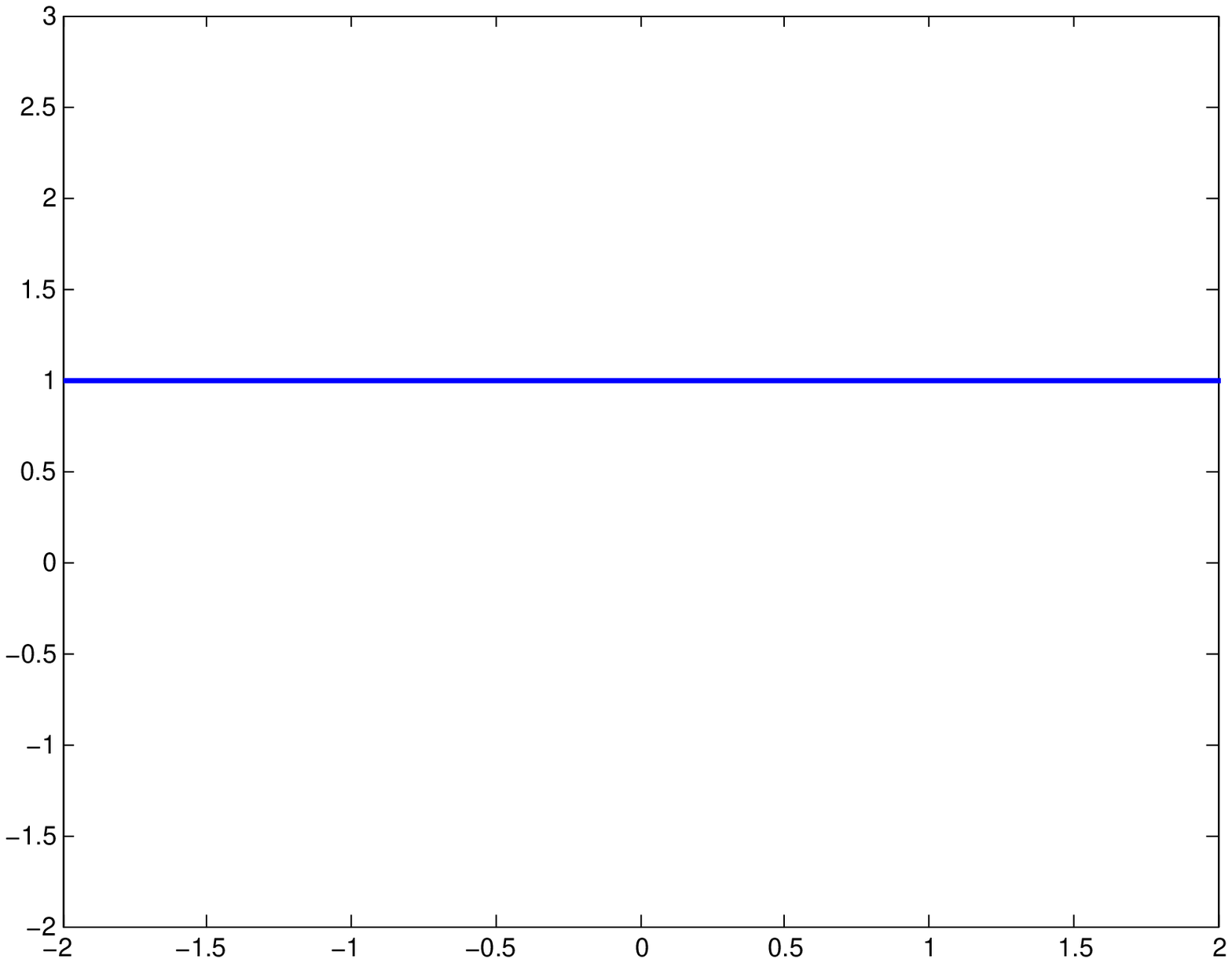}
  \includegraphics[width=3cm]{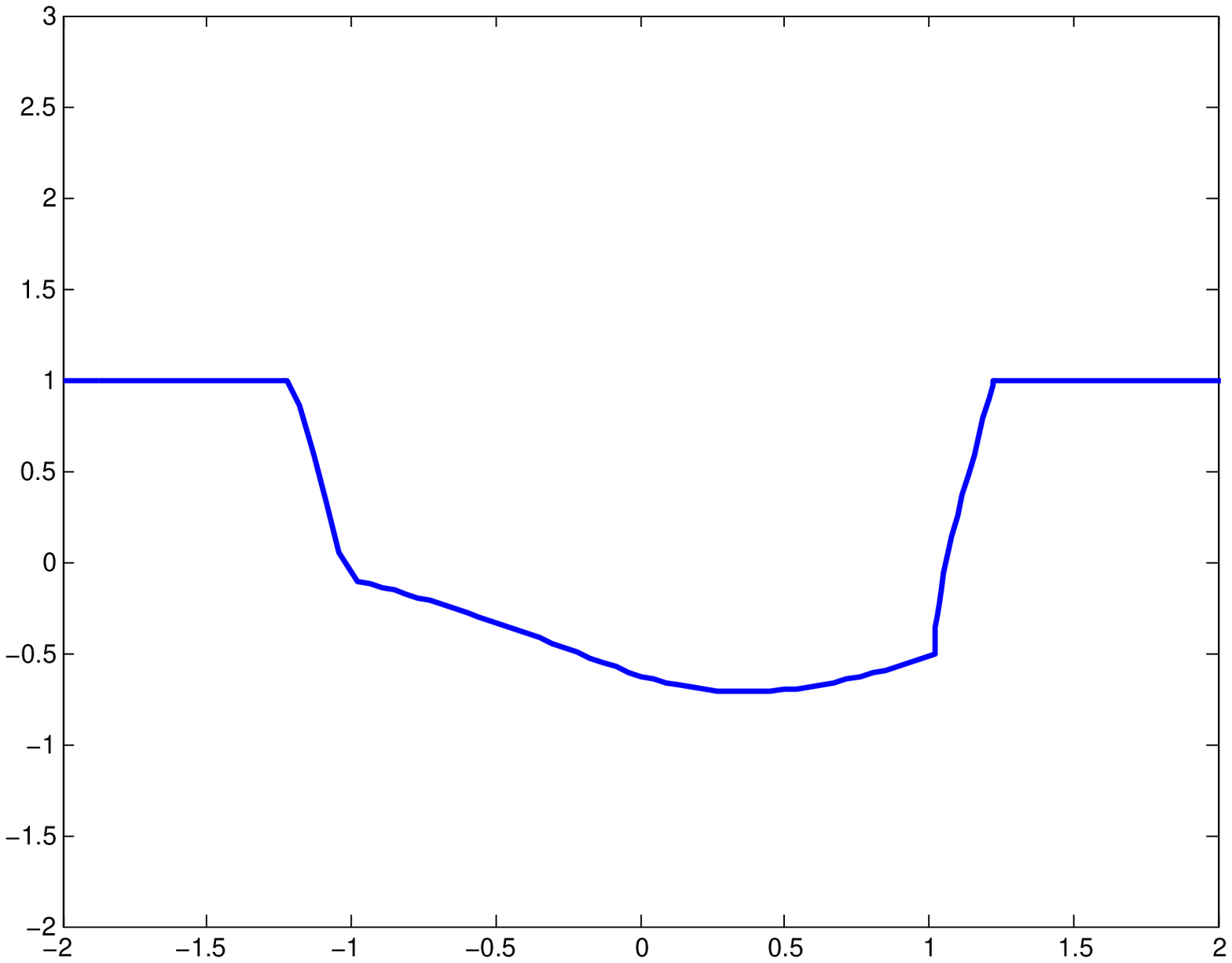}
  \caption{Plot of $u(t,x)$ for $t=-3$ (left), $t=0$ (center), $t=3$ (right).}
  \label{fig:difftimes}
\end{figure}
where $\delta$ is Dirac's delta function. Clearly
this complicates the existence and uniqueness
question for this equation.  As we want to
construct a semigroup of solutions for this type
of solutions, we have to know the location and
amount of backward ($R^2$) and forward ($S^2$)
energy that has concentrated on sets of zero
measure, an information which is not given by the
function $u$ itself. (In our example, since, at
$t=0$, the function $u$ is identically one and its
time derivative $u_t$ identically zero, we cannot infer where the
energy has concentrated.)  Thus we introduce the
set $\D$ whose elements, in addition to $u$, $R$,
$S$, contain two measures, $\mu$ and $\nu$, corresponding to forward
and backward energy density. More precisely, the measures are nonnegative Radon measures that satisfy
\begin{equation*}
 \muac=\frac14 R^2\,dx,\quad\nuac=\frac14 S^2\,dx.
\end{equation*}

Our main contribution in this article is
to present a rigorous construction of the
semigroup of conservative solutions in $\D$. Note
that the set $\D$ is the natural set of solutions for
conservative solutions, and the semigroup property
can only be established in $\D$, as 
illustrated by the example of Figure
\ref{fig:difftimes}.  Furthermore, by incorporating the energy measures as independent variables the formation of singularities is natural, and it allows for more general initial data.  The present approach also provides a natural numerical method for conservative solutions.  

As in \cite{BreZhe:06}, the construction of the
solutions is achieved via a change of variables
into a new coordinate system $(X,Y)$ that
straightens the characteristics. Even if we
use different variables, the solutions we obtain
are the same, but by extending the solutions to
the set $\D$, we are able to establish that the
solutions we construct satisfy the semigroup
property. We have to study in details the change
of variables mapping --- from the original
variables to the new variables and vice versa ---
because, in order to prove the semigroup property,
we have to establish that the two sets of
variables match in an appropriate way. Compared to
the variables used in \cite{BreZhe:06}, we prefer
variables with a more direct physical
interpretation. Namely, the variables we are
considering are time, $t(X,Y)$, space, $x(X,Y)$,
the solution function $U(X,Y)$, which formally
satisfies $u(t(X,Y),x(X,Y))=U(X,Y)$ and the energy
potentials $J$ and $K$. The definition of the
energy potentials $J$ and $K$ follows from
\eqref{eq:evolRSsq}, which says that the forms
$\frac14(R^2+S^2)\,dx+\frac14c(u)(R^2-S^2)\,dt$
and
$\frac1{4c(u)}(R^2-S^2)\,dx+\frac14(R^2+S^2)dt$
are closed, so that, by Poincar\'e's lemma, there
exist functions, here denoted the energy
potentials $J$ and $K$, whose differentials are
equal to the given forms. Thus the new set of
variables we will be considering equals
$Z=(t,x,U,J,K)$ and, after rewriting the governing
equations \eqref{eq:defRS} and \eqref{eq:evolRS}
in the new coordinate system $(X,Y)$, we get a
system of equations of the form
\begin{equation}
  \label{eq:eqsysintro}
  Z_{XY}=F(Z)(Z_X,Z_Y)
\end{equation}
where $F(Z)\colon\Real^5\times\Real^5\to\Real^5$ is a
bi-linear and symmetric operator, which depends
only on $U$, cf.~\eqref{eq:goveq}. 

In the new coordinates, the initial data 
corresponds to the set
$\Gamma_0=\{(X,Y)\in\Real^2\mid t(X,Y)=0\}$. In the smooth case, $\Gamma_0$ will be a strictly monotone curve. However, in our setting, $\Gamma_0$  may not even be a curve, and even if it is a curve, it may not be continuous nor strictly monotone. Indeed, it may contain horizontal and vertical segments, and furthermore, rectangular boxes corresponding to the situation where both $\mu$ and $\nu$ are singular at the same point.  If $\Gamma_0$ is a curve
with no vertical or horizontal parts and
the initial data is bounded in $L^\infty$ (by initial data, we
mean the values of $Z$, $Z_X$ and $Z_Y$ on
$\Gamma_0$), then the existence and uniqueness of
solutions to \eqref{eq:eqsysintro} is a classical
result, see, for example, \cite[Ch.~4]{garabedian}. In the present paper we
have to deal with unbounded data in $\D$ ($u_x$ and $u_t$ are unbounded in $L^\infty$). The new 
coordinates $(X,Y)$ are given by
\begin{equation*}
  dx-c(u)dt=0\text{ if and only if } dY=0
\end{equation*}
and
\begin{equation*}
  dx+c(u)dt=0\text{ if and only if } dX=0,
\end{equation*}
that is,  the characteristics  are mapped
to  horizontal and vertical lines.
We denote by $\Lb$ the mapping from the possible initial data in $\D$ to the set $\F$ defined by 
$\Gamma_0$ and the value of the initial data on $\Gamma_0$, thus $\Lb\colon\D\to\F$, see Definition \ref{def:mappingL}.  From 
$\Gamma_0$ we have to select one curve that can be used as initial data for the equation 
\eqref{eq:eqsysintro}.  There is a certain nonuniqueness due to fact that $\Gamma_0$ may not be curve. Let $\G_0$ denote the set of all curves, including the information about the initial data. We let $\Cb$ denote the mapping that from a set $\Gamma_0$ selects one possible curve, that is, 
$\Cb\colon \F\to \G_0$, see Definition \ref{def:C}. The inverse map that from curve determines the corresponding set in $\F$ is denoted $\Db$, see Definition \ref{def:D}. Once we have a curve with the initial data, we can in principle compute the solution by solving \eqref{eq:eqsysintro}. To show the existence of a global solution  we use the the bi-linearity of \eqref{eq:eqsysintro} and an a priori bound on the
energy potentials $J$ and $K$, see Section \ref{sec:existence}. We let the set of all possible solutions be denoted by $\H$, and let $\Sb\colon\G\to\H$ denote the map that computes the solution that passes through the curve in $\G$, see Theorem \ref{th:globalsol}. Here $\G$ is defined as $\G_0$ without the constraint that $t=0$, see Definition \ref{def:setG}. Recall that as $t$ now is a dependent variable, it does not make sense to compute the solution up to a specific time, but rather we determine the global solution for all times. Thus we need a mapping that extracts the solution $Z$ for a given time $T$, that is, the intersection of the solution in $\H$ with the set where $t(X,Y)=T$.  Let $\Eb\colon\H\to \G_0$ denote the map that from any given solution in $\H$ extracts the solution at $t=0$, that is, in $\G_0$, see Definition \ref{def:Eb}.  Next we define the operator $\tb_T\colon \H\to\H$ that shifts time in a solution in $\H$ by a given time $T$, see Definition \ref{def:tT}.  Now we can define the map $S_T\colon\F\to \F$ by 
$S_T=\Db\circ \Eb\circ\tb_T\circ \Sb\circ \Cb$, see Definition \ref{def:ST}. A key result is that $S_T$ is a semigroup on $\F$, see Theorem \ref{th:Stsemigroup}.
Next we need to return to the original variables. Let $\Mb\colon\F\to\D$ denote that map, see Definition \ref{def:M}.  Thus the solution operator $\bar S_T\colon\D\to\D$ is defined by 
(Definition \ref{def:bST}) 
\begin{equation}
  \label{eq:semigprop}
  \bar S_T=\Mb\circ S_T\circ\Lb.
\end{equation}
It remains to show that $\bar S_T$ is a semigroup.
However, since $\Mb$ is not inverse of $\Lb$, as
$\Lb\circ\Mb\neq\id_{\F}$, the semigroup property
of $\bar S_T$ still does not follow from
\eqref{eq:semigprop}. This fact is explained as
follows. When changing variables, we have
introduced a degree of freedom that we now want to
eliminate. This degree of freedom can be
identified precisely with the action of the group
$\Gr^2$, where $\Gr$ denotes the group of
diffeomorphisms of the real line. Indeed, by simply
using the bi-linearity of \eqref{eq:eqsysintro},
one can check that if $Z$ is a solution to
\eqref{eq:eqsysintro}, then $\bar
Z(X,Y)=Z(f(X),g(Y))$, where $(f,g)\in\Gr^2$, is
also a solution to the same equation. The
transformation $(X,Y)\mapsto(f(X),g(Y))$
corresponds to a stretching of the plane $\Real^2$
in the $X$ and $Y$ directions. Note that this
transformation maps horizontal (resp. vertical)
lines to horizontal (resp. vertical) lines and
therefore preserves the directions of the
characteristics.  Moreover, this transformation
does not affect the solution in the original 
coordinates. To illustrate this we ignore for the moment for the sake
of simplicity, the energies $\mu$ and
$\nu$ in the definition of $\D$. The solution $u(t,x)$ can be seen as the
surface in $\Real^3$ given by $(t,x,u(t,x))$ where
$(t,x)\in\Real^2$ are parameters. Through our
change of variables, we obtain another
parametrization of the same surface, namely,
\begin{equation}
  \label{eq:parsurf}
  (t(X,Y),x(X,Y),U(X,Y))
\end{equation}
where $(X,Y)\in\Real^2$ are the new
parameters. Additional properties of the solution
$Z=(t,x,U,J,K)$ which are contained in the
definition of $\H$ guarantee that the surface
defined by \eqref{eq:parsurf} does not fold over
itself so that it is in fact a graph. It is then
clear from \eqref{eq:parsurf} that the
transformation $(X,Y)\mapsto(f(X),g(Y))$ is simply
a re-parametrization of the same surface, which
defines  $u(t,x)$ uniquely. At the level of the set
$\F$, which corresponds to a parametrization of
the initial data in the new  coordinates, we
can also define the action of the group $\Gr^2$
that denote $\psi\times(f,g)\mapsto\psi\act(f,g)$
for any $\psi\in\F$ and $(f,g)\in\Gr^2$. We prove
that two elements which are equivalent correspond
to the same element in $\D$, that is,
\begin{equation}
  \label{eq:equivsaD}
  \Mb(\bar\psi)=\Mb(\psi)
\end{equation}
where $\bar\psi=\psi\act(f,g)$ for some
$(f,g)\in\Gr^2$. From \eqref{eq:equivsaD}, it is
now clear why $\Lb\circ\Mb\neq\id_{\F}$ as, in
general, $\bar\psi$ and $\psi$ are distinct. We
introduce a subset $\F_0$ of $\F$ which
corresponds to a section of $\F$ with respect to
the action of the group $\Gr^2$, which means that
the set $\F_0$ contains only one representative of
each equivalence class so that $\Fquot$ and $\F_0$
are in bijection. The system \eqref{eq:eqsysintro}
preserves the strict positivity of the quantities
$x_X+J_X$ and $x_Y+J_Y$ and the set $\F$ somehow
inherits this property which makes it possible to
define the projection $\Pi\colon\F\to\F_0$. The
projection $\Pi$ associates to any element in $\F$
its unique representative in $\F_0$ which belongs
to the same equivalence class. As expected, since
we have now eliminated the degree of freedom we
introduced by changing variables, we obtain that
$\F_0$ and $\D$ are in bijection. We are then able
to prove that $\bar S_t$ is a semigroup.

Our main result, Theorem \ref{th:main}, reads as follows:\\
\textbf{Theorem.} \textit{Given $(u_0,R_0,S_0,\mu_0,\nu_0)\in\D$, let us
  denote $(u,R,S,\mu,\nu)(t)=\bar
  S_t(u_0,R_0,S_0,\mu_0,\nu_0)$. Then $u$
  is  a weak solution of the nonlinear variational
  wave equation \eqref{eq:nvw}, that is,
  \begin{equation}
    \label{eq:soldist1A}
    \int_{\Real^2}(\phi_t-(c(u)\phi)_x)R\,dxdt+\int_{\Real^2}(\phi_t+(c(u)\phi)_x)S\,dxdt=0
  \end{equation}
  for all smooth functions $\phi$ with compact
  support and where
  \begin{equation}
    \label{eq:weakderu1A}
    R=u_t+c(u)u_x,\quad S=u_t-c(u)u_x.
  \end{equation}
  Moreover, the measures $\mu(t)$ and $\nu(t)$
  satisfy the following equations in the sense of
  distribution
  \begin{subequations}
    \label{eq:meassolA}
    \begin{equation}
      \label{eq:meassol1A}
      (\mu+\nu)_t-(c(\mu-\nu))_x=0
    \end{equation}
    and
    \begin{equation}
      \label{eq:meassol2A} 
      (\frac{1}{c}(\mu-\nu))_t-(\mu+\nu)_x=0.
    \end{equation}
  \end{subequations}
  The mapping $\bar S_T:\D\to\D$ is a semigroup, that is,
  \begin{equation*}
    \bar S_{t+t'}=\bar S_t\circ \bar S_{t'}
  \end{equation*}
  for all positive $t$ and $t'$.
}

Furthermore, we note the following important result (Theorem \ref{th:energconcentration}):\\
\textbf{Theorem.}  \textit{The solution satisfies the following properties:
 \begin{enumerate}
  \item[(i)] For all $t\in\Real$
    \begin{equation}
      \label{eq:prestotenergA}
      \mu(t)(\Real)+\nu(t)(\Real)=\mu_0(\Real)+\nu_0(\Real).
    \end{equation}
  \item[(ii)] For almost every $t\in\Real$, the
    singular part of $\mu(t)$ and $\nu(t)$ are
    concentrated on the set where $c'(u)=0$.
  \end{enumerate}
}

\smallskip
In this article, we do not study the stability of
the solutions. However, since the solutions we
obtain coincide with the ones obtained in
\cite{BreZhe:06} for initial data which do not
contain any singular measure, the solutions in
that case also satisfy the stability result stated
in \cite[Theorem 2]{BreZhe:06}. To obtain a
continuous semigroup of solution in $\D$, we would
like to follow the approach developed in
\cite{HolRay:07}, \cite{BrHoRa:09} for the
Camassa--Holm equation and Hunter--Saxton
equations. In these two papers, the conservative
solutions are also obtained via a change of
variables which is invariant with respect to
relabeling (i.e., with respect to the action of
the group of diffeomorphisms $\Gr$). We define a
distance between equivalence classes in the new 
coordinates. This distance is then mapped back
to the original set of coordinates so that we
obtain a continuous semigroup for this metric. In
the case of the nonlinear wave equation, in
particular, because of the truly two dimensional
nature of the problem, it is not so easy to
formulate a stability result in the new 
coordinates which holds when mapping back to the
original set of variables. In
Lemma \ref{lem:stabL2} we  present a result
in that direction.

There is a lack of explicit solutions to NVW. In this paper we consider two explicit examples. The first example, see Section \ref{sec:linearWE}, is the simplest possible, namely the linear wave equation (with $c$ constant), but with general initial data. We recover as expected the familiar d'Alembert solution. The energy measures are transported with velocity $\pm c$.  The second example, see Section \ref{sec:nonlin},  is a truly nonlinear case with velocity given by \eqref{eq:fartLC}. However, here we choose the simplest nontrivial initial data with energy concentration initially for both measures. The corresponding equation \eqref{eq:eqsysintro} is solved numerically, and the result is illustrated on Figs.~\ref{fig:difftimes}, \ref{fig:isotimes}--\ref{fig:paramsurface}.

The numerical method that yields conservative solutions is described in Section \ref{sec:num}.

\subsection{Physical motivation for the nonlinear variational wave equation} \label{sec:physics}

The NVW equation was first derived in the context of nematic liquid crystals, see 
\cite{Saxt:89,HS:91}. More precisely, a nematic crystal can be described, when we ignore the motion of the fluid,  by the dynamics of the so-called director field $\boldn=\boldn(x,y,z,t)\in\Real^3$ describing the orientation of rod-like molecules.  Thus $\abs{\boldn}=1$.  The Oseen--Franck strain-energy potential is given by
\begin{equation}
W(\boldn,\nabla\boldn)=\alpha\abs{\boldn \times(\nabla\times\boldn)}^2+\beta(\nabla\cdot\boldn)^2+\gamma(\boldn\cdot\nabla\times\boldn)^2,
\end{equation}
where $\alpha,\beta, \gamma$ are constitutive constants.  Consider next the highly simplified case of director fields of the type
\begin{equation}
\boldn=\boldn(x,t)=\cos(u(t,x)) \mathbf e_x+\sin(u(t,x)) \mathbf e_y
\end{equation}
where $\mathbf e_x$ and $\mathbf e_y$ are unit vectors in the $x$ and $y$ direction, respectively. In this case the functional $W(\boldn,\nabla\boldn)$ vastly simplifies to
\begin{equation}
W(\boldn,\nabla\boldn)=(\beta\cos^2 u+\alpha\sin^2 u)u_x^2,
\end{equation}
and $\abs{\boldn_t}^2=u_t^2$.
The dynamics is described by the variational  principle
\begin{equation}
\frac{\delta}{\delta u}\iint \big(u_t^2-c^2(u)u_x^2 \big) dxdt=0,
\end{equation}
where
\begin{equation}\label{eq:fartLC}
c^2(u)=\beta\cos^2 u+\alpha\sin^2 u,
\end{equation}
which results in the nonlinear variational wave equation 
\begin{equation}
   u_{tt}-c(u)(c(u)u_x)_x=0.
\end{equation}

\section{Equivalent system for the NVW equation}
\label{sec:equivsys}

In this section, we assume the existence of a
smooth solution $u=u(t,x)$ to \eqref{eq:nvw}. We
introduce the change of variables
$(t,x)\mapsto(X,Y)$ which straightens out the
characteristics: The forward characteristics,
which are given by the solutions of
$\frac{dx}{dt}=c(u(t,x(t)))$, are mapped to the
horizontal lines while the backward
characteristics, which are given by the solutions
of $\frac{dx}{dt}=-c(u(t,x(t)))$, are mapped to
the vertical lines. Formally, we can rewrite these
conditions as
\begin{equation}
  \label{eq:formch1}
  dx-c(u)dt=0\text{ if and only if } dY=0
\end{equation}
and
\begin{equation}
  \label{eq:formch2}
  dx+c(u)dt=0\text{ if and only if } dX=0.
\end{equation}
Our goal now is to rewrite the governing equation
\eqref{eq:nvw} in terms of the new variables
$(X,Y)$. The variables $(t,x)$ become functions of
$(X,Y)$ that we denote $t(X,Y)$ and $x(X,Y)$. We
set
\begin{equation}
  \label{eq:Uequ}
  U(X,Y)=u(t,x).
\end{equation}
Since
\begin{equation*}
  dx=x_XdX+x_YdY
  \ \text{ and }\   dt=t_XdX+t_YdY,
\end{equation*}
we obtain from \eqref{eq:formch1} and
\eqref{eq:formch2} that
\begin{equation}
  \label{eq:chardef}
  x_X=c(U)t_X\ \text{ and }\ x_Y=-c(U)t_Y.
\end{equation}
From \eqref{eq:evolRSsq}, we infer that the forms
$\frac14(R^2+S^2)dx+\frac{c}4(R^2-S^2)dt$ and
$\frac1{4c}(R^2-S^2)dx+\frac14(R^2+S^2)dt$ are
closed. Therefore, by Poincar\'e's lemma, we infer
the existence of two functions $J$ and $K$ for
which these forms are the differentials, that is,
\begin{equation}
  \label{eq:dJdef}
  dJ=\frac14(R^2+S^2)dx+\frac{c}4(R^2-S^2)dt
\end{equation}
and 
\begin{equation}
  \label{eq:dKdef}
  dK=\frac{1}{4c}(R^2-S^2)dx+\frac{1}{4}(R^2+S^2)dt.
\end{equation}
We have, after using \eqref{eq:chardef}, 
\begin{align*}
  dJ&=\frac14(R^2+S^2)dx+\frac{c}4(R^2-S^2)dt\\
  &=\frac14(R^2+S^2)(x_XdX+x_YdY)+\frac{c}4(R^2-S^2)(t_XdX+t_YdY)\\
  &=\frac12 R^2 x_XdX+\frac12 S^2 x_YdY,
\end{align*}
and, similarly, we get
\begin{equation*}
  dK=\frac{R^2}{2c}x_XdX-\frac{S^2}{2c}x_YdY
\end{equation*}
so that 
\begin{equation}
  \label{eq:derJKrel}
  J_X=c(U)K_X\ \text{ and }\ J_Y=-c(U)K_Y
\end{equation}
hold. Note the similarity between the relations
\eqref{eq:derJKrel} for the pair $(J,K)$ and the
relations \eqref{eq:chardef} for the pair
$(t,x)$. We want to compute the mixed second
derivatives of our new variables, namely, $t$,
$x$, $U$, $J$ and $K$. By \eqref{eq:chardef}, we
obtain
\begin{equation*}
  dt=t_XdX+t_YdY=\frac{1}{c}x_XdX-\frac{1}{c}x_YdY.
\end{equation*}
By expressing the fact that the form $dt$ is
closed (since it is exact), we get
\begin{equation*}
  \ldY\left(\frac1cx_X\right)=-\ldX\left(\frac1cx_Y\right)
\end{equation*}
which implies
\begin{equation*}
  \dXY{x}=\frac{c'}{2c}\left(u_Yx_X+u_Xx_Y\right).
\end{equation*}
Similarly, since the form
\begin{equation*}
  dx=x_XdX+x_YdY=ct_XdX-ct_YdY
\end{equation*}
is closed, we obtain
\begin{equation*}
  \ldY\left(ct_X\right)=-\ldX\left(ct_Y\right)
\end{equation*}
which implies
\begin{equation*}
  \dXY{t}=-\frac{c'}{2c}\left(u_Xt_Y+u_Yt_X\right).
\end{equation*}
By using the relations \eqref{eq:derJKrel}, the
form $dK$ can be rewritten as
\begin{equation*}
  dK=\frac{1}{c}J_XdX-\frac1cJ_YdY
\end{equation*}
and, expressing the fact that $dK$ is a closed, we
obtain
\begin{equation*}
  \ldY\left(\frac{1}{c}J_X\right)=\ldX\left(-\frac1cJ_Y\right)
\end{equation*}
which yields
\begin{equation*}
  \dXY{J}=\frac{c'}{2c}\left(J_XU_Y+J_YU_X\right).
\end{equation*}
Similarly, we can rewrite the form $dJ$ as
\begin{equation*}
  dK=cK_XdX-cK_YdY
\end{equation*}
and, expressing the fact $dK$ is closed, we get
\begin{equation*}
  \dXY{K}=-\frac{c'}{2c}\left(K_XU_Y+K_YU_X\right).
\end{equation*}
Let us consider the forms
\begin{equation}
  \label{eq:omega1}
  \omega_1=\frac{R}{2c}dx+\frac12 Rdt
\end{equation}
and
\begin{equation}
  \label{eq:omega2}
  \omega_2=\frac{S}{2c}dx-\frac12 Sdt.
\end{equation}
In the new variables, these forms rewrite
\begin{align}
  \notag
  \omega_1&=\frac{u_t+cu_x}{2c}(x_XdX+x_YdY)+\frac12(u_t+cu_x)(t_XdX+t_YdY)\\
  \notag
  &=(u_tt_X+u_xx_X)dX\quad\text{ (after using \eqref{eq:chardef})}\\
  \label{eq:rewrom}
  &=U_XdX,
\end{align}
and, similarly, we find
\begin{equation}
  \label{eq:rewrom2}
  \omega_2=-U_YdY.
\end{equation}
From \eqref{eq:omega1}, by using
\eqref{eq:evolRS}, we obtain
\begin{align*}
  d\omega_1&=(\frac{R_t}{2c}-\frac{c'R}{2c^2}u_t)dt\wedge dx+\frac12 R_x dx\wedge dt\\
  &=\frac{R_t-cR_x}{2c}dt\wedge dx-\frac{c'R}{2c^2}\frac12(R+S)dt\wedge dx\\
  &=\Big(\frac{c'(R^2-S^2)}{8c^2}-\frac{c'R}{2c^2}\frac12(R+S)\Big)\,dt\wedge dx\\
  &=\frac{c'}{2c^2}\left(\frac12(R+S)\right)^2dx\wedge dt=\frac{c'}{2c^2}u_t^2dx\wedge dt,
\end{align*}
and, furthermore, we obtain 
\begin{align}
  \notag
  d\omega_1&=\frac{c'}{2c^2}\Big(\frac14(R^2+S^2)dx\wedge dt+\frac12 RS dx\wedge dt\Big)\\
  \label{eq:domega}
  &=\frac{c'}{2c^2}dJ\wedge dt-\frac{c'}{2c}\omega_1\wedge\omega_2
\end{align}
because
\begin{align*}
  \omega_1\wedge\omega_2=-\frac{RS}{2c}dx\wedge dt.
\end{align*}
We rewrite \eqref{eq:domega} in the new set of
variables
\begin{align*}
  d\omega_1&=\frac{c'}{2c^2}(J_XdX+J_YdY)\wedge (t_XdX+t_YdY)-\frac{c'}{2c}\omega_1\wedge\omega_2\\
  &=-\frac{c'}{2c^3}(J_Xx_Y+J_Yx_X)dX\wedge dY+\frac{c'}{2c}U_XU_YdX\wedge dY.
\end{align*}
At the same time, by \eqref{eq:rewrom}, we have
$d\omega_1=-\dXY{U}dX\wedge dY$, and therefore
it follows that
\begin{equation*}
  \dXY{U}=\frac{c'}{2c^3}\left(J_Xx_X+J_Yx_X\right)-\frac{c'}{2c}U_XU_Y.
\end{equation*}
Finally, we obtain following system of equations
\begin{subequations}
  \label{eq:goveq}
  \begin{align}
    \label{eq:goveqt}
    \dXY{t}&=-\frac{c'}{2c}\left(U_Xt_Y+U_Yt_X\right),\\
    \label{eq:goveqx}
    \dXY{x}&=\frac{c'}{2c}\left(U_Yx_X+U_Xx_Y\right),\\
    \label{eq:govequ}
    \dXY{U}&=\frac{c'}{2c^3}\left(x_YJ_X+J_Yx_X\right)-\frac{c'}{2c}U_YU_X,\\
    \label{eq:goveqJ}
    \dXY{J}&=\frac{c'}{2c}\left(J_XU_Y+J_YU_X\right),\\
    \label{eq:goveqK}
    \dXY{K}&=-\frac{c'}{2c}\left(K_XU_Y+K_YU_X\right).
  \end{align}
\end{subequations}
Let $Z$ denote the vector $(t,x,U,J,K)$. The
system \eqref{eq:goveq} then rewrites  as
\begin{equation}
  \label{eq:condgoveq}
  Z_{XY}=F(Z)(Z_X,Z_Y)
\end{equation}
where $F(Z)$ is a bi-linear and symmetric tensor
from $\Real^5\times\Real^5$ to $\Real^5$. Due to
the relations \eqref{eq:chardef}, either one of  the equations
\eqref{eq:goveqt} and \eqref{eq:goveqx} is
redundant: one could remove one of them, and the
system would remain well-posed, and one retrieves $t$ or
$x$ by using \eqref{eq:chardef}. Similarly, either one of the equations 
\eqref{eq:goveqJ} and \eqref{eq:goveqK} becomes
redundant by \eqref{eq:derJKrel}. However, we find
it convenient to work with the complete set of
variables, that is, $Z=(t,x,U,J,K)$. We will see
later that the solutions of the system \eqref{eq:goveq} preserve
these conditions.

To prove the existence of solutions to
\eqref{eq:goveq}, we use a fixed point
argument. The argument is similar to the one that
can be found for example in \cite{garabedian} and
in \cite{BreZhe:06}. However, in order to take
into account the non-regularity of the data
($u_{0x}$ and $u_{0t}$ are in $L^2$ and the energy
can concentrate on sets of zero measure), we have to
consider, in the new set of coordinates, data
given on curves which have parts which are
parallel to the characteristic directions. In
particular, the curves are not given as  graphs of
a function.  We are looking for a solution that
satisfies a given initial condition at time
$t=0$. In the $(X,Y)$ plane, the set of points
which correspond to initial time, that is,
$t(X,Y)=0$, may be a curve,
$(\X(s),\Y(s))\in\Real^2$, parametrized by
$s\in\Real$, but it may also be a more complicated
set, see Figure \ref{fig:tconst} that we will
comment on later. We consider curves of the following
type.
\begin{definition}
  \label{def:setC}
  We denote by $\C$ the set of curves in the plane 
  $\Real^2$ parametrized by $(\X(s),\Y(s))$
  with $s\in\Real$, such that 
  \begin{subequations}
    \begin{align}
      \label{eq:regXY}
      &\X-\id,\ \Y-\id \in W^{1,\infty}(\Real),\\
      &\dot\X\geq0,\quad\dot\Y\geq0
    \end{align}
    and the normalization
    \begin{equation}
      \label{eq:regXY3}
      \frac12(\X(s)+\Y(s))=s,\text{ for all
      }s\in\Real.
    \end{equation}
  \end{subequations}
  We set
  \begin{equation}
    \label{eq:normC}
    \norm{(\X,\Y)}_{\C}=\norm{\X-\id}_{L^\infty}+\norm{\X-\id}_{L^\infty}.
  \end{equation}
\end{definition}
From the initial data $(u_0,R_0,S_0)$, we want to
define the curve $\Gamma_0=(\X(s),\Y(s))$ in $\C$
which corresponds to the initial time and the
value of $Z$ on this curve. To solve the governing
equations \eqref{eq:goveq}, we need to know the
values of $Z$, $Z_X$ and $Z_Y$ on the curve
$\Gamma_0$. In total, we have to determine 17
unknown functions. Given the initial data
$(u_0,R_0,S_0)$, there is no unique way to define
the curve $(\X(s),\Y(s))$ and the values of $Z$ on
this curve in order to obtain to the desired
solution. This fact is due to the relabeling
symmetry, a degree of freedom which is embedded in
the set of equations \eqref{eq:goveq} that we
precisely identify in Section
\ref{sec:relasym}. For now, the goal is to use
this degree of freedom to construct an initial
data which is bounded in $L^\infty(\Real)$ on the
curve. Let us now explain how we proceed for an
initial data $(u_0,R_0,S_0)\in[L^2(\Real)]^3$ for
which energy has not concentrated and we will see
later how to extend this construction to initial
data containing singular measures. In this case,
the function $\X$ and $\Y$ are invertible and,
slightly abusing notation, we denote by $Z(s)$,
$Z_X(X)$ and $Z_Y(Y)$ the values of
$Z(\X(s),\Y(s))$, $Z_X(X,\Y(\X^{-1}(X)))$ and
$Z_Y(\X(\Y^{-1}(Y)),Y)$, respectively. By
definition, we have
\begin{equation}
  \label{eq:deftini}
  t(s)=0,
\end{equation}
and, naturally, we set
\begin{equation}
  \label{eq:defUini}
  U(s)=u_0(x(s)).
\end{equation}
From the formal derivation of the previous
section, we have the following relations
\begin{align}
  \label{eq:defJXYini}
  J_X(\X)&=c(u)K_X(\X)=\frac12
  R_0^2(x)x_X(\X),&J_Y(\Y)&=-c(u)K_Y(\Y)=\frac12 S_0^2(x) x_Y(\Y),\\
  \label{eq:defUXYini}
  U_X(\X)&=\frac{R_0(x)}{c(u(x))}x_X(\X),&U_Y(\Y)&=-\frac{S_0(x)}{c(u(x))}x_Y(\Y).
\end{align}
We have the compatibility condition
\begin{equation}
  \label{eq:compaform}
  \dot Z(s)=Z_X(\X(s))\dot\X(s)+Z_Y(\Y(s))\dot\Y(s).
\end{equation}
We have 17 unknowns ($\X$,$\Y$,$Z$,$Z_X$,$Z_Y$)
and 15 equations, namely
\eqref{eq:deftini}--\eqref{eq:compaform},
\eqref{eq:regXY3} and \eqref{eq:chardef}. We use
the two degrees of freedom that remain in order to
obtain $Z_X$ and $Z_Y$ bounded. We set
\begin{equation}
  \label{eq:normJder}
  2x_X(X)+J_X(X)=1\quad\text{ and }\quad 2x_Y(Y)+J_Y(Y)=1.
\end{equation}
Since $x_X$ and $J_X$ are positive, it follows
from \eqref{eq:normJder} that these two quantities
are bounded. From the fact that
$2x_XJ_X=(c(U)U_X)^2$, it also follows that $U_X$
is bounded so that $Z_X$ is bounded. The same
conclusion holds for $Z_Y$. The normalisation
\eqref{eq:normJder} is convenient but arbitrary,
see Section \ref{sec:relasym}.  In particular, the
coefficient $2$ in front of $x_X$ and $x_Y$ in
\eqref{eq:normJder} does not have any importance;
it is used here to make the definition 
compatible with the normalization we will
introduce in Section \ref{sec:initdata} for the
general case. From \eqref{eq:compaform},
\eqref{eq:defUXYini} and \eqref{eq:defJXYini}, we
get
\begin{equation}
  \label{eq:xXYfrac}
  x_X(\X)=\frac2{4+R_0^2}(x)\quad \text{ and }\quad x_Y(\Y)=\frac2{4+S_0^2}(x),
\end{equation}
\begin{align}
  \label{eq:defJXYini2}
  J_X(\X)&=\frac{1}{c}K_X(\X)=\frac{R_0^2}{4+R_0^2}(x),&J_Y(\Y)&=-\frac{1}{c}K_Y(\Y)=\frac{S_0^2}{4+S_0^2}(x),\\
  \label{eq:defUXYini2}
  U_X(\X)&=\frac{2R_0}{c(4+R_0^2)}(x),&U_Y(\Y)&=-\frac{2S_0}{c(4+S_0^2)}(x).
\end{align}
Equation \eqref{eq:deftini} implies
\begin{equation*}
  0=t_X(\X)\dot\X+t_Y(\Y)\dot\Y=x_X(\X)\dot\X-x_Y(\Y)\dot\Y
\end{equation*}
and, at the same time, we have by the chain rule
\begin{equation*}
  \dot x(s)=x_X(\X)\dot\X+x_Y(\Y)\dot\Y
\end{equation*}
and therefore
\begin{equation}
  \label{eq:relxXYini}
  x_X(\X)\dot\X=x_Y(\Y)\dot\Y=\frac{\dot x}2.
\end{equation}
Hence, by \eqref{eq:regXY3}, \eqref{eq:xXYfrac}
and \eqref{eq:relxXYini}, we get
\begin{equation*}
  2=\dot\X+\dot\Y=\left(2+\frac14(R_0^2+S_0^2)(x)\right)\dot{x}
\end{equation*}
and we define $x(s)$ implicitly as
\begin{equation}
  \label{eq:defphi}
  2x(s)+\int_{-\infty}^{x(s)}\frac14(R_0^2+S_0^2)\,dx=2s.
\end{equation}
We have
\begin{equation*}
  2\dot x+\dot J=2x_X(\X)\dot\X+J_X(\X)\dot\X+2x_Y(\Y)\dot\Y+J_Y(\Y)\dot\Y=2
\end{equation*}
because of \eqref{eq:compaform} and
\eqref{eq:regXY3} so that $2x+J=2s$. Hence,
\begin{equation}
  \label{eq:defJini}
  J(s)=\frac18\int_{-\infty}^{x(s)}(R_0^2+S_0^2)\,dx
\end{equation}
and 
\begin{equation}
  \label{eq:defKini}
  K(s)=\int_{-\infty}^{x(s)}\frac{R_0^2-S_0^2}{8c}\,dx,
\end{equation}
which are also defined as the integrals of the
forms $dJ$ and $dK$ given by \eqref{eq:dJdef} and
\eqref{eq:dKdef} on the line
$(t,x)=\{0\}\times(-\infty,x(s))$. From
\eqref{eq:relxXYini} and \eqref{eq:xXYfrac}, it
follows that
\begin{equation}
  \label{eq:defderXYinit} 
  \dot \X(s)=\dot x(s)(1+\frac14 R_0^2)(x(s))\quad \text{ and }\quad
  \dot \Y(s)=\dot x(s)(1+\frac14 S_0^2)(x(s)),
\end{equation}
and we set
\begin{equation}
  \label{eq:defXYinit}
  \X(s)=x(s)+\frac14\int_{-\infty}^{x(s)} R_0^2\,dx\quad \text{ and
  }\quad \Y(s)=x(s)+\frac14\int_{-\infty}^{x(s)} S_0^2\,dx.
\end{equation}

\section{The initial data}
\label{sec:initdata}

In order to construct a semigroup of conservative
solutions, we have to take into account the part
of the energy which has concentrated in sets of
measure zero and we need to consider initial data
in the set $\D$ that we now define.
\begin{definition}
  The set $\D$ consists of the elements
  $(u,R,S,\mu,\nu)$ such that
  \begin{equation*}
    (u,R,S)\in [L²(\Real)]^3,
  \end{equation*}
  $u_x=\frac1{2c}(R-S)$ and $\mu$ and $\nu$ are
  finite positive Radon measures with
  \begin{equation}
    \label{eq:mudef}
    \muac=\frac14 R^2\,dx,\quad\nuac=\frac14 S^2\,dx.
  \end{equation}
\end{definition}
The measures $\mu$ and $\nu$ correspond to the
left and right traveling energy densities,
respectively.  Given the initial data
$(u_0,R_0,S_0,\mu_0,\nu_0)$, we have defined an
element in $\G_0$ where the set $\G_0$ is defined
below and which correspond to a parametrization of
the initial data in the new system of
coordinates. Elements of $\G$ consists of a curve
$(\X(s),\Y(s))$ (for $\G_0$, this curve
corresponds to time equal to zero) and three
variables, $\Z$, $\V$ and $\W$, that we now
introduce. These functions correspond to the data
that matches the solution $Z$ to \eqref{eq:goveq}
on the curve $(\X,\Y)$ in the sense that
\begin{subequations}
  \label{eq:ZcondininG}
  \begin{equation}
    \label{eq:ZcondininG1}
    \Z(s)=Z(\X(s),\Y(s))
  \end{equation}
  and
  \begin{equation}
    \label{eq:ZcondininG2}
    \V(\X(s))=Z_X(\X(s),\Y(s))\text{ and }\W(\X(s))=Z_Y(\X(s),\Y(s))
  \end{equation}
\end{subequations}
It is then convenient to introduce the following
notation: To any triplet $(\Z,\V,\W)$ of five
dimensional vector functions (we write $\Z=(\Z_1,\Z_2,\Z_3,\Z_4,\Z_5)$, etc), we associate the
triplet $(\Z^a,\V^a,\W^a)$ given by
\begin{subequations}
  \label{eq:defZa}
  \begin{equation}
    \Z^a_2=\Z_2-\id,\quad\V^a_2=\V_2-\frac12,\quad \W^a_2=\W_2-\frac12 
  \end{equation}
  and
  \begin{equation}
    \Z^a_i=\Z_i,\quad \V^a_i=\V_i,\quad \W^a_i=\W_i
  \end{equation}
\end{subequations}
for $i\in\{1,3,4,5\}$.
\begin{definition}
  \label{def:setG}
  The set $\G$ is the set of all elements which
  consist of a curve $(\X(s),\Y(s))$ and three
  vector valued functions from $\Real$ to
  $\Real^5$ denoted $\Z(s),\V(X),\W(Y)$. We denote
  $\Theta=(\X,\Y,\Z,\V,\W)$ and set
  \begin{equation}
    \label{eq:defnormG}
    \norm{\Theta}_{\G}=\norm{U}_{L^2(\Real)}
    +\norm{\V^a}_{L^2}+\norm{\W^a}_{L^2}
  \end{equation}
  where we denote $U=\Z_3$ and
  \begin{multline}
    \label{eq:deftnormG}
    \tnorm{\Theta}_{\G}=\norm{(\X,\Y)}_{\C}+\norms{\frac1{\V_2+\V_4}}_{L^\infty(\Real)}+\norms{\frac1{\W_2+\W_4}}_{L^\infty(\Real)}\\
    +\norm{\Z^a}_{L^\infty}+\norm{\V^a}_{L^\infty}+\norm{\W^a}_{L^\infty}.
  \end{multline}
  The element $\Theta\in\G$ if
  \begin{enumerate}
  \item[(i)] 
    \begin{equation*}
      \norm{\Theta}_{\G}<\infty\quad\text{and}\quad\tnorm{\Theta}_{\G}<\infty;
    \end{equation*}
  \item[(ii)]
    \begin{equation}
      \label{eq:posVWb}
      \V_2,\W_2,\V_4,\W_4\geq0;
    \end{equation}
  \item[(iii)] for almost every $s$, we have
    \begin{equation}
      \label{eq:relZVWG}
      \dot\Z(s)=\V(\X(s))\dot\X(s)+\W(\Y(s))\dot\Y(s);
    \end{equation}
  \item[(iv)] for almost every $X$ and $Y$, we have
    \begin{subequations}
    \begin{align}
      \label{eq:relVW}
      2\V_4(\X)\V_2(\X)&=(c(U)\V_3(\X))^2,& 2\W_4(\Y)\W_2(\Y)&=(c(U)\W_3(\Y))^2,\\
      \label{eq:reltx1}
      \V_{2}(\X)&=c(U)\V_{1}(\X),& \W_{2}(\Y)&=-c(U)\W_{1}(\Y),\\
      \label{eq:reltJK1}
      \V_{4}(\X)&=c(U)\V_{5}(\X),&\W_{4}(\Y)&=-c(U)\W_{5}(\Y).
    \end{align}
    \end{subequations}
  \item[(v)] We require
    \begin{equation}
      \label{eq:normalizationJ}
      \lim_{s\to-\infty}J(s)=0
    \end{equation}
    where we denote $J(s)=\Z_4(s)$. 
  \end{enumerate}
  We denote by $\G_0$ the subset of $\G$ which
  parametrizes data at time $t=0$, that is,
  \begin{equation*}
    \G_0=\{\Theta\in\G\mid \Z_1=0\}.
  \end{equation*}
\end{definition}
\begin{figure}
  \label{fig:tconst}
  \includegraphics[scale=1]{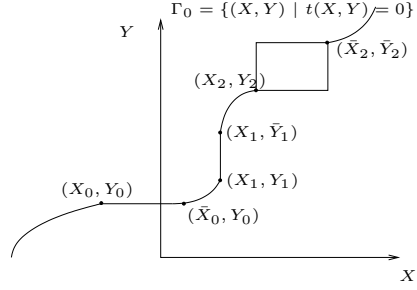}
  \caption{The domain $t(X,Y)=0$ in the $X,Y$
    plane consists of the union of a graph of a
    strictly increasing function, vertical and
    horizontal segments and rectangular boxes.}
\end{figure}
The requirement \eqref{eq:normalizationJ}
corresponds to a normalization of the energy
potential (or cumulative energy) to zero at minus
infinity. The variables $\Z$, $\V$ and $\W$ are
not independent of one another as it can be
seen from \eqref{eq:relZVWG}, \eqref{eq:reltx1},
\eqref{eq:reltJK1} but selecting a set of
independent variables will require an arbitrary
choice that we prefer to avoid and that is why we
consider all the variables at the same level. For
$\Theta\in\G_0$, we get by using \eqref{eq:relZVWG}
and \eqref{eq:reltx1}, that
\begin{equation}
  \label{eq:V2W2equala}
  \V_2(\X(s))\dot \X(s)=\W_2(\Y(s))\dot \Y(s).
\end{equation}
By using the normalization
\eqref{eq:regXY3}, we obtain that
\begin{align}
  \label{eq:edoforXY}
  \dot\X=\frac{2\W_2(\Y)}{\V_2(\X)+\W_2(\Y)},&&\dot\Y=\frac{2\V_2(\X)}{\V_2(\X)+\W_2(\Y)}
\end{align}
and, in principle, by integrating
\eqref{eq:edoforXY}, we recover $\X$ and
$\Y$. However, there are two obstacles to that:
The function $\V_2$ and $\W_2$ are in general not
Lipschitz so that we cannot use the standard
existence theorems for the solutions to
\eqref{eq:edoforXY} and, in addition, both $\V_2$
and $\W_2$ may vanish (it is what happens in the
case of a box) and \eqref{eq:edoforXY} does not
make sense any more.  Given
$(u_0,R_0,S_0,\mu_0,\nu_0)$, in the case where
$\mu_0=(\mu_0)_{\text{ac}}$ and
$\nu_0=(\nu_0)_{\text{ac}}$, we have defined
$\Theta=(\X,\Y,\Z,\V,\W)\in\G_0$ by
\eqref{eq:deftini}, \eqref{eq:defUini},
\eqref{eq:defphi}, \eqref{eq:defJini},
\eqref{eq:defKini}, \eqref{eq:xXYfrac},
\eqref{eq:defXYinit} and
\begin{align*}
  \V_4(\X(s))&=c(U)\V_5(\X(s))=\frac{R_0^2}{4+R_0^2}(x(s)),\\\W_4(\X(s))&=-c(U)\W_5(\X(s))=\frac{S_0^2}{4+S_0^2}(x(s)),\\
\V_3(\X(s))&=\frac{2R_0}{c(4+R_0^2)}(x(s)),&\W_3(\Y(s))&=-\frac{2S_0}{c(4+S_0^2)}(x(s)).
\end{align*}
We do not prove here that, for this definition, we
indeed have $(\X,\Y,\Z,\V,\W)\in\G_0$ because it
will be done later in more generality, see
Definitions \ref{def:mappingL} and \ref{def:C}. In
the next section we consider
$(\X,\Y,\Z,\V,\W)\in\G_0$ and construct solutions
of \eqref{eq:goveq} which satisfy
\eqref{eq:ZcondininG}. However, the set $\G_0$ is
not adequate when it comes to parametrize initial
data. In the case where there is no concentration
of the measures, that is,
$\mu_0=(\mu_0)_{\text{ac}}$ and
$\nu_0=(\nu_0)_{\text{ac}}$, we can see from
\eqref{eq:defXYinit} and \eqref{eq:defphi} that
$\dot\X>0$ and $\dot\Y>0$ almost everywhere so
that the curve does not contain strictly vertical
or horizontal regions. This property is
\textit{not} preserved by the equation. In
particular it means that at a later time, say
$T>0$, we can find a curve $(\bar\X,\bar\Y)\in\C$
such that $t(\bar\X(s),\bar \Y(s))=T$ and
$\dot\X(s)=0$ or $\dot\Y(s)=0$ on an interval
$[s_l,s_r]$, with $s_l<s_r$.  In general, the set
of points
\begin{equation*}
  \Gamma_T=\{(X,Y)\in\Real^2\mid t(X,Y)=T\}
\end{equation*}
is not a curve but a domain which consists of the
union of a graph of a strictly increasing
function, vertical and horizontal segments and
rectangular boxes, see Figure \ref{fig:tconst}. If
$\Gamma_T$ contains regions with boxes or vertical
or horizontal lines, it means that part of the
energy of the solution is concentrated at time $T$
in sets of zero measure, see Section
\ref{sec:backto}. We want to parametrize domains
$\Gamma_0$ (or $\Gamma_T$) depicted in Figure
\ref{fig:tconst}, which give the solution at time
zero (or a given time $T$) and which may contain
boxes. The set $\G_0$ defined above is
inappropriate. When considering an element in
$\G_0$, we choose a curve and in the case of a
box, the choice of the curve which joins the two
diagonal corners of the box while remaining inside
the box is arbitrary. Thus we introduce an
unwanted degree of freedom in the parametrization
of the initial data. The domain $\Gamma_0$
depicted in Figure \ref{fig:tconst} can be
parametrized by using two nondecreasing functions
$x_1(X)$ and $x_1(Y)$ and by considering the set
$\{(X,Y)\in\Real\mid x_1(X)=x_2(Y)\}$. Such sets
consist exactly of the union of the graph of a
strictly increasing function (when $x_1'>0$ and
$x_2'>0$), a horizontal segment (when $x_1'(X)=0$
for $X\in[X_0,\bar X_0]$ and $x_2'(Y_0)>0$), a
vertical segment (when $x_1'(X_1)>0$ and
$x_2'(Y)=0$ for $Y\in[Y_1,\bar Y_1]$) and a
rectangular box (when $x_1'(X)=x_2'(Y)=0$ for
$X\in[X_2,\bar X_2]$ and $Y\in[Y_2,\bar Y_2]$),
see Figure \ref{fig:tconst}. This observation
(partially) justifies the definition of the set
$\F$ which is introduced below. The set $\F$ can
be considered as a consistent way to parametrize
initial data. However, to construct the solutions,
we need to choose a curve and we use the
description of the initial data given by $\G_0$ so
that, finally, both sets are needed. To define
$\F$, we have to introduce the group $\Gr$ of
diffeomorphisms with some regularity conditions.
\begin{definition}
  \label{def:Gr}
  The group $G$ is given by all invertible functions $f$ such
  that
  \begin{equation}
    \label{eq:Hcond}
    f-\id\text{ and }f^{-1}-\id\text{ both belong to }\Winf(\Real),
  \end{equation}
  and
  \begin{equation*}
    (f-\id)'\in L^2(\Real).
  \end{equation*}
\end{definition}
We can now define the set $\F$.
\begin{definition}
  \label{def:F}
  We define the set $\F$ consisting of all
  function $\psi=(\psi_1,\psi_2)$ such that
  \begin{align*}
    \psi_1(X)&=(x_1(X),U_1(X),V_1(X),J_1(X),K_1(X))\\
    \text{ and }\quad
    \psi_2(Y)&=(x_2(Y),U_2(Y),V_2(Y),J_2(Y),K_2(Y))
  \end{align*}
  satisfy the following regularity and decay
  conditions
  \begin{subequations}
    \begin{equation}
      \label{eq:regx1x2}
      x_1-\id,\ x_2-\id,\ J_1,\ J_2,\ K_1,\ K_2\in W^{1,\infty}(\Real),
    \end{equation}
    \begin{equation}
      \label{eq:regx1x2b}
      x_1'-1,\ x_2'-1,\ J_1',\ J_2',\ K_1',\ K_2'\in L^2(\Real)\cap L^\infty(\Real),
    \end{equation}
    \begin{equation}
      \label{eq:regU1U2}
      U_1,\ U_2\in H^1(\Real),
    \end{equation}
    \begin{equation}
      \label{eq:regV1V2}
      V_1,\ V_2\in L^2(\Real)\cap L^\infty(\Real),
    \end{equation}
  \end{subequations}
  and which satisfy the additional conditions that
  \begin{equation}
    \label{eq:posit}
    x_1',x_2',J_1',J_2'\geq0,
  \end{equation}
  \begin{equation}
    \label{eq:J1K1rel}
    J_1'=c(U_1)K_1',\quad J_2'=-c(U_2)K_2',
  \end{equation}
  \begin{equation}
    \label{eq:relxJV1}
    x_1'J_1'=(c(U_1)V_1)^2,\quad x_2'J_2'=(c(U_2)V_2)^2,
  \end{equation}
  \begin{equation}
    \label{eq:Gprop}
    x_1+J_1,\quad x_2+J_2\in\Gr,
  \end{equation}
  \begin{equation}
    \label{eq:limJzero}
    \lim_{X\to-\infty}J_1(X)=\lim_{Y\to-\infty}J_2(Y)=0
  \end{equation}
  and, for any curve $(\X,\Y)\in\C$ such that
  \begin{equation*}
    x_1(\X(s))=x_2(\Y(s))\text{ for all }s\in\Real,
  \end{equation*}
  we have
  \begin{subequations}
    \label{eq:condcurvF}
    \begin{equation}
      \label{eq:U1U2equal}
      U_1(\X(s))=U_2(\Y(s))
    \end{equation}
    for all $s\in\Real$ and
    \begin{equation}
      \label{eq:condU1U_2}
      U_1'(\X(s))\dot\X(s)=U_2'(\Y(s))\dot\Y(s)=V_1(\X(s))\dot\X(s)+V_2(Y(s))\dot\Y(s)
    \end{equation}
  \end{subequations}
  for almost all $s\in\Real$.
\end{definition}
We show in Section \ref{sec:semigroupinG} that,
given a solution $Z$ of \eqref{eq:goveq}, there
exists a unique element $\psi\in\F$ which
describes in a unique way the set
$\Gamma_0=\{(X,Y)\in\Real^2\mid t(X,Y)=0\}$ and the
values of $Z$, $Z_X$ and $Z_Y$ on this set. The
functions $x_1$ and $x_2$ define the set
$\Gamma_0$ by $\Gamma_0=\{(X,Y)\in\Real^2\mid
x_1(X)=x_2(Y)\}$. It means in particular that, for
any curve $(\X,\Y)\in\C$ such that
$x_1(\X(s))=x_2(\Y(s))$, we have
$t(\X(s),\Y(s))=0$. The functions $U_1$ and $U_2$
give the value of $U(X,Y)$ on the set $\Gamma_0$,
as a function of $X$ and $Y$. To be more concrete,
let us consider the example where $x_1$ and $x_2$
are smooth, invertible and the inverses are also
smooth. In that case, which in fact corresponds to
the case where $u_0$, $R_0$ and $S_0$ are smooth
and there is no concentration of energy, i.e.,
$\mu_0=(\mu_0)_{\text{ac}}$ and
$\nu_0=(\nu_0)_{\text{ac}}$, the set $\Gamma_0$ is
the graph of a strictly increasing function (there
is no rectangular box and no vertical or
horizontal segments). The curve $\Gamma_0$ is
given by either $Y=x_1^{-1}\circ x_2(X)$ or
$X=x_2^{-1}\circ x_1(Y)$ and, just for this
paragraph, for the sake of simplicity, we denote
$Y(X)=x_1^{-1}\circ x_2(X)$ and
$X(Y)=x_2^{-1}\circ x_1(Y)$. Then, we have
\begin{equation}
  \label{eq:explU1}
  U_1(X)=U(X,Y(X))\text{ and }U_2(Y)=U(X(Y),Y).
\end{equation}
The functions of $V_1$ and $V_2$ give the partial
derivative of $U$. We have
\begin{equation}
  \label{eq:explV1V2}
  V_1(X)=U_X(X,Y(X))\text{ and }V_2(Y)=U_Y(X(Y),Y).
\end{equation}
As we can see in this example, the functions
$U_1$, $U_2$, $V_1$ and $V_2$ are not independent
 from one  another, and the way they depend one
another is given by \eqref{eq:U1U2equal}
and \eqref{eq:condU1U_2}. The function $J_1(X)$
gives the amount of forward energy contained on
the curve $Y=Y(X)$ between $-\infty$ and $X$, that
is,
\begin{equation*}
  J_1(X)=\int_{-\infty}^{X}J_X(X,Y(X))\,dX.
\end{equation*}
In the original set of coordinates, it gives
$J_1(X)=\frac14\int_{-\infty}^{x_1(X)} R_0^2\,dx$. Similarly,
the function $J_2(Y)$ gives the amount of backward
energy which is contained on the same curve
between $-\infty$ and $Y$, that is,
\begin{equation*}
  J_2(Y)=\int_{-\infty}^{Y}J_Y(X(Y),Y)\,dY.
\end{equation*}
In the original set of coordinates, it gives
$J_2(Y)=\frac14\int_{-\infty}^{x_2(Y)}S_0^2\,dx$. We
recall that these expressions hold only for smooth
initial data with no concentration of
energy. Still in this case, the functions $x_1$
and $x_2$ are strictly increasing so that $x_1'>0$
and $x_2'>0$ and the conditions \eqref{eq:relxJV1}
entirely determine the energy densities, which are
given $J_1'$ and $J_2'$ in the new sets of
coordinates. We have
\begin{equation*}
  \mu_0=(\mu_0)_{\text{ac}}=\frac14 R_0^2(x)\,dx=\frac{J_1'}{x_1'}\circ x_1^{-1}(x)\,dx
\end{equation*}
and the corresponding expression for $\nu_0$.  In
the case where there is concentration of energy,
the functions $x_1'$ or $x_2'$ vanish. The set
where $x_1$ (respectively $x_2$) vanishes
corresponds to the region where the energy density
$\mu_0$ (respectively $\nu_0$) has a singular
part. On those sets, the energy densities $J_1'$
and $J_2'$ cannot be retrieved from
\eqref{eq:relxJV1}. It is consistent with the fact
that the singular parts of the energy $\mu$ and
$\nu$ cannot be recovered by the knowledge of the
function $u$, $R$ and $S$, as illustrated in the
example presented in the introduction. As we will
see in Section \ref{sec:backto}, the relations
\eqref{eq:relxJV1} correspond to a reformulation
in the new coordinate system of \eqref{eq:mudef}.

We define a mapping $\Cb$ which to any given
initial data $\psi\in\F$ associate the
corresponding data $\Theta=(\X,\Y,\Z,\V,\W)\in\G_0$.
\begin{definition}
  \label{def:C}
  For any $\psi=(\psi_1,\psi_2)\in\F$, we define
  \begin{equation}
    \label{eq:defXs}
    \X(s)=\sup\{X\in\Real\mid x_1(X')<x_2(2s-X')\text{  for all }X'<X\}
  \end{equation}
  and set $\Y(s)=2s-\X(s)$. We have
  \begin{equation}
    \label{eq:x1eqx2}
    x_1(\X(s))=x_2(\Y(s)).
  \end{equation}
  We define
  \begin{subequations}
    \label{eq:barZmapC}
    \begin{align}
      &t(s)=0,\\
      \label{eq:barZmapC2}
      &x(s)=x_1(\X(s))=x_2(\Y(s)),\\
      &U(s)=U_1(\X(s))=U_2(\Y(s)),\\
      &J(s)=J_1(\X(s))+J_2(\Y(s)),\\
      &K(s)=K_1(\X(s))+K_2(\Y(s))
    \end{align}
  \end{subequations}
and
  \begin{align*}
    \V_1(X)&=\frac1{2c(U_1(X))} x_1'(X),&\W_1(Y)&=-\frac1{2c(U_2(Y))} x_2'(Y),\\
    \V_2(X)&=\frac12x_1'(X),&\W_2(Y)&=\frac12x_2'(Y),\\
    \V_3(X)&=V_1(X),&\W_3(Y)&=V_2(Y),\\
    \V_4(X)&=J_1'(X),&\W_4(Y)&=J_2'(Y),\\
    \V_5(X)&=K_1'(X),&\W_5(Y)&=K_2'(Y).
  \end{align*}
  Let $\Cb$ be the mapping from $\F$ to $\G_0$
  which to any $\psi\in\F$ associates the element
  $(\X,\Y,\Z,\V,\W)$ defined above.
\end{definition}

\begin{proof}[Proof of the well-posedness of Definition
  \ref{def:C}]
  Let us prove that $\X$ is increasing. Given
  $\bar s>s$, we consider a sequence $X_i$ which
  converges to $\X(s)$ with $X_i<\X(s)$. We have
  $x_1(X_i)<x_2(2s-X_i)$ which implies
  $x_1(X_i)<x_2(2\bar s-X_i)$ because $x_2$ is
  increasing. Hence $X_i<\X(\bar s)$. By letting
  $i$ tend to infinity, we get that $\X(s)\leq
  \X(\bar s)$. By continuity of $x_1$ and $x_2$, we
  have $x_1(\X(s))=x_2(\Y(s))$. We claim that $\X$
  is Lipschitz with a Lipshitz constant no bigger
  than 2, i.e.,
  \begin{equation}
    \label{eq:lipX0}
    \abs{\X(\bar s)-\X(\bar s)}\leq2\abs{\bar s-s}.
  \end{equation}
  Let us assume without loss of generality that
  $\bar s>s$. If \eqref{eq:lipX0} does not hold,
  we have
  \begin{equation}
    \label{eq:contlip}
    \X(\bar s)-\X(s)>2(\bar s-s)
  \end{equation}
  for some $s$ and $\bar s$ in $\Real$. It
  implies $\Y(\bar s)<\Y(s)$. Then, by
  monotonicity of $x_2$,
  \begin{equation*}
    x_1(\X(s))=x_2(\Y(s))\geq x_2(\Y(\bar s))=x_1(\X(\bar s)),
  \end{equation*}
  and therefore $x_1(\X(s))=x_1(\X(\bar s))$
  because $x_1$ is an increasing function and
  $\X(s)<\X(\bar s)$. It follows that $x_1$ is
  constant on $[\X(s),\X(\bar s)]$. Similarly, one
  proves that $x_2$ is constant on $[\Y(\bar
  s),\Y(s)]$. Let us consider the point $(X,Y)$
  given by $Y=\Y(s)$ and $X=2\bar s-\Y(s)$. We
  have
  \begin{equation*}
    \X(s)=2s-\Y(s)< X <2\bar s-\Y(\bar s)=\X(\bar s)
  \end{equation*}
  so that $x_1(X)=x_1(\X(s))=x_2(\Y(s))=x_2(2\bar
  s-Y)$ and $X<\X(\bar s)$, which contradicts the
  definition of $\X(\bar s)$. Hence,
  \eqref{eq:contlip} cannot hold and we have
  proved \eqref{eq:lipX0}. Let us prove that
  $\X-\id\in L^\infty$. We have
  \begin{equation*}
    \X(s)-s=\frac12(\X(s)-\Y(s))=\frac12(\X(s)-x_1(\X(s))+x_2(\Y(s))-\Y(s))
  \end{equation*}
  which is bounded as $x_1-\id$ and $x_2-\id$ belong
  to $L^\infty$. Let 
  \begin{equation*}
    B=\{s\in\Real\mid \dot \X(s)\geq 1\}.
  \end{equation*}
  Since $\dot \X+\dot \Y=2$, we have $\dot
  \Y\geq1$ on $B^c$. Hence,
  \begin{align*}
    \int_\Real U^2(s)\,ds&=\int_B U^2(s)\,ds+\int_{B^c} U^2(s)\,ds\\
    &\leq\int_B U_1^2(\X(s))\dot \X(s)\,ds+\int_{B^c} U_2^2(\Y(s))\dot \Y(s)\,ds\\
    &\leq\norm{U_1}_{L^2}^2+\norm{U_2}_{L^2}^2.
  \end{align*}
  It is then straightforward to check that the
  remaining properties that enter in the
  definition of $\G_0$ are fulfilled by
  $(\Z,\V,\W)$. To check that \eqref{eq:Gprop} is
  fulfilled, we use Lemma \ref{lem:charH} which
  is stated below.
\end{proof}

\begin{lemma} 
\label{lem:charH}
Let $\alpha\geq0$. If $f$ satisfies
\eqref{eq:Hcond}, then $1/(1+\alpha)\leq f_\xi\leq
1+\alpha$ almost everywhere. Conversely, if $f$ is
absolutely continuous, $f-\id\in{L^\infty(\Real)}$
and there exists $c\geq 1$ such that $1/c\leq
f_\xi\leq c$ almost everywhere, then $f$ satisfies
\eqref{eq:Hcond} and
\begin{equation*}
  \norm{f-\id}_{\Winf(\Real)}+\norm{f\inv-\id}_{\Winf(\Real)}\leq\alpha
\end{equation*}
for some $\alpha$ depending only on $c$ and
$\norm{f-\id}_{L^\infty(\Real)}$.
\end{lemma}
The proof of this short lemma is given in
\cite{HolRay:07}.  In the opposite direction, to
any element $(\X,\Y,\Z,\V,\W)\in\G_0$, there
corresponds an element $(\psi_1,\psi_2)\in\F$
given by the mapping $\Db$ that we define next.
\begin{definition}
  \label{def:D}
  Given $(\X,\Y,Z,\V,\W)\in\G_0$, let
  $\psi_1=(x_1,U_1,J_1,K_1,V_1)$ and
  $\psi_2=(x_2,U_2,J_2,K_2,V_2)$ be defined as
  \begin{equation}
    \label{eq:x1x2x}
    x_1(\X(s))=x_2(\Y(s))=x(s)
  \end{equation}
  where we denote $x(s)=\Z_2(s)$ and
  \begin{equation}
    \label{eq:U1U2U}
    U_1(\X(s))=U_2(\Y(s))=U(s)
  \end{equation}
  where we denote $U(s)=\Z_3(s)$ and
  \begin{align}
    \label{eq:defJ1J2inD}
    J_1(\X(s))&=\int_{-\infty}^{s}\V_4(\X(s))\dot \X(s)\,ds,& 
J_2(\Y(s))&=\int_{-\infty}^{s}\W_4(\Y(s))\dot \Y(s)\,ds,\\
     \label{eq:defK1K2inD}
    K_1(\X(s))&=\int_{-\infty}^{s}\V_5(\X(s))\dot \X(s)\,ds,& 
K_2(\Y(s))&=\int_{-\infty}^{s}\W_5(\Y(s))\dot \Y(s)\,ds,
  \end{align}
  and
  \begin{equation}
    \label{eq:V1V3V2W3}
    V_1=\V_3,\quad V_2=\W_3.
  \end{equation}
  We denote by $\Db$ the mapping from $\G_0$ to
  $\F$ which to any $(\X,\Y,\Z,\V,\W)\in\G_0$
  associates the element $\psi$ as defined above.
\end{definition}
\begin{proof}[Well-posedness of Definition \ref{def:D}]
  Since $t(s)=0$ we have
  \begin{equation*}
    0=\dot t(s)=\V_1(\X(s))\dot \X(s)+\W_1(\Y(s))\dot \Y(s)
  \end{equation*}
  which implies that 
  \begin{equation}
    \label{eq:V2W2equal}
    \V_2(\X(s))\dot \X(s)=\W_2(\Y(s))\dot \Y(s)
  \end{equation}
  by \eqref{eq:reltx1}. We check the
  well-posedness of \eqref{eq:x1x2x} and
  \eqref{eq:U1U2U}.  Let us consider $s$ and $\bar
  s$ such that $\X(s)=\X(\bar s)$. Since $\X$ is
  increasing, it implies $\dot \X(\tilde s)=0$ and
  $\dot\Y(\tilde s)=2$ for all $\tilde s\in[s,\bar
  s]$. From \eqref{eq:V2W2equal}, it follows that
  $\W_2(\Y(\tilde s))=0$ for all $\tilde
  s\in[s,\bar s]$. Hence,
  \begin{equation*}
    \dot x(\tilde s)=\V_2(\X(\tilde s))\dot \X(\tilde s)+\W_2(\Y(\tilde s))\dot \Y(\tilde s)=0
  \end{equation*}
  and $x(s)=x(\bar s)$ so that the definition
  \eqref{eq:x1x2x} is well-posed. For $\tilde
  s\in[s,\bar s]$, we have $\W_3(\Y(\tilde s))=0$, by
  \eqref{eq:relVW} and the fact that
  $\W_2(\Y(\tilde s))=0$. Hence,
  \begin{equation*}
    \dot U(\tilde s)=\V_3(\X(\tilde s))\dot \X(\tilde s)+\W_3(\Y(\tilde s))\dot \Y(\tilde s)=0
  \end{equation*}
  and $U(s)=U(\bar s)$ so that the definition
  \eqref{eq:U1U2U} is well-posed. Let us prove
  that $x_1$ is Lipschitz. We have
  \begin{align*}
    x_1(\X(s))-x_1(\X(\bar s))&=x(s)-x(\bar s)\\
    &=\int_{\bar s}^{s}\dot x(s)\,ds\\
    &=\int_{\bar s}^{s}\V_2(\X(s))\dot
    \X(s)+\W_2(\Y(s))\dot \Y(s)\,ds\\
    &=2\int_{\bar s}^{s}\V_2(\X(s))\dot \X(s)\,ds&&\text{(by \eqref{eq:V2W2equal})}\\
    &\leq\norm{\V_2}_{L^\infty}\abs{\X(s)-\X(\bar s)}.
  \end{align*}
  Hence, $x_1$ is Lipschitz. One proves in the same
  way that $x_2$ is Lipschitz. Since
  \begin{equation*}
    0\leq\V_4(\X(s))\dot \X(s)\leq\V_4(\X(s))\dot \X(s)+\W_4(\Y(s))\dot \Y(s)=\dot J(s)
  \end{equation*}
  the function $\V_4(\X(s))\dot \X(s)$ belongs to
  $L^1(\Real)$. Assume that there exists an $s<\bar s$
  such that $\X(s)=\X(\bar s)$. Since $\X$ is
  increasing, it implies that $\dot\X(s)=$ for all
  $s\in[s,\bar s]$ and therefore
  $\int_{-\infty}^{s}\V_4(\X(s))\dot
  \X(s)\,ds=\int_{-\infty}^{\bar s}\V_4(\X(s))\dot
  \X(s)\,ds$ and the definition
  \eqref{eq:defJ1J2inD} of $J_1$ is
  well-posed. The same results hold for $J_2$. Let
  us prove that $U_1$ is absolutely continuous on
  any compact set. We consider $X_1<\cdots<X_N$ and
  $s_i$ such $\X(s_i)=X_i$. We have
  \begin{align*}
    \sum_{i=1}^N
    \abs{U_1(X_{i+1})-U_1(X_i)}&=\sum_{i=1}^N
    \abs{U_1(s_{i+1})-U_1(s_i)}\\
    &\leq\int_{\cup_i(s_i,s_{i+1})}\abs{\dot U_1(s)}\,ds\\
    &\leq\int_{\cup_i(s_i,s_{i+1})}(\V_3(\X)\dot \X+\W_3(\Y)\dot \Y)\,ds\\
    &\leq\norm{\V_3}_{L^\infty}\int_{\cup_i(s_i,s_{i+1})}\dot
    \X\,ds\\
    &\quad+\meas(\cup_i(s_i,s_{i+1}))^{1/2}\big(\int_{\cup_i(s_i,s_{i+1})}\W_3(\Y)^2\dot
    \Y^2\,ds\big)^{1/2}.
  \end{align*}
  By \eqref{eq:relVW}, we get $\W_3^2\leq
  2\kappa\norm{\W_4}_{L^\infty(\Real)}\W_2$, and
  therefore $\W_3^2(\Y)\dot \Y^2\leq C\W_2(\Y)\dot
  \Y^2=C \V_2(\X)\dot\X\dot\Y$, by
  \eqref{eq:V2W2equal}, for some constant
  $C$. Hence,
  \begin{align*}
    \int_{\cup_i(s_i,s_{i+1})}\W_3(\Y)^2\dot
    \Y^2\,ds&\leq
    C\int_{\cup_i(s_i,s_{i+1})}\V_2(\X)\dot
    \X\,ds\\
    &\leq C\int_{\cup_i(s_i,s_{i+1})}\dot
    \X\,ds=C\meas(\cup_{i}(X_i,X_{i+1}))
  \end{align*}
  for some constant $C$. Finally,
  \begin{equation*}
    \sum_{i=1}^N\abs{U_1(X_{i+1})-U_1(X_i)}\leq C(\meas(\cup_{i}(X_i,X_{i+1}))+\meas(\cup_{i}(X_i,X_{i+1}))^{1/2})
  \end{equation*}
  and $U_1$ is absolutely continuous. After
  differentiating \eqref{eq:U1U2U}, we get
  \begin{equation*}
    U_1'(\X)\dot \X=\V_3(\X)\dot \X+\W_3(\Y)\dot \Y
  \end{equation*}
  and, after taking the square of this expression,
  we obtain
  \begin{equation}
    \label{eq:derU1sum}
    U_1'(\X)^2\dot \X^2\leq 2(\V_3(\X)^2\dot \X^2+\W_3(\Y)^2\dot \Y^2).
  \end{equation}
  Since
  \begin{equation*}
    \W_3(\Y)^2\dot \Y^2=c\W_2\W_4\dot \Y^2=c\V_2\dot \X\W_4\dot \Y
\text{ (by \eqref{eq:V2W2equal})},
  \end{equation*}
  we have that \eqref{eq:derU1sum} implies
  \begin{align*}
    U_1'(\X)^2\dot \X&\leq C(\V_3(\X)^2\dot \X+\W_4(\Y)\dot \Y)\\
    &\leq C(\V_3(\X)^2\dot \X+\dot J)
  \end{align*}
  and, after a change of variables, we obtain
  \begin{equation*}
    \norm{U_1'}_{L^2}^2\leq C(\norm{\V_3}_{L^2}^2+\norm{J}_{L^\infty})<\infty.
  \end{equation*}
  Hence, $U_1'$ belongs to $L^2$. Similarly one
  proves that $U_2$ is absolutely continuous on
  any compact set and $U_2'\in L^2(\Real)$. To
  prove that the property \eqref{eq:Gprop} is
  fulfilled, we use Lemma \ref{lem:charH} and the
  fact that $1/(\V_2+\V_4),1/(\W_2+\W_4)\in
  L^{\infty}(\Real)$. The other properties of $\F$
  that $\psi$ has to fulfill can be checked more
  or less directly from the definition of $\G_0$.
\end{proof}

The sets $\F$ and $\G_0$ are not in bijection;
otherwise we would not have introduced $\F$, and
indeed one can show that
$\Cb\circ\Db\neq\id_{\G_0}$. However, we have
$\Db\circ\Cb=\id_{\F}$, as we will see in Lemma
\ref{lem:cdeeqe}.

Now we define how, from any initial data in $\D$,
that is, in the set of original coordinates, we
define the corresponding element in $\F$.
\begin{definition}
  \label{def:mappingL}
  We define the mapping $\Lb\colon\D\to\F$ where,
  for any $(u,R,S,\mu,\nu)\in\D$,
  $\psi=(\psi_1,\psi_2)=\Lb(u,R,S,\mu,\nu)$ is
  defined as follows. We set
  \begin{subequations}
    \label{eq:defvarL}
    \begin{align}
      \label{eq:defx1}
      x_1(X)&=\sup\{x\in\Real\mid x'+\mu(-\infty,x')<X\text{  for all }x'<x\},\\
      \label{eq:defx2}
      x_2(Y)&=\sup\{x\in\Real\mid x'+\nu(-\infty,x')<Y\text{  for all }x'<x\}
    \end{align}
    and
    \begin{equation}
      \label{eq:defJ1J2}
      J_1(X)=X-x_1(X),\quad     J_2(Y)=Y-x_2(Y)
    \end{equation}
    and
    \begin{equation}
      \label{eq:defU1U2}
      U_1(X)=u(x_1(X)),\quad  U_2(Y)=u(x_2(Y))
    \end{equation}
    and
    \begin{equation}
      \label{eq:defV1V2}
      V_1(X)=\left[\frac{R}{2c(U_1)}\right](x_1(X))x_1'(X),\     V_2(Y)=-\left[\frac{S}{2c(U_2)}\right](x_2(Y))x_2'(Y)
    \end{equation}
    and
    \begin{equation}
      \label{eq:defK1K2}
      K_1(X)=\int_{-\infty}^X\frac{J_1'(\bar X)}{c(U_1(\bar X))}\,d\bar X,\quad K_2(Y)=-\int_{-\infty}^Y\frac{J_2'(\bar Y)}{c(U_2(\bar Y))}\,d\bar Y.
    \end{equation}
  \end{subequations}
\end{definition}
Before proving the well-posedness of this
definition, we check that we end up with the same
initial data that was obtained at the end of
Section \ref{sec:equivsys}, where $\mu$ and $\nu$
were assumed to be absolutely continuous with
respect to the Lebesgue measure. Then, the
functions
\begin{equation*}
  \mu(-\infty,x')=\int_{-\infty}^{x'}\frac14 R^2\,dx\quad\text{and}\quad\nu(-\infty,x')=\int_{-\infty}^{x'}\frac14 S^2\,dx
\end{equation*}
are continuous, and furthermore, \eqref{eq:defx1} and
\eqref{eq:defx2} rewrite as 
\begin{equation*}
  x_1(X)+\int_{-\infty}^{x_1(X)}\frac14 R^2\,dx=X\quad\text{ and }\quad x_2(Y)+\int_{-\infty}^{x_2(Y)}\frac14 S^2\,dx=Y.
\end{equation*}
We sum these two equalities, and, since
$x_1(\X(s))=x_2(\Y(s))=x(s)$ and $\X+\Y=2s$, we
get
\begin{equation*}
  2x(s)+\int_{-\infty}^{x(s)}\frac14(R^2+S^2)\,dx=2s
\end{equation*}
and recover \eqref{eq:defphi}. In the definitions
\eqref{eq:defvarL}, we use the degree of freedom
we have in the new set of coordinates to set the
values of $x_1$ and $x_2$ in such a way that their
derivatives, $x_1'$ and $x_2'$, are bounded.
\begin{proof}[Proof of well-posedness of Definition \ref{def:mappingL}]
  Clearly, the definition of $x_1$ yields an
  increasing function and
  $\lim_{X\rightarrow\pm\infty}x_1(X)=\pm\infty$. For
  any $z>x_1(X)$, we have $X\leq
  z+\mu((-\infty,z))$. Hence, $X-z\leq\mu(\Real)$
  and, since we can choose $z$ arbitrarily close
  to $x_1(X)$, we get
  $X-x_1(X)\leq\mu(\Real)$. It is not hard to
  check that $x_1(X)\leq X$. Hence,
  \begin{equation}
    \label{eq:ximybd}
    \abs{x_1(X)-X}\leq\mu(\Real)
  \end{equation}
  and
  $\norm{x_1-\id}_{L^\infty}\leq\mu(\Real)$. Let
  us prove that $x_1$ is Lipschitz with Lipschitz
  constant at most one. We consider $X$, $X'$ in
  $\Real$ such that $X<X'$ and
  $x_1(X)<x_1(X')$. It follows from the definition
  that there exists an increasing sequence,
  $z_i'$, and a decreasing one, $z_i$, such that
  $\lim_{i\rightarrow\infty}z_i=x_1(X)$,
  $\lim_{i\rightarrow\infty}z_i'=x_1(X')$ with
  $\mu((-\infty,z_i'))+z_i'<X'$ and
  $\mu((-\infty,z_i))+z_i\geq X$. Combining the
  these two inequalities, we
  obtain
  \begin{equation}
    \label{eq:diffmu}
    \mu((-\infty,z_i'))-\mu((-\infty,z_i))+z_i'-z_i<X'-X.
  \end{equation}
  For $j$ large enough, since by assumption
  $x_1(X)<x_1(X')$, we have $z_i<z_i'$ and
  therefore
  $\mu((-\infty,z_i'))-\mu((-\infty,z_i))=\mu([z_i,z_i'))\geq0$. Hence,
  $z_i'-z_i<X'-X$. Letting $i$ tend to infinity,
  we get $x_1(X')-x_1(X)\leq X'-X$. Hence, $x_1$
  is Lipschitz with Lipschitz constant bounded by
  one and, by Rademacher's theorem, differentiable
  almost everywhere. Following \cite{Folland}, we
  decompose $\mu$ into its absolute continuous,
  singular continuous and singular part, denoted
  $\muac$, $\mu_{\text{sc}}$ and $\mus$,
  respectively. We have
  $\muac=\frac14 R^2\,dx$. The support of $\mus$
  consists of a countable set of points. Let
  $H(x)=\mu((-\infty,x))$, then $H$ is lower
  semi-continuous and its points of discontinuity
  exactly coincide with the support of $\mus$ (see
  \cite{Folland}). Let $A$ denote the complement
  of $x_1\inv(\supp(\mus))$. We claim that for any
  $X\in A$, we have
  \begin{equation}
    \label{eq:claimmu}
    \mu((-\infty,x_1(X)))+x_1(X)=X.
  \end{equation}
  From the definition of $x_1(X)$ follows the
  existence of an increasing sequence $z_i$ which
  converges to $x_1(X)$ and such that
  $H(z_i)+z_i<X$. Since $H$ is lower
  semi-continuous,
  $\lim_{i\rightarrow\infty}H(z_i)=H(x_1(X))$ and
  therefore
  \begin{equation}
    \label{eq:Fxi2}
    H(x_1(X))+x_1(X)\leq X.
  \end{equation}
  Let us assume that $H(x_1(X))+x_1(X)<X$. Since
  $x_1(X)$ is a point of continuity of $H$, we can
  then find an $x$ such that $x>x_1(X)$ and
  $H(x)+x<X$. This contradicts the definition of
  $x_1(X)$ and proves our claim
  \eqref{eq:claimmu}. In order to check that
  \eqref{eq:relxJV1} is satisfied, we have to
  compute the derivative of $x_{1}$. We define the
  set $B_1$ as
  \begin{equation*}
    B_1=\left\{x\in\Real\mid\lim_{\rho\downarrow0}\frac{1}{2\rho}\mu((x-\rho,x+\rho))=\frac14
      R^2(x)\right\}.
  \end{equation*}
  Since $\frac14 R^2(x)\,dx$ is the absolutely
  continuous part of $\mu$, we have, from
  Besicovitch's derivation theorem (see
  \cite{Ambrosio}), that $\meas(B_1^c)=0$. Given
  $X\in x_1\inv(B_1)$, we denote $x=x_1(X)$. We
  claim that for all $i\in\Nat$, there exists
  $0<\rho<\frac{1}{i}$ such that $x-\rho$ and
  $x+\rho$ both belong to $\supp(\mus)^c$. Assume
  namely the opposite. Then for any
  $z\in(x-\frac{1}{i},x+\frac{1}{i})\setminus\supp(\mus)$,
  we have that $z'=2x-z$ belongs to
  $\supp(\mus)$. Thus we can construct an
  injection between the uncountable set
  $(x-\frac{1}{i},x+\frac{1}{i})\setminus\supp(\mus)$
  and the countable set $\supp(\mus)$. This is
  impossible, and our claim is proved. Hence,
  since $x_1$ is surjective, we can find two
  sequences $X_i$ and $X_i'$ in $A$ such that
  $\frac12(x_1(X_i)+x_1(X_i'))=x_1(X)$ and
  $x_1(X_i')-x_1(X_i)<\frac{1}{i}$. We have, by
  \eqref{eq:claimmu}, since $x_1(X_i)$ and
  $x_1(X_i')$ belong to $A$,
  \begin{equation}
    \label{eq:muyxii}
    \mu([x_1(X_i),x_1(X_i')))+x_1(X_i')-x_1(X_i)=X_i'-X_i.
  \end{equation}
  Since $x_1(X_i)\notin\supp(\mus)$, we infer that
  $\mu(\{x_1(X_i)\})=0$ and
  $\mu([x_1(X_i),x_1(X_i')))=\mu((x_1(X_i),x_1(X_i')))$. Dividing
  \eqref{eq:muyxii} by $X_i'-X_i$ and letting $i$
  tend to $\infty$, we obtain
  \begin{equation}
    \label{eq:yxidef}
    x_1'(X)\frac14 R^2(x_1(X))+x_1'(X)=1
  \end{equation}
  where $x_1$ is differentiable in $x_1\inv(B_1)$,
  that is, almost everywhere in $x_1\inv(B_1)$. We
  will use several times this short lemma whose
  proof can be found in \cite{HolRay:07}.
  \begin{lemma}
    \label{lem:short} 
    Given an increasing Lipschitz function
    $f\colon\Real\to\Real$, for any set $B$ of
    measure zero, we have $f'=0$ almost everywhere
    in $f\inv(B)$.
  \end{lemma}
  We apply Lemma \ref{lem:short} to $B_1^c$ and
  get, since $\meas(B_1^c)=0$, that $x_1'=0$
  almost everywhere on $x_1\inv(B_1^c)$. On
  $x_1\inv(B_1)$, we proved that $x_1'$
  satisfies \eqref{eq:yxidef}. It follows that
  $0\leq x_1'\leq1$ almost everywhere, which
  implies, since $J_1'=1-x_1'$, that
  $J_1'\geq0$. From \eqref{eq:yxidef}, we get
  \begin{equation*}
    x_1(X)'J_1'(X)=x_1'(X)^2\frac14 R^2(x_1(X))=(c(U_1(X))V_1(X))^2.
  \end{equation*}
  Let us prove that $U_1$ is absolutely continuous
  on any bounded interval. We consider a
  partition $X_1\leq\cdots\leq X_N$. We have
  \begin{equation*}
    \sum_{i=1}^N\abs{U_1(X_{i+1})-U_1(X_i)}\leq\int_{\cup_{i=1}^N(x_1(X_i),x_1(X_{i+1}))}\abs{u_x}\,dx.
  \end{equation*}
  Given $M>0$, for any $\epsi>0$, there exists
  $\delta$ such that for any set $A\subset[-M,M]$,
  we have that $\meas(A)<\delta$ implies
  $\int_A\abs{u_x}\,dx<\epsi$, because $u_x\in
  L^1_{\rm loc}$. We have
  \begin{equation*}
    \meas(\cup_{i=1}^N(x_1(X_i),x_1(X_{i+1})))\leq\norm{x_1'}_{L^\infty}\meas(\cup_{i=1}^N(X_i,X_{i+1}))
  \end{equation*}
  and since $x_1'\in L^\infty$, it follows that
  for any partition such that
  $\sum_{i}^N\abs{X_i-X_{i+1}}<\delta$, we have
  $\sum_{i=1}^N\abs{U_1(X_{i+1})-U_1(X_i)}<\epsi$
  and $U_1$ is absolutely continuous in any
  compact. Let us consider a curve $(\X,\Y)\in\C$
  such that $x_1(\X(s))=x_2(\Y(s))$ (such curves
  exist, see Definition \ref{def:C}). We
  differentiate $U_1(\X)$ and obtain
  \begin{align*}
    U_1'(\X)\dot \X&=u_x(x_1(\X))x_1'(\X)\dot \X=\frac{(R-S)(x_1(\X))}{2c(U_1(\X))}x_1'(\X)\dot \X\\
    &=\frac{R(x_1(\X))}{2c(U_1(\X))}x_1'(\X)\dot \X-\frac{S(x_2(\Y))}{2c(U_2(\Y))}x_2'(\Y)\dot \Y\\
    &=V_1(\X)\dot \X+V_2(\Y)\dot \Y.
  \end{align*}
  Here we have use the fact that
  \begin{equation*}
    x_1'(\X)\dot \X=x_2'(\Y)\dot \Y
  \end{equation*}
  which follows from $x_1(\X)=x_2(\Y)$.  From
  \eqref{eq:ximybd}, we obtain
  $\norm{J_1}_{L^\infty}\leq\mu(\Real)$. We have,
  since $J_1'\geq0$,
  \begin{equation*}
    \norm{J_1'}_{L^2}^2\leq \norm{
      J_1'}_{L^\infty} \norm{J_1'}_{L^1}
    \leq\norm{J_1}_{L^\infty}\leq\mu(\Real).
  \end{equation*} 
  After a change of variables, we get
  \begin{equation*}
    \int_{\Real}V_1(X)^2\,dX\leq\frac{\kappa^2}4\int_{\Real}R^2(x)\,dx<\infty
  \end{equation*}
  and
  \begin{align*}
    \int_\Real U_1'^2(X)\,dX&=\int_\Real
    u_x^2(x_1(X))x_1'^2(X)\,dX\\&\leq\int_\Real
    u_x^2(x_1(X))x_1'(X)\,dX=\int_{\Real}u_x^2(x)\,dx<\infty,
  \end{align*}
  so that both $V_1$ and $U_1'$ belong to
  $L^2$. Similarly, one proves that $V_2$ and
  $U_2'$ belong to $L^2$.  Let
  $B_3=\{\xi\in\Real\mid
  x_1'<\frac{1}{2}\}$. Since $x_1'-1\geq0$,
  $B_3=\{\xi\in\Real\mid
  \abs{x_1'-1}>\frac{1}{2}\}$, and, after using
  the Chebychev inequality, as
  $x_1'-1=-J_1'\in{L^2}$, we obtain
  $\meas(B_3)<\infty$. Hence,
  \begin{align*}
    \int_\Real U_1^2(X)\,dX&=\int_{B_3}
    U_1^2(X)\,dX+\int_{B_3^c} U_1^2(X)\,dX\\
    &\leq \meas(B_3)\norm{u}_{L^\infty}^2+2\int_{B_3^c}(u\circ x_1)^2x_1'\,dX\\
    &\leq \meas(B_3)\norm{u}_{L^\infty}^2+2\norm{u}_{L^2}^2,
  \end{align*}
  after a change of variables, and $U_1\in
  L^2$. Similarly, one proves that $U_2\in L^2$.
\end{proof}
In this section we have shown how to construct,
from a given initial data in $\D$, an element in
$\F$ (via the mapping $\Lb$) and then, from an
element in $\F$, and element in $\G_0$ (via the
mapping $\Cb$). From an element in $\G_0$, we can
finally construct the corresponding solution of
\eqref{eq:goveq}.

Now, we turn to the existence
of solution to \eqref{eq:goveq} for given data in
$\G$.
 
\section{Existence of solution for the equivalent system}

\label{sec:existence}

\subsection{Short-range existence}

We first establish the short-range existence of
solutions to \eqref{eq:goveq}. The difficulty here
consists of taking into account initial data
defined on a curve which may be parallel to the
characteristic curves $X=\text{constant}$ or
$Y=\text{constant}$. In the following, we will
denote by $\Omega$ any
rectangular domain of the type
\begin{equation*}
  \Omega=[X_l,X_r]\times[Y_l,Y_r],
\end{equation*}
and we denote $s_l=\frac12(X_l+Y_l)$ and
$s_r=\frac12(X_r+Y_r)$. We define curves in
$\Omega$ as follows.
\begin{definition}
  Given $\Omega=[X_l,X_r]\times[Y_l,Y_r]$, we
  denote by $\C(\Omega)$ the set of curves in
  $\Omega$ given by $(\X(s),\Y(s))$ for
  $s\in[s_l,s_r]$ which match the diagonal points
  of $\Omega$, that is, $\X(s_l)=X_l$,
  $\X(s_r)=X_r$, $\Y(s_l)=Y_l$, $\Y(s_r)=Y_r$, and
  such that
  \begin{subequations}
    \begin{align}
      &\X-\id,\ \Y-\id \in W^{1,\infty}([s_l,s_r]),\\
      &\dot\X\geq0,\quad\dot\Y\geq0,\\
      &\frac12(\X(s)+\Y(s))=s,\text{ for all
      }s\in\Real.
    \end{align}
  \end{subequations}
  We set
  \begin{equation*}
    \norm{(\X,\Y)}_{\C(\Omega)}=\norm{\X-\id}_{L^\infty([s_l,s_r])}+\norm{\Y-\id}_{L^\infty([s_l,s_r])}.
  \end{equation*}
\end{definition}
In this subsection we will construct solutions on
small rectangular domains $\Omega$. We introduce
the set $\G(\Omega)$ which is the counterpart of
$\G$ on bounded intervals. Elements of
$\G(\Omega)$ correspond to a curve in $\C(\Omega)$
and data on this curve.
\begin{definition}
  \label{def:setGI}
  Given a rectangular domain
  $\Omega=[X_l,X_r]\times[Y_l,Y_r]$, let
  $\Theta=(\X,\Y,\Z,\V,\W)$ where
  $(\X,\Y)\in\C(\Omega)$ and $\Z(s),\V(X),\W(Y)$
  are three five-dimensional vector-valued
  measurable functions. Using the same notation as
  in \eqref{eq:defZa}, we set
  \begin{equation*}
    \norm{\Theta}_{\G(\Omega)}=\norm{U}_{L^2([s_l,s_r])}
    +\norm{\V^a}_{L^2([X_l,X_r])}+\norm{\W^a}_{L^2([Y_l,Y_r])}
  \end{equation*}
  where we denote $U=\Z_3$ and
  \begin{multline*}
    \tnorm{\Theta}_{\G(\Omega)}=\norm{(\X,\Y)}_{\C(\Omega)}+\norms{\frac1{\V_2+\V_4}}_{L^\infty([X_l,X_r])}+\norms{\frac1{\W_2+\W_4}}_{L^\infty([Y_l,Y_r])}\\
    +\norm{\Z^a}_{L^\infty([s_l,s_r])}+\norm{\V^a}_{L^\infty([X_l,X_r])}+\norm{\W^a}_{L^\infty([Y_l,Y_r])}.
  \end{multline*}
  The element $\Theta$ belongs to $\G(\Omega)$ if the following four
  conditions hold:
  \begin{enumerate}
  \item[(i)]
    \begin{equation*}
      \tnorm{\Theta}_{\G(\Omega)}<\infty
    \end{equation*}
    and therefore
    $\norm{\Theta}_{\G(\Omega)}<\infty$ because we
    here consider a bounded domain.
  \item[(ii)]
    \begin{equation}
      \label{eq:posVW}
      \V_2,\W_2,\Z_4,\V_4,\W_4\geq0.
    \end{equation}
  \item[(iii)] For almost every $s$, we have
    \begin{equation}
      \label{eq:relZVW}
      \dot\Z(s)=\V(\X(s))\dot\X(s)+\W(\Y(s))\dot\Y(s).
    \end{equation}
  \item[(iv)]
    For almost every $X$ and $Y$, we have
    \begin{subequations}
      \begin{align}
        \label{eq:relVWb}
        2\V_4(\X)\V_2(\X)&=(c(U)\V_3(\X))^2,& 2\W_4(\Y)\W_2(\Y)&=(c(U)\W_3(\Y))^2,\\
        \label{eq:reltx1b}
        \V_{2}(\X)&=c(U)\V_{1}(\X),& \W_{2}(\Y)&=-c(U)\W_{1}(\Y),\\
        \label{eq:reltJK1b}
        \V_{4}(\X)&=c(U)\V_{5}(\X),&\W_{4}(\Y)&=-c(U)\W_{5}(\Y).
      \end{align}
    \end{subequations}
  \end{enumerate}
\end{definition}
We introduce the Banach spaces $\WX(\Omega)$ and
$\WY(\Omega)$ defined as
\begin{equation}
  \label{eq:defWXY1}
  \WX(\Omega)=L^\infty([Y_l,Y_r],W^{1,\infty}([X_l,X_r])),\quad \WY(\Omega)=L^\infty([X_l,X_r],W^{1,\infty}([Y_l,Y_r])),
\end{equation}
and the Banach spaces $L_X^\infty(\Omega)$ and
$L_Y^\infty(\Omega)$ defined as
\begin{equation*}
  L_X^\infty(\Omega)=L^\infty([Y_l,Y_r],C([X_l,X_r])),\quad L_Y^\infty(\Omega)=L^\infty([X_l,X_r],C([Y_l,Y_r])).
\end{equation*}
Let us consider $(\X,\Y,\Z,\V,\W)\in\G(\Omega)$.
By definition, the functions $\X$ and $\Y$ are
increasing. To any increasing function, one can
associate its generalized inverse, a concept which
is exposed for example in Brenier
\cite{brenier:09}. More generally, an increasing
function (not necessarily continuous) can be
identified as the subdifferential of a convex function
(which is a multivalued function). The generalized
inverse is then the subdifferential of the
conjugate of this convex function. We do not use
this framework here and prove directly the results
we need.
\begin{definition}
  \label{def:XYmap}
  Given $\Omega=[X_l,X_r]\times[Y_l,Y_r]$ and a
  curve $(\X,\Y)\in\C(\Omega)$, we define the
  generalized inverse of $\X$ and $\Y$,
  respectively, as
  \begin{align}
    \label{eq:defalpha}
    \alpha(X)&=\sup\{s\in[s_l,s_r]\mid
    \X(s)<X\} \text{   for
    }X\in (X_l,X_r],\\
    \label{eq:defbeta}
    \beta(Y)&=\sup\{s\in[s_l,s_r]\mid
    \Y(s)<Y\} \text{   for
    }Y\in (Y_l,Y_r].
  \end{align}
  We denote $\X^{-1}=\alpha$ and $\Y^{-1}=\beta$.
\end{definition}
The generalized inverse functions $\X^{-1}$ and
$\Y^{-1}$ enjoy the following properties.
\begin{lemma}
  \label{lem:geninv}
  The functions $\X^{-1}$ and $\Y^{-1}$ are lower
  semicontinuous nondecreasing functions. We have
  \begin{equation}
    \label{eq:relXalpha}
    \X\circ\X^{-1}=\id\text{ and }\Y\circ\Y^{-1}=\id,
  \end{equation}
  and
  \begin{align}
    \label{eq:relalphaX}
    \X^{-1}&\circ \X(s)=s\text{  for any }s\text{ such that }\dot\X(s)>0,\\
    \label{eq:relalphaY}
    \Y^{-1}&\circ \Y(s)=s\text{  for any }s\text{ such that }\dot\Y(s)>0.
  \end{align}
\end{lemma}
Definition \ref{def:XYmap} extends naturally to
curves in $\C$ and Lemma \ref{lem:geninv} still
holds.
\begin{proof}
  We prove the lemma only for $\X^{-1}$, as the
  results for $\Y^{-1}$ can be proved in the same
  way. Let us prove that $\alpha$ is
  nondecreasing. For any $X<\bar X$, there exists
  a sequence $s_i$ such that
  $\lim_{i\to\infty}s_i=\alpha(X)$ and
  $\X(s_i)<X$. Hence, $\X(s_i)<\bar X$ which
  implies $s_i\leq \alpha(\bar X)$, which after
  letting $i$ tend to infinity, gives
  $\alpha(X)\leq \alpha(\bar X)$. Let us prove
  that $\alpha$ is lower semicontinuous. Given a
  sequence $X_i$ such that
  $\lim_{i\to\infty}X_i=X$, for any $\epsi>0$,
  there exists $s\in[s_l,s_r]$ such that
  \begin{equation}
    \label{eq:xstrbd}
    \alpha(X)>s>\alpha(X)-\epsi
  \end{equation}
  because $\alpha(X)>s_l$ for all $X\in(X_l,X_r]$.
  It implies $\X(s)<X$ as, otherwise, $X\leq
  \X(s)$ would yield $\alpha(X)\leq s$, which
  contradicts \eqref{eq:xstrbd}. Thus, for large
  enough $i$, we have $\X(s)<X_i$ so that
  $s<\alpha(X_i)$. Combined with
  \eqref{eq:xstrbd}, it implies
  \begin{equation*}
    \alpha(X)-\epsi<s\leq\liminf\alpha(X_i)
  \end{equation*}
  and, as $\epsi$ is arbitrary, we get that
  $\alpha$ is lower semicontinuous. Let us prove
  \eqref{eq:relXalpha}. Given $X\in(X_l,X_r]$, we
  consider an increasing sequence $s_i$ such that
  $\lim_{i\to\infty} s_i=\alpha(X)$ and
  $\X(s_i)<X$. Letting $i$ tend to infinity, since
  $\X$ is continuous, we get $\X(\alpha(X))\leq
  X$. Assume that $\X(\alpha(X))<X$, since $\X$ is
  continuous, there exists $s$ such that
  $\X(\alpha(X))<\X(s)$ and $\X(s)<X$. The latter
  inequality implies that $s\leq\alpha(X)$ which,
  by the monotonicity of $\X$, yields
  $\X(s)\leq\X(\alpha(X))$ and we obtain a
  contradiction. Let us prove
  \eqref{eq:relalphaX}. We denote
  $\N=\{s\in[s_l,s_r]\mid \dot X(s)>0\}$. We
  consider a fixed element $s_0\in \N$. We have
  \begin{equation}
    \label{eq:ineqalphaX}
    \alpha\circ \X(s_0)\leq s_0.
  \end{equation}
  Indeed, by the monotonicity of $\X$, for any
  $s\in \{s\in[s_l,s_r]\mid \X(s)<\X(s_0)\}$, we
  have $s<s_0$ and therefore, after taking the
  supremum, we obtain \eqref{eq:ineqalphaX}. Let
  us assume that $\alpha\circ \X(s_0)< s_0$. We
  denote $s_1=\alpha\circ\X(s_0)$. By
  \eqref{eq:relXalpha}, $\X(s_1)=\X(s_0)$, and
  from the monotonicity of $\X$, it follows that
  $\X(s)=\X(s_0)=\X(s_1)$ for all
  $s\in[s_0,s_1]$. It implies that
  $\dot\X(s_0)=0$, which contradicts the fact that
  $s_0\in\N$.
\end{proof}
In the following and when there is no ambiguity,
we will slightly abuse the notation and denote
$\Y\circ \X^{-1}(X)$ and $\X\circ \Y^{-1}(Y)$ by
$\Y(X)$ and $\X(Y)$, respectively. The curve
$(\X(s),\Y(s))$ is \textit{almost} a graph as it
consists of the union of the graphs of the
functions $X\mapsto\Y(X)$ and (after rotating the
axes by $\frac\pi2$) $Y\mapsto\X(Y)$. We prove
the existence of solutions to \eqref{eq:goveq} on
rectangular boxes. First we give the definition of
solutions.
\begin{definition}
  \label{def:sollocal}
  We say that $Z$ is solution to \eqref{eq:goveq}
  in $\Omega=[X_l,X_r]\times[Y_l,Y_r]$ if
  \begin{enumerate}
  \item[(i)] we have 
    \begin{equation*}
      Z\in\Winf(\Omega),\quad Z_X\in\WY(\Omega),\quad Z_Y\in\WX(\Omega);
    \end{equation*}
  \item[(ii)] and for almost every
    $X\in[X_l,X_r]$,
    \begin{equation}
      \label{eq:dercroisX}
      (Z_X(X,Y))_Y=F(Z)(Z_X,Z_Y)(X,Y);
    \end{equation}
    and, for almost every $Y\in[Y_l,Y_r]$,
    \begin{equation}
      (Z_Y(X,Y))_X=F(Z)(Z_X,Z_Y)(X,Y).
    \end{equation}
  \end{enumerate}
  We say that $Z$ is a global solution to
  \eqref{eq:goveq} if $Z$ is a solution to
  \eqref{eq:goveq} as defined above, for any
  rectangular domain $\Omega$.
\end{definition}
The regularity that we impose is also necessary to
extract the relevant data on a curve from a
function defined in the plane, as it is explained
in the following lemma.
\begin{lemma}[Extraction of data from a curve]
  We consider a five-dimensional vector function
  $Z$ in $\Real^2$ such that
  \begin{equation*}
    Z\in\Winf(\Omega),\quad Z_X\in\WY(\Omega),\quad Z_Y\in\WX(\Omega)
  \end{equation*}
  for any rectangular domain $\Omega$.  Then,
  given a curve $(\X,\Y)\in\C$, let $(\Z,\V,\W)$
  be defined as
  \begin{equation}
    \label{eq:defextdatZ}
    \Z(s)=Z(\X(s),\Y(s))\text{ for all }s\in\Real
  \end{equation}
  and
  \begin{subequations}
    \label{eq:defextdata}
    \begin{align}
      \V(X)&=Z_X(X,\Y(X))\text{   for a.e.~$X\in\Real$},\\
      \W(Y)&=Z_Y(\X(Y),Y)\text{   for a.e.~$Y\in\Real$},
    \end{align}
  \end{subequations}
  or, equivalently,
  \begin{align*}
    \V(\X(s))&=Z_X(\X(s),\Y(s))\text{   for a.e.~$s\in\Real$  
                   such that $\dot \X(s)>0$},\\
    \W(\Y(s))&=Z_Y(\X(s),\Y(s))\text{   for a.e.~$s\in\Real$
    such that $\dot \Y(s)>0$}.
  \end{align*}
  We have $(\Z,\V,\W)\in
  L^\infty_{\rm loc}(\Real)$ and we denote
  $\Theta=(\X,\Y,\Z,\V,\W)$ by
  \begin{equation*}
    Z\bullet(\X,\Y).
  \end{equation*}
\end{lemma}
\begin{proof}
  We consider a domain
  $\Omega=[X_l,X_r]\times[Y_l,Y_r]$. We claim that
  for any $f\in\WY(\Omega)$ then $\tilde
  f(X)=f(X,\Y(X))$ is measurable and $f(X,\Y(X))\in
  L^{\infty}([X_l,X_r])$. It suffices to show that
  the linear mapping $f\mapsto \tilde{f}$ from
  $\WY(\Omega)$ to $\Linf([X_l,X_r])$ is
  well-defined on simple functions and
  continuous. We assume that $f$ is a simple
  function, that is,
  \begin{equation*}
    f(X,Y)=\sum_{j=1}^{N}g_{j}(Y)\chi_{A_{j}}(X)
  \end{equation*}
  where $\chi_A$ denotes the indicator function of
  the set $A$, $A_{j}$ are disjoint measurable
  sets and $g_{j}\in\Winf([Y_l,Y_r])$. Then,
  $\tilde
  f(X)=\sum_{j=1}^{N}g_{j}(\Y(X))\chi_{A_{j}}(X)$
  is measurable (as $X\mapsto\Y(X)$ is lower
  semicontinuous) and
  \begin{align*}
    \esssup_{X\in[X_l,X_r]} \abs{\tilde f(X)}&\leq \max_{j\in\{1,\ldots,N\}}\esssup_{X\in[X_l,X_r]}\abs{g_{j}(\Y(X))}\\
    &\leq\max_{j\in\{1,\ldots,N\}}\norm{g_{j}}_{\Winf([Y_l,Y_r])}\\
    &\leq\norm{f}_{\WY(\Omega)}
  \end{align*}
  so that $\tilde f\in L^\infty([X_l,X_r])$. Note
  that we need $g_j\in \Winf([Y_l,Y_r])$ as, if
  $g_j$ only belongs to $L^\infty([Y_l,Y_r])$, we
  \textit{do not have} in general
  $\esssup_{X\in[X_l,X_r]}\abs{g_{j}(\Y(X))}\leq\norm{g_{j}}_{L^\infty([Y_l,Y_r])}$
  as the function $X\mapsto\Y(X)$ may send sets of
  strictly positive measure to a set of measure
  zero (for example if $\Y$ is constant on an
  interval). Therefore the continuity in the $Y$
  direction which is necessary to make meaning of
  \eqref{eq:defextdata}. Using the same type of
  estimate, one gets that
  \begin{equation*}
    \norms{\tilde f}_{\Linf([X_l,X_r])}\leq\norm{f}_{\WY(\Omega)},
  \end{equation*}
  which concludes the proof of the claim. Similarly
  one proves that, for any $f\in\WX(\Omega)$, the
  mapping $Y\mapsto f(\X(Y),Y)$ is measurable and
  belongs to $L^{\infty}([X_l,X_r])$. Hence, we
  get that $(\Z,\V,\W)\in
  L^\infty_{\rm loc}(\Real)$.
\end{proof}

The decay of $Z$ at infinity in the diagonal
direction is more conveniently expressed in term
of the function $Z^a$ which we now define as
\begin{equation}
  \label{eq:Zdefwa}
  Z^a_2=Z_2-\frac12(X+Y)
  \text{ and }Z^a_i=Z_i\text{ for }i\in\{1,3,4,5\}.
\end{equation}
Even if we are not concerned yet with the behavior
at infinity, it is convenient to introduce $Z^a$
already here to write the estimate in a convenient
way.  We now introduce the set $\H(\Omega)$ of all
solutions to \eqref{eq:goveq} on rectangular
domains, which satisfy additional properties.

\begin{definition}
  \label{def:HO} 
  Given a rectangular domain
  $\Omega=[X_l,X_r]\times[Y_l,Y_r]$, let
  $\H(\Omega)$ be the set of all functions $Z$
  which are solutions to \eqref{eq:goveq} in the
  sense of Definition \ref{def:sollocal} and which
  satisfy the following properties
  \begin{subequations}
    \label{eq:relpres}
    \begin{align}
      \label{eq:reltx}
      x_X&=c(U)t_X,&x_Y&=-c(U)t_Y,\\
      \label{eq:reltJK}
      J_X&=c(U)K_X,&J_Y&=-c(U)K_Y,\\
       \label{eq:energrel2}
      2J_Xx_X&=\left(c(U)U_X\right)^2,& 
2J_Yx_Y&=\left(c(U)U_Y\right)^2, \\
       \label{eq:posxX}
      x_X&\geq0, &J_X&\geq0,\\
      \label{eq:posxY}
      x_Y&\geq0,&J_Y&\geq0,\\
       \label{eq:invbd0}
      x_X+J_X&>0,& x_Y+J_Y&>0.
    \end{align}
  \end{subequations}
\end{definition}

We have the following short-range existence
theorem.

\begin{theorem}
  \label{th:shortrange}
  There exists an increasing function $C$ such
  that, for any $\Omega=[X_l,X_r]\times[Y_l,Y_r]$
  and $\Theta=(\X,\Y,\Z,\V,\W)$ in $\G(\Omega)$,
  if
  $s_r-s_l\leq 1/C(\tnorm{\Theta}_{\G(\Omega)})$,
  then there exists a unique solution
  $Z\in\H(\Omega)$ such that
  \begin{equation}
    \label{eq:matchdata}
    \Theta=Z\bullet(\X,\Y).
  \end{equation}
\end{theorem}


\begin{proof}
  We use a Picard fixed-point argument. 
  Define $\B$ as the set of elements
  $(Z_h,Z_v,V,W)$ such that
  \begin{equation*}
    Z_h\in[\LX]^5,\quad Z_v\in[\LY]^5,\quad V\in[\LY]^5,\quad W\in[\LX]^5
  \end{equation*}
  and
  \begin{equation}
    \label{eq:defBM}
    \sum_{i=1}^5(\norm{Z^a_{h,i}}_{\LX}+\norm{Z^a_{v,i}}_{\LY}+\norm{V_i}_{L_X^\infty}+\norm{W_i}_{L_X^\infty})\leq 2\tnorm{\Theta}_{\G(\Omega)}
  \end{equation}
  where we use for $Z_h$ and $Z_v$ the same
  notation given in \eqref{eq:Zdefwa} for $Z$. For
  the fixed point, the functions $Z_h$ and $Z_v$
  coincide and are equal to the solution $Z$, see
  below, but it is convenient to define both
  quantities in this proof and keep the symmetry
  of the problem with respect to the $X$ and $Y$
  variables.  We introduce the mapping $\P$ given,
  for any $(Z_h,Z_v,V,W)\in\B$, by
  $\P(Z_h,Z_v,V,W)=(\bar Z_h,\bar Z_v,\bar V,\bar
  W)$ where
  \begin{subequations}
    \label{eq:iter}
    \begin{align}
      \label{eq:iter11}
      \bar Z_h(X,Y)&=\Z(\Y^{-1}(Y))+\int_{\X(Y)}^X V(\tilde X,Y)\,d\tilde X\\
      \intertext{for a.e. $Y\in[Y_l,Y_r]$ and all
        $X\in[X_l,X_r]$,}
      \label{eq:iter12}
      \bar Z_v(X,Y)&=\Z(\X^{-1}(X))+\int_{\Y(X)}^Y W(X,\tilde Y)\,d\tilde Y\\
      \intertext{for a.e. $X\in[X_l,X_r]$ and all
        $Y\in[Y_l,Y_r]$,}
      \label{eq:iter2}
      \bar V(X,Y)&=\V(X)+\int_{\Y(X)}^Y F(Z_h)(V,W)(X,\tilde Y)\,d\tilde Y\\
      \intertext{for a.e. $X\in[X_l,X_r]$ and all $Y\in[Y_l,Y_r]$,}
      \label{eq:iter3}
      \bar W(X,Y)&=\W(Y)+\int_{\X(Y)}^X
      F(Z_h)(V,W)(\tilde X,Y)\,d\tilde X
    \end{align}
  \end{subequations}
  for a.e. $Y\in[Y_l,Y_r]$ and all
  $X\in[X_l,X_r]$. Let us consider a solution $Z$
  to \eqref{eq:goveq} which satisfies
  \eqref{eq:matchdata}. For any $(X,Y)\in\Omega$,
  we have
  \begin{equation*}
    Z(X,\Y(s))=\Z(s)+\int_{\X(s)}^X Z_X(\tilde X,\Y(s))\,d\tilde X
  \end{equation*}
  which, after taking $s=\Y^{-1}(Y)$, yields
  \begin{equation*}
    Z(X,Y)=\Z(\Y^{-1}(Y))+\int_{\X(Y)}^X Z_X(\tilde X,Y)\,d\tilde X,
  \end{equation*}
  by \eqref{eq:relXalpha}. Similarly, one proves that
  \begin{equation*}
    Z(X,Y)=\Z(\X^{-1}(X))+\int_{\Y(X)}^Y Z_Y(X,\tilde Y)\,d\tilde Y.
  \end{equation*}
  For every almost every $X\in[X_l,X_r]$ and all
  $Y\in[Y_l,Y_r]$, we have
  \begin{equation*}
    Z_X(X,Y)=Z_X(X,\Y(X))+\int_{\Y(X)}^Y
    F(Z)(Z_X,Z_Y)(X,\tilde Y)\,d\tilde Y,
  \end{equation*}
  and therefore, by \eqref{eq:matchdata}, we get 
  \begin{equation}
    \label{eq:ZXXX}
    Z_X(X,Y)=\V(X)+\int_{\Y(X)}^Y
    F(Z)(Z_X,Z_Y)(X,\tilde Y)\,d\tilde Y.
  \end{equation}
  Similarly, we have
  \begin{equation}
    Z_Y(X,Y)=\W(Y)+\int_{\X(Y)}^X
    F(Z)(Z_X,Z_Y)(\tilde X,Y)\,d\tilde X
  \end{equation}
  for all $X\in[X_l,X_r]$ and
  a.e. $Y\in[Y_l,Y_r]$. Thus, if $Z$ is a solution
  to \eqref{eq:goveq} (in the sense of Definition
  \ref{def:sollocal}) which satisfies
  \eqref{eq:matchdata} then $(Z,Z,Z_X,Z_Y)$ is a
  fixed point of $\P$. Since $0\leq\dot \X\leq 2$
  and $0\leq\dot Y\leq 2$, we have
  \begin{equation}
    \label{eq:bdxmxpi}
    \abs{X-\X(Y)}=\abs{\X(\alpha(X))-\X(\alpha(Y))}\leq \X(s_r)-\X(s_l)\leq 2(s_l-s_r)\leq 2\delta
  \end{equation}
  and, similarly,
  \begin{equation}
    \label{eq:bdymypi}
    \abs{Y-\Y(X)}\leq 2(s_l-s_r)\leq 2\delta.
  \end{equation}
  We can choose $\delta$ small enough, depending
  only on $\tnorm{\Theta}_{\G(\Omega)}$, such that
  the mapping $\P$ maps $\B$ into $\B$. Let us
  check this in more details only for the second
  component of $Z_h$. We have
  \begin{equation*}
    \bar Z^a_{h,2}=\bar Z_{h,2}-\frac12(X+Y)
    =\Z_2(\Y^{-1}(Y))-\frac12(X+Y)+\int_{\X(Y)}^X V(\tilde X,Y)\,d\tilde X
  \end{equation*}
  and, denoting $s=\Y^{-1}(Y)$,
  \begin{align*}
    \Z_2(s)-\frac12(X+Y)&=\Z_2(s)-\frac12(X+\Y(s))=\Z_2(s)-\frac12(\X(s)+\Y(s))+\frac12(X-\X(s))\\
    &=\Z_2^a(s)+\frac12(X-\X(Y)).
  \end{align*}
  Hence,
  \begin{equation*}
    \norm{\bar Z^a_{h,2}}_{L^{\infty}(\Omega)}\leq 
\tnorm{\Z_2^a}_{\G(\Omega)}+\delta(1+2\tnorm{\Theta}_{\G(\Omega)})
  \end{equation*}
  by \eqref{eq:bdxmxpi} and
  \eqref{eq:defBM}. After doing the same for the
  other components, we get
  \begin{equation*}
    \sum_{i=1}^5(\norm{\bar Z^a_{h,i}}_{\LX}+\norm{\bar Z^a_{v,i}}_{\LY}+\norm{\bar V_i}_{L_X^\infty}+\norm{\bar W_i}_{L_X^\infty})\leq \tnorm{\Theta}_{\G(\Omega)}+\delta C
  \end{equation*}
  for a constant $C$ which depends only on
  $\tnorm{\Theta}_{\G(\Omega)}$. Hence, by taking
  $\delta$ small enough, the mapping $\P$ maps
  $\B$ into $\B$. Using the fact that $F$ is
  locally Lipschitz (because it is bi-linear with
  respect to the two last variables and depends
  smoothly on $U=Z_3$), we prove that $\P$ is
  contractive. Hence, $\P$ admits a unique fixed
  point that we denote $(Z_h,Z_v,V,W)$. Let us
  prove that $Z_h=Z_v$. It basically follows from
  the fact that $W_X=V_Y$. Let us now denote
  by $\N_X$ the set of points $X\in[X_l,X_r]$ for
  which \eqref{eq:iter12} and \eqref{eq:iter2}
  hold (by definition, the set $\N_X$ has full
  measure). Similarly we denote by $\N_Y$ the set
  of points $Y\in[Y_l,Y_r]$ for which
  \eqref{eq:iter11} and \eqref{eq:iter3} hold. We
  have
  $\meas([X_l,X_r]\setminus\N_X)=\meas([Y_l,Y_r]\setminus\N_Y)=0$
  so that
  $\meas(\Omega\setminus\N_X\times\N_Y)=0$. For
  any $(X,Y)\in\N_X\times\N_Y$, we have
  \begin{align}
    \notag
    Z_h(X,Y)-Z_v(X,Y)&=\Z(\Y^{-1}(Y))+\int_{\X(Y)}^{X}V(\tilde X,Y)\,d\tilde X\\
    \label{eq:diffZ1Z2}
    &\quad -\Z(\X^{-1}(X))-\int_{\Y(X)}^YW(X,\tilde Y)\,d\tilde Y.
  \end{align}
  Since the terms involving $F$ cancel, we obtain
  by \eqref{eq:iter2} and \eqref{eq:iter3} that
  \begin{equation*}
    \int_{\X(Y)}^{X}V(\tilde X,Y)\,d\tilde X-\int_{\Y(X)}^YW(X,\tilde Y)\,d\tilde Y=\int_{\X(Y)}^X\V(\tilde X)\,d\tilde
    X-\int_{\Y(X)}^Y\W(\tilde Y)\,d\tilde
    Y.
  \end{equation*}
  Returning to the rigorous notation, we get
  \begin{equation}
    \label{eq:proofcom}
    \int_{\X(Y)}^{X}V(\tilde X,Y)\,d\tilde
    X-\int_{\Y(X)}^YW(X,\tilde Y)\,d\tilde
    Y=\int_{\X(\Y^{-1}(Y))}^{\X(\X^{-1}(X))}\V(\tilde
    X)\,d\tilde
    X-\int_{\Y(\X^{-1}(X))}^{\Y(\Y^{-1}(Y))}\W(\tilde
    Y)\,d\tilde Y
  \end{equation}
  where we have also used \eqref{eq:relXalpha}.
  We proceed with a change of variables in the two
  integrals on the right-hand side of
  \eqref{eq:proofcom} and get
  \begin{align}
    \notag
    \int_{\X(\Y^{-1}(Y))}^{\X(\X^{-1}(X))}&\V(\tilde
    X)\,d\tilde
    X-\int_{\Y(\X^{-1}(X))}^{\Y(\Y^{-1}(Y))}\W(\tilde
    Y)\,d\tilde
    Y&\\
    \notag
    &\quad=-\int_{\X^{-1}(X)}^{\Y^{-1}(Y)}\left(\W(\Y(s))\dot
      \Y(s)+\V(\X(s))\dot
      \X(s)\right)\,ds&\\
    \notag \notag
    &\quad=-\int_{\X^{-1}(X)}^{\Y^{-1}(Y)}\dot
    \Z(s)\,ds&\text{ by \eqref{eq:relZVW}}\\
    \label{eq:proofcom2}
    &\quad=\Z(\X^{-1}(X))-\Z(\Y^{-1}(Y))
  \end{align}
  and combining \eqref{eq:diffZ1Z2},
  \eqref{eq:proofcom} and \eqref{eq:proofcom2}, we
  get that $Z_h(X,Y)=Z_v(X,Y)$ for all
  $(X,Y)\in\N_X\times\N_Y$, that is, almost
  everywhere. We denote $Z=Z_h=Z_v$. For any
  $(X,Y)$ and $(\bar X,\bar Y)$ belonging to
  $\N_X\times\N_Y$, we get, by using
  \eqref{eq:iter11} and \eqref{eq:iter12}, that
  \begin{equation*}
    Z(X,Y)-Z(\bar X,\bar Y)=\int_{\bar X}^{X}V(\tilde X,Y)\,d\tilde X+\int_{\bar Y}^{Y}W(\bar X,\tilde Y)\,d\tilde Y.
  \end{equation*}
  Hence, by using the bound \eqref{eq:defBM},
  \begin{equation*}
    \abs{Z(X,Y)-Z(\bar X,\bar Y)}\leq 2\tnorm{\Theta}_{\G(\Omega)}(\abs{X-\bar X}+\abs{Y-\bar Y})
  \end{equation*}
  and $Z$ is Lipschitz in $\N_X\times\N_Y$. It
  implies that $Z$ is uniformly continuous in
  $\N_X\times\N_Y$ and there exists a unique
  continuous extension of $Z$ to the closure of
  $\N_X\times\N_Y$, that is, $\Omega$. From
  \eqref{eq:iter11} and \eqref{eq:iter12}, we get
  that
  \begin{equation}
    \label{eq:ZXeqV}
    Z_X(X,Y)=V(X,Y)    
  \end{equation}
  for all $Y\in[Y_l,Y_r]$ and a.e. $X\in[X_l,X_r]$
  and
  \begin{equation}
    \label{eq:ZYeqW}
    Z_Y(X,Y)=W(X,Y)
  \end{equation}
  for all $X\in[X_l,X_r]$ and
  a.e. $Y\in[Y_l,Y_r]$.  By using the fact that
  $(Z,Z,Z_X,Z_Y)$ is a fixed point in $\B$ and
  \eqref{eq:ZXeqV} and \eqref{eq:ZYeqW}, we can
  check that $Z_X\in\WY(\Omega)$ and
  $Z_Y\in\WX(\Omega)$.  By density, we can prove
  that
  \begin{equation}
    \label{eq:ZintforaX}
    Z(X,Y)-Z(\bar X,Y)=\int_{\bar X}^{X}Z_X(\tilde X,Y)\,dX,
  \end{equation}
  not only for almost every $Y\in[Y_l,Y_r]$ as
  \eqref{eq:iter11} yields,  but \textit{for all}
  $Y\in[Y_l,Y_r]$. Indeed, for any
  $Y\in[Y_l,Y_r]$, there exists a sequence
  $Y_n\in\N_Y$ such that $\lim_{n\to\infty}Y_n=Y$
  as $\meas([Y_l,Y_r]\setminus \N_Y)=0$ and we have
  \begin{equation}
    \label{eq:ZintforaXn}
    Z(X,Y_n)-Z(\bar X,Y_n)=\int_{\bar X}^{X}Z_X(\tilde X,Y_n)\,dX.
  \end{equation}
  Since $Z_X\in\WY(\Omega)$, we have that
  $\norm{Z_X(\dott,Y_n)}_{L^{\infty}([X_l,X_r])}\leq\norm{Z_X}_{\WY(\Omega)}\leq
  2\tnorm{\Theta}_{\G(\Omega)}$ and, for a.e. $\tilde
  X\in[X_l,X_r]$, $\lim_{n\to\infty}Z_X(\tilde
  X,Y_n)=Z_X(\tilde X,Y)$ for almost every $\tilde
  X$. Hence, by Lebesgue dominated convergence
  theorem and the continuity of $Z$,
  \eqref{eq:ZintforaXn} implies
  \eqref{eq:ZintforaX}. It remains to check that
  $Z$ satisfies \eqref{eq:matchdata}. Since
  $(Z,Z,Z_X,Z_Y)$ is a fixed point of $\P$, we
  have, by \eqref{eq:iter2} and \eqref{eq:iter3},
  that
  \begin{equation*}
    Z_X(X,\Y(X))=\V(X)\text{ and }Z_Y(\X(Y),Y)=\W(Y)
  \end{equation*}
  for a.e. $X\in[X_l,X_r]$ and $Y\in[Y_l,Y_r]$,
  respectively. It remains to check that $Z$
  satisfies \eqref{eq:defextdatZ}. On one hand, we
  have that
  \begin{equation}
    \label{eq:matchbcZ}
    Z(\X(s),\Y(s))=\Z(\X^{-1}(\X(s)))
  \end{equation}
  by \eqref{eq:iter12}  and,
  by \eqref{eq:relalphaX}, it implies
  \eqref{eq:defextdatZ} for all $s\in [s_l,s_r]$ such that
  $\dot\X(s)>0$. On the other hand we have
  \begin{equation}
    \label{eq:matchbcZ2}
    Z(\X(s),\Y(s))=\Z(\Y^{-1}(\Y(s)))
  \end{equation}
  by \eqref{eq:iter12} and, by
  \eqref{eq:relalphaY}, it implies
  \eqref{eq:defextdatZ} for all $s\in [s_l,s_r]$
  such that $\dot\Y(s)>0$. Since $\dot \X+\dot
  \Y=2$, the set of all $s\in [s_l,s_r]$ such that
  $\dot\X(s)>0$ or $\dot\Y(s)>0$ has full measure
  and therefore, for almost every $s\in
  [s_l,s_r]$, \eqref{eq:defextdatZ} holds. By
  continuity, we infer that \eqref{eq:defextdatZ}
  holds for all $s\in [s_l,s_r]$. Hence, we have
  proved that $Z$ is a solution to
  \eqref{eq:goveq} which satisfies
  \eqref{eq:matchdata} if and only if it is a
  fixed point of $\P$. Since the fixed point
  exists and is unique, we have proved the
  existence and uniqueness of the solution. Let us
  define the functions $v\in\WY(\Omega)$ and
  $w\in$ in $\WX(\Omega)$ as
  \begin{equation*}
    v=x_X-c(U)t_X\quad\text{ and }\quad w=x_Y+c(U)t_Y. 
  \end{equation*}
  We want to prove that $v$ and $w$ are both
  zero. After some computations using the governing
  equations \eqref{eq:goveq}, we obtain
  \begin{align}
    \notag
    v_Y&=\dXY{x}-c'(U)U_Yt_X-c(U)\dXY{t}\\
    \label{eq:vy}
    &=\frac{c'(U)}{2c(U)}\left(U_Yv+U_Xw\right)
  \end{align}
  and
  \begin{align}
    \notag
    w_X&=\dXY{x}+c'(U)U_Xt_Y-c(U)\dXY{t}\\
    \label{eq:wy}
    &=\frac{c'(U)}{2c(U)}\left(U_Yv+U_Xw\right).
  \end{align}
  It follows that
  \begin{align*}
    v(X,Y)&=\int_{\Y(X)}^Y\frac{c'(U)}{2c(U)}\left(U_Yv+U_Xw\right),\\
    w(X,Y)&=\int_{\X(X)}^X\frac{c'(U)}{2c(U)}\left(U_Yv+U_Xw\right),
  \end{align*}
  which implies, after using \eqref{eq:bdxmxpi},
  \eqref{eq:bdymypi} and \eqref{eq:defBM}, that
  \begin{align*}
    \norm{v}_{\LY}&\leq\delta
    2\tnorm{\Theta}_{\G(\Omega)}\max\frac{c'(U)}{2c(U)}
(\norm{v}_{\LY}+\norm{w}_{\LX}),\\
    \norm{w}_{\LX}&\leq\delta 2\tnorm{\Theta}_{\G(\Omega)}\max\frac{c'(U)}{2c(U)}(\norm{v}_{\LY}+\norm{w}_{\LX}).
  \end{align*}
  Hence, by taking $\delta$ smaller if necessary,
  we get $\norm{v}_{\LY}=\norm{w}_{\LX}=0$. Let us
  now introduce $z\in\WY(\Omega)$ as
  $z=2J_Xx_X-\left(c(U)U_X\right)^2$. We
  have
  \begin{align}
    \notag
    z_Y&=2\dXY{J}x_X+2J_X\dXY{x}-2c(U)^2U_X\dXY{U}-2c(U)c'(U)U_Y\left(U_X\right)^2\\
    \label{eq:zy}
    &=\frac{c'(U)}{c(U)}U_Yz
  \end{align}
  and $z(X,\Y(X))=0$ for $X\in[X_l,X_r]$, the
  unique solution to \eqref{eq:zy} is $z=0$. One
  proves in the same way that
  $J_Yx_Y=2\left(c(U)U_Y\right)^2$.  Let
  us now prove \eqref{eq:invbd0}. Since the
  initial data belongs to $\G(\Omega)$, we have
  $\norms{1/(x_X+J_X)(X,\Y(X))}_{L^\infty([X_l,X_r])}\leq
  \tnorm{\Theta}_{\G(\Omega)}$ as
  $1/(x_X+J_X)(X,\Y(X))=1/(\V_2+\W_2)(X)$
  for a.e. $X\in[X_l,X_r{}]$. For all fixed $X$
  such that $1/(x_X+J_X)(X,\Y(X))\leq
  \tnorm{\Theta}_{\G(\Omega)}$, that is for almost
  every $X\in[X_l,X_r]$, we define
  \begin{equation*}
    Y_*=\inf\{Y\in[Y_l,Y_r]\mid Y\leq \Y(X)\text{ and }(x_X+J_X)(X,Y')>0\text{ for all }Y'>Y\}
  \end{equation*}
  and similarly 
  \begin{equation*}
    Y^*=\sup\{Y\in[Y_l,Y_r]\mid Y\geq \Y(X)\text{ and }(x_X+J_X)(X,Y')>0\text{ for all }Y'<Y\}.
  \end{equation*}
  On $(Y_*,Y^*)$, we have $(x_X+J_X)(X,Y)>0$
  and we define
  \begin{equation*}
    q(Y)=\frac{1}{(x_X+J_X)(X,Y)}.
  \end{equation*}
  Let us assume that $Y_*<Y_r$ and therefore, by
  continuity, 
  \begin{equation}
    \label{eq:contblY}
    (x_X+J_X)(X,Y^*)=0.    
  \end{equation}
  On $(Y_*,Y^*)$, we have $q(Y)\geq0$ and, since
  $J_Xx_X\geq0$ by \eqref{eq:energrel2}, it
  implies that 
  \begin{equation}
    \label{eq:posxypr}
    x_X\geq0\text{ and }J_X\geq0.     
  \end{equation}
  By using \eqref{eq:goveq}, we obtain
  \begin{align}
    \notag
    q_Y&=-\frac{\dXY{x}+\dXY{J}}{(x_X+J_X)^2}\\
    \label{eq:vynob}
    &=-\frac{c'(U)}{2c(U)}\frac{U_Y(x_X+J_X)+(J_Y+x_Y)U_X}{(x_X+J_X)^2}.
  \end{align}
  From \eqref{eq:energrel2}, we infer that
  \begin{equation*}
    \abs{U_X}=\frac1{c(U)\sqrt{2}}\sqrt{J_Xx_X}\leq\frac1{2c(U)\sqrt{2}}(J_X+x_X).
  \end{equation*}
  Hence, \eqref{eq:vynob} yields
  \begin{equation*}
    q_Y\leq\frac{\abs{c'(U)}}{2c(U)}(\abs{U_Y}+\abs{J_Y}+\abs{x_Y})q\leq Cq
  \end{equation*}
  for some constant $C$ which depends only 
  \begin{equation}
    \label{eq:CdepesU}
    \esssup_{Y\in[Y_l,Y_r]}(\abs{U}(X,Y)+\abs{Z_Y}(X,Y)),
  \end{equation}
  that is $\tnorm{\Theta}_{\G(\Omega)}$, by
  \eqref{eq:defBM}. By Gronwall's lemma, it
  follows that $q$ cannot blow up in $Y^*$,
  contradicting \eqref{eq:contblY} and
  \begin{equation}
    \label{eq:bdforq}
    \frac1{x_X+J_X}(X,Y)\leq \frac1{\V_3+\V_4}(X)e^{C\abs{Y-\Y(X)}}
  \end{equation}
  for a.e. $X\in[X_l,X_r]$ and all
  $Y\in[Y_l,Y_r]$. In the same way, one proves
  that $Y_*=Y_l$. Hence, we have proved that, for
  almost every $X\in[X_l,X_r]$,
  \begin{equation*}
    x_X(X,Y)\geq0\text{ and }J_X(X,Y)\geq0\text{ for all }Y
  \end{equation*}
  and
  $\norm{1/(x_X+J_X)(X,Y)}_{W^\infty(Y_l,Y_r)}$
  is bounded, for almost every $X\in[X_l,X_r]$,
  by a constant which is independent of $X$ and
  therefore $1/(x_X+J_X)\in\WY(\Omega)$.
  This concludes the proof of \eqref{eq:posxX} and
  the first identity \eqref{eq:invbd0} while
  \eqref{eq:posxY} and the second identity in
  \eqref{eq:invbd0} can be proven in the same way.
\end{proof}

\subsection{A priori estimates}

Given a positive constant $L$, we call
 domains of the type
\begin{equation*}
  D=\{(X,Y)\in\Real^2\mid \abs{Y-X}<2L\}
\end{equation*}
for \textit{strip domain}. Strip domains are
correspond to domains where time is bounded. We
have the following a priori estimates for the
solution of \eqref{eq:goveq}. The energy $J(X,Y)$
is bounded in the whole plane while $Z^a$ (that is,
$Z$, up to a shift in the second component) and
its derivatives are bounded in every strip domain.
\begin{lemma}
  \label{lem:aprioribd}
  Given $\Omega=[X_l,X_r]\times[Y_l,Y_r]$ and
  $\Theta=(\X,\Y,\Z,\V,\W)\in\G(\Omega)$, let
  $Z\in\H(\Omega)$ be a solution to
  \eqref{eq:goveq} such that
  $\Theta=Z\bullet(\X,\Y)$. Let
  $\E_0=\norm{\Z_4}_{L^\infty([s_l,s_r])}+\norm{\Z_5}_{L^\infty([s_l,s_r])}$. Then
  the following statements hold:
  \begin{enumerate}
  \item[(i)] Global boundedness of the energy, more precisely,
    \begin{equation}
      \label{eq:bdJK}
      0\leq J(X,Y)\leq\E_0\text{ for all }(X,Y)\in\Omega\quad\text{ and }\quad\norm{K}_{L^\infty(\Omega)}\leq(1+\kappa)\E_0
    \end{equation}
    where $J=Z_4$ and $K=Z_5$.
  \item[(ii)] The function $Z$ and its derivatives
    remain uniformly bounded in strip
    domains. More precisely there exists a
    nondecreasing function
    $C_1=C_1(L,\tnorm{\Theta}_{\G(\Omega)})$, such
    that, for any $L>0$ and any $X$ and $Y$ such
    that $\abs{X-Y}\leq 2L$, we have
    \begin{equation}
      \label{eq:bdtxu}
      \abs{Z^a(X,Y)}\leq C_1,\quad \abs{Z_X(X,Y)}\leq C_1,\quad \abs{Z_Y(X,Y)}\leq C_1
    \end{equation}
    and
    \begin{equation}
      \label{eq:invbd}
      \frac1{x_X+J_X}(X,Y)\leq C_1,
\quad \frac1{x_Y+J_Y}(X,Y)\leq C_1.
    \end{equation}
  \end{enumerate}
  The condition (ii) above is equivalent to the following condition
  (iii):
  \begin{enumerate}
  \item[(iii)] For any curve
    $(\bar\X,\bar\Y)\in\C(\Omega)$, we have
    \begin{equation}
      \label{eq:bdexbuk}
      \tnorm{Z\bullet(\bar\X,\bar\Y)}_{\G(\Omega)}\leq C_1
    \end{equation}
\end{enumerate}
    where $C_1=
    C_1(\norms{(\bar\X,\bar\Y)}_{\C(\Omega)},\tnorm{\Theta}_{\G(\Omega)})$
    is a given function which is increasing with
    respect to both its arguments.
\end{lemma}
The inequalities in \eqref{eq:bdtxu} hold in fact
in $\Linf(\Omega)$, $\WY(\Omega)$ and
$\WX(\Omega)$, respectively. The inequalities in
\eqref{eq:invbd} hold in $\WY(\Omega)$ and
$\WX(\Omega)$, respectively.
\begin{proof}
  Given $P=(X,Y)\in\Omega$, let $s_0=\Y^{-1}(Y)$
  and $s_1=\X^{-1}(X)$. We denote
  $P_0=(X(s_0),Y(s_0))$ and
  $P_1=(X(s_1),Y(s_1))$. We assume that $s_0\leq
  s_1$ (the proof for the other case is very
  similar). Since $\dot\X$ and $\dot\Y$ are
  positive, it implies that $X=\X(s_1)\geq\X(s_0)$
  and $Y=Y(s_0)\leq\Y(s_1)$. Then, because
  $J_X\geq0$, $J_Y\geq0$ and $\Z_4\geq0$, we
  have
  \begin{equation}
    \label{eq:bdJK1}
    0\leq\Z_4(s_0)=J(P_0)\leq J(P)\leq J(P_1)=\Z_4(s_1)\leq\E_0
  \end{equation}
  which gives the first inequality in
  \eqref{eq:bdJK}. By \eqref{eq:reltJK}, we get
  \begin{equation*}
    \abs{K(P)-K(P_0)}\leq\kappa(J(P)-J(P_0))
  \end{equation*}
  which implies
  \begin{equation}
    \label{eq:bdabsK}
    \abs{K(P)}\leq\abs{K(P_0)}+\kappa(J(P)-J(P_0))
  \end{equation}
  and
  \begin{equation*}
    \abs{K(P)}\leq(1+\kappa)\E_0
  \end{equation*}
  by \eqref{eq:bdJK}. Since $x_X\geq0$, we have
  \begin{equation}
    \label{eq:bdxphim2}
    x(P)\geq x(P_0)=\Z_2(s_0)\geq -\tnorm{\Theta}_{\G(\Omega)}+s_0.
  \end{equation}
  Since $\frac12(X+Y)=Y+\frac12(X-Y)\leq Y(s_0)+L$,
  it follows that
  \begin{equation*}
    x(P)-\frac12(X+Y)\geq-\tnorm{\Theta}_{\G(\Omega)}+s_0-\Y(s_0)-L\geq-2\tnorm{\Theta}_{\G(\Omega)}-L.
  \end{equation*}
  Similarly, using that $x_Y\geq0$, we get
  \begin{equation}
    \label{eq:bdxphim3}
    x(P)\leq x(P_1)=\Z_2(s_1)\leq \tnorm{\Theta}_{\G(\Omega)}+s_1.
  \end{equation}
  and
  \begin{equation}
    \label{eq:bdxphim}
    x(P)-\frac12(X+Y)\leq 2\tnorm{\Theta}_{\G(\Omega)}+L.
  \end{equation}
  Hence, $\abs{x(P)-\frac12(X+Y)}\leq 2\tnorm{\Theta}_{\G(\Omega)}+L$. We
  have
  \begin{align}
    \label{eq:abst}
    \abs{t(P)}=\abs{\int_{Y}^{\Y(s_1)}t_Y(X,\tilde
      Y)\,d\tilde
      Y}\leq\int_{Y}^{\Y(s_1)}\frac{x_Y}{c(U)}\,d\tilde
    Y\leq\kappa(x(P_1)-x(P)).
  \end{align}
  Since
  \begin{equation}
    \label{eq:xp1xp0bd}
    x(P_1)-x(P)\leq x(P_1)-x(P_0)=\Z_2(s_1)-\Z_2(s_0)\leq 2\tnorm{\Theta}_{\G(\Omega)}+s_1-s_0
  \end{equation}
  and 
  \begin{equation}
    \label{eq:diffsbd}
    s_1-s_0=s_1-\Y(s_1)-(s_0-\X(s_0))+Y-X\leq 2\tnorm{\Theta}_{\G(\Omega)}+2L,
  \end{equation}
  it follows from \eqref{eq:abst} that
  $\abs{t(P)}\leq\kappa(4\tnorm{\Theta}_{\G(\Omega)}+2L)$. We have
  \begin{equation*}
    \abs{U(P)}\leq\abs{U(P_1)}+\int_{Y}^{\Y(s_1)}U_Y\,d\tilde Y.
  \end{equation*}
  By \eqref{eq:energrel2}, we have that
  \begin{equation}
    \label{eq:bduy}
    \abs{U_Y}\leq\frac{\kappa}{2\sqrt2}(J_Y+x_Y).
  \end{equation}
  Hence,
  \begin{align}
    \notag
    \abs{U(P)}&\leq\abs{U(P_1)}+\frac{\kappa}{2\sqrt2}\abs{J(P_1)+x(P_1)-J(P)-x(P)}\\
     \label{eq:bdUstrip}
     &\leq
     \tnorm{\Theta}_{\G(\Omega)}+\frac{\kappa}{2\sqrt2}(\E_0
     +4\tnorm{\Theta}_{\G(\Omega)}+2L)\text{ by
       \eqref{eq:xp1xp0bd} and \eqref{eq:bdJK1}}.
  \end{align}
  To prove that $Z_X$ and $Z_Y$ remain bounded, we
  use the bi-linearity embedded in the governing
  equation \eqref{eq:condgoveq}. We use first the
  linearity of $F(Z)$ with respect to the first
  variable and for almost every $X\in[X_l,X_r]$,
  we get, after applying Gronwall's lemma, that
  \begin{align}
    \notag \abs{Z_X(X,Y)}&\leq
    \abs{Z_X(X,\Y(X))}\exp(\int_{Y}^{\Y^{-1}(X)}\abs{F(Z)(\dott,Z_Y)}\,d\tilde Y)\\
    \label{eq:gronV}
    &=
    \abs{\V(X)}\exp(\int_{Y}^{\Y^{-1}(X)}\abs{F(Z)(\dott,Z_Y)}\,d\tilde
    Y).
  \end{align}
  Here $F(Z)(\cdot,W)$ denotes the matrix
  $V\mapsto F(Z)(V,W)$ and we use any matrix norm
  as they are all equivalent. We also assume (as
  earlier) that $Y\leq\Y^{-1}(X) $ (otherwise we
  have to interchange the bounds in the integral)
  and we denote $P_1=(X,\Y^{-1}(X))$. We have
  $\abs{\V(X)}\leq \tnorm{\Theta}_{\G(\Omega)}$. After using
  \eqref{eq:reltx}, \eqref{eq:reltJK} and
  \eqref{eq:bduy}, we obtain that
  \begin{equation*}
    \abs{F(Z)(\cdot,Z_Y)}\leq C(\abs{t_Y}+\abs{x_Y}+\abs{U_Y}+\abs{J_Y}+\abs{K_Y})
  \end{equation*}
  and we have used here the linearity of $F(Z)$
  with respect to its second variable. Hence,
  \begin{align*}
    \abs{F(Z)(\cdot,Z_Y)}&= C\left(\frac1c(x_Y+J_Y)+x_Y+J_Y+\abs{U_Y}\right)\\
    &\leq C(x_Y+J_Y)
  \end{align*}
  for a constant $C$ that depends only on $c(U)$
  and therefore only on
  $\tnorm{\Theta}_{\G(\Omega)}$ and $L$,
  by \eqref{eq:bdUstrip}. Hence,
  \begin{align*}
    \int_Y^{\Y^{-1}(X)}\abs{F(Z)(\dott,Z_Y)}\,d\tilde Y&\leq C\int_Y^{\Y^{-1}(X)}\left(x_Y+J_Y\right)\,d\tilde Y\\
    &=C(x(P_1)-x(P)+J(P_1)-J(P))\\
    &\leq C(\E_0+4\tnorm{\Theta}_{\G(\Omega)}+2L),
  \end{align*}
  by \eqref{eq:bdJK1}, \eqref{eq:xp1xp0bd} and
  \eqref{eq:diffsbd}. Combined with
  \eqref{eq:gronV}, it yields
  \begin{equation*}
    \abs{Z_X(X,Y)}\leq C
  \end{equation*}
  for some constant $C$ which depends only on $L$
  and $\tnorm{\Theta}_{\G(\Omega)}$. Similarly, one proves
   the bound on $Z_Y$. The estimate
  \eqref{eq:invbd} follows from the estimate
  \eqref{eq:bdforq} in the proof of Theorem
  \ref{th:shortrange} as the constant $C$ in
  \eqref{eq:bdforq} depends only on $L$ and
  $\tnorm{\Theta}_{\G(\Omega)}$, by
  \eqref{eq:bdtxu}, and
  \begin{equation*}
    \abs{Y-\Y(X)}=\abs{Y-X+\X(\X^{-1}(X))-\Y(\X^{-1}(X))}\leq L+\tnorm{\Theta}_{\G(\Omega)}.
  \end{equation*}
\end{proof}

\subsection{Global existence}

We obtain the following existence and uniqueness
lemma.

\begin{lemma}[Existence and uniqueness on
  arbitrarily large rectangles]
  \label{lem:globalsol}
  Given a rectangular domain
  $\Omega=[X_l,X_r]\times[Y_l,Y_r]$ and
  $\Theta=(\X,\Y,\Z,\V,\W)$ in $\G(\Omega)$, there
  exists a unique $Z\in\H(\Omega)$ such that
  \begin{equation*}
    \Theta=Z\bullet(\X,\Y).
  \end{equation*}
\end{lemma}

\begin{proof}
  \begin{figure}
    \centering
    \includegraphics{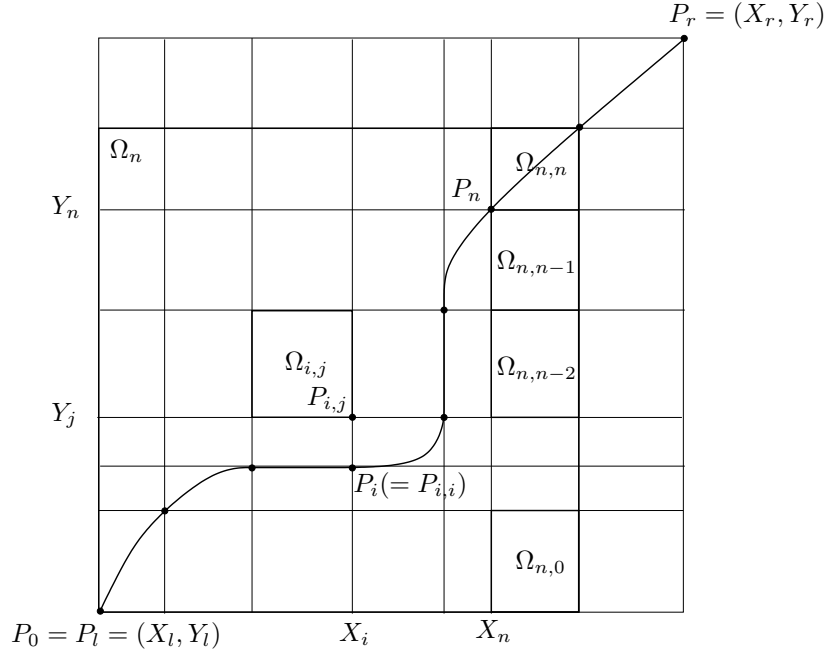}
    \caption{Construction of the global solution.}
    \label{fig:globalgrid}
  \end{figure} 
  Let $N$ denote an integer that we will set later
  and $\delta=\frac{s_r-s_l}N$. For
  $i=0,\ldots,N$, let $s_i=i\delta+s_l$ and we
  consider the sequence of points
  $P_i=(X_i,Y_i)=(X(s_i),Y(s_i))$. For
  $i,j=0,\ldots,N$, we construct a grid which
  consists of the points
  $P_{i,j}=(X_{i,j},Y_{i,j})$ where $X_{i,j}=X_i$
  and $Y_{i,j}=Y_j$, see Figure
  \ref{fig:globalgrid}. We denote by by
  $\Omega_{i,j}$ the rectangle with diagonal
  points $P_{i,j}$ and $P_{i+1,j+1}$. Let
  $\Omega_n$ denote the rectangle with diagonal
  points given by $(X_0,Y_0)$ and $(X_n,Y_n)$. We
  prove by induction that there exists a unique
  $Z\in\H(\Omega_n)$ such $\Theta=Z\bullet(X,Y)$
  (Here we use the same notation for
  $\Theta\in\G(\Omega)$ and $(\X,\Y)\in\C(\Omega)$
  and their restriction to $\Omega_n$ which belong
  to $\G(\Omega_n)$ and $\C(\Omega_n)$,
  respectively). On $\Omega_1$, we can choose $N$
  large enough and depending only on
  $\tnorm{\Theta}_{\G(\Omega_1)}\leq\tnorm{\Theta}_{\G(\Omega)}$
  such that
  \begin{equation*}
    s_1-s_0\leq\delta\leq
    C(\tnorm{\Theta}_{\G(\Omega)})^{-1}\leq
    C(\tnorm{\Theta}_{\G(\Omega_1)})^{-1},
  \end{equation*}
  and, by Theorem \ref{th:shortrange}, there
  exists a unique solution $Z\in\H(\Omega_1)$ such
  that $\Theta=Z\bullet(X,Y)$.  We assume that
  there exists a unique solution $Z\in\Omega_n$
  and prove that there exists a solution on
  $\Omega_{n+1}$. On $\Omega_{n,n}$, we get the
  existence of a unique solution by Theorem
  \ref{th:shortrange} as
  \begin{equation*}
    s_{n+1}-s_n\leq\delta\leq
    C(\tnorm{\Theta}_{\G(\Omega)})^{-1}\leq
    C(\tnorm{\Theta}_{\G(\Omega_{n,n})})^{-1}.    
  \end{equation*}
  For $j=n-1,\ldots,0$, we construct iteratively
  the unique solution in $\Omega_{n,j}$ and
  $\Omega_{j,n}$ as follows. We treat only the
  case of $\Omega_{n,j}$. We assume the solution
  is known on $\Omega_{n,j+1}$, then we define
  $\tilde\Theta=(\tilde \X,\tilde \Y,\tilde
  \Z,\tilde \V,\tilde \W)\in\G(\Omega_{n,j})$ as
  follows: The curve $\tilde \X(s),\tilde \Y(s)$
  is given by
  \begin{equation*}
    \tilde \X(s)=X_n,\quad \tilde\Y(s)=2s-X_n
  \end{equation*}
 for $\frac12(Y_j+X_n)\leq s\leq\frac12(Y_{j+1}+X_n)$,
 \begin{equation*}
   \tilde \X(s)=2s-Y_{j+1},\quad
   \tilde\Y(s)=Y_{j+1}       
 \end{equation*}
 for $\frac12(Y_{j+1}+X_n)\leq s\leq
 \frac12(Y_{j+1}+X_{n+1})$ and set
 \begin{align*}
   \tilde \Z(s)&=Z(\tilde \X(s),\tilde \Y(s))
   \text{ for $s\in[\frac12(Y_{j}+X_n),\frac12(Y_{j+1}+X_{n+1})]$},\\
   \tilde\V(X)&=Z_X(X,Y_{j+1})\text{   for a.e. $X\in[X_n,X_{n+1}]$},\\
   \tilde\W(Y)&=Z_Y(X_n,Y) \text{ for
     a.e. $X\in[Y_j,Y_{j+1}]$}.
 \end{align*}
 Using Lemma \ref{lem:aprioribd}, we can check that
 $\tnorm{\tilde\Theta}_{\G(\Omega_{n,j})}$ is
 bounded by a constant $C_2$ that depends only on
 $L$ and $\tnorm{\Theta}_{\G(\Omega)}$. We have
  \begin{align*}
    \frac12(Y_{j+1}+X_{n+1}-Y_{j}-X_n)=\frac12(\Y(s_{j+1})+\X(s_{n+1})-\Y(s_{j})-\X(s_{n}))\leq 2\delta.
  \end{align*}
  Here we have used that $\X$ and $\Y$ are
  Lipschitz with Lipschitz constant smaller than
  $2$. By taking $N$ large enough so that
  $2\delta$ is smaller that $C(C_2)^{-1}$, we can
  apply Theorem \ref{th:shortrange} to
  $\Omega_{n,j}$ and obtain the existence of a
  unique solution in $\H(\Omega_{n,j})$. Similarly
  we get the existence of a unique solution in
  $\H(\Omega_{j,n})$. Since
  \begin{equation*}
    \Omega_{n+1}=\Omega_{n}\cup(\cup_{j=0}^n\Omega_{j,n})\cup(\cup_{j=0}^n\Omega_{n,j}),
  \end{equation*}
  we have proved the existence of a unique
  solution in $\Omega_{n+1}$.
\end{proof}

In Lemma \ref{lem:aprioribd}, we establish
$L^\infty$-bounds on the derivatives on a strip
domain. It turns out that we can also establish
$L^2$-bounds on the derivatives as stated in the
next lemma. In this context, by $L^2$-bounds, we
mean that we can bound the integrals of the
differential forms $(Z_X^a)^2\,dX$ and
$(Z_Y^a)^2\,dY$ along a curve in $\C$. It is
useful to have in mind that, for any given
time $T$, we can find a curve $(\X,\Y)\in\C$ which
corresponds to this given time $T$, that is,
$t(\X(s),\Y(s))=T$ for all $s\in\Real$. Thus the
$L^2$-bound we now establish is fundamental to
obtain $L^2$-bounds in the initial set of
coordinates.
 
\begin{lemma}[A Gronwall lemma for
  curves]\label{lem:gron} Given $\Omega=[X_l,X_r]\times[Y_l,Y_r]$,
  $Z\in\H(\Omega)$ and
  $(\X,\Y)\in\C(\Omega)$. Then, for any
  $(\bar\X,\bar\Y)\in\C(\Omega)$,
  \begin{equation}
    \label{eq:L2bd}
    \norms{Z\bullet(\bar\X,\bar\Y)}_{\G(\Omega)}\leq C\norms{Z\bullet(\X,\Y)}_{\G(\Omega)}
  \end{equation}
  where
  $C=C(\norms{(\bar\X,\bar\Y)}_{\C(\Omega)},\tnorm{Z\bullet(\X,\Y)}_{\G(\Omega)})$
  is a given increasing function with respect to
  both its arguments.
\end{lemma}

\begin{proof} 
  Note that, for any function in $f\in\WY(\Omega)$
  (and respectively $g\in\WX(\Omega)$), the forms
  $f(X,Y)\,dX$ (respectively $g(X,Y)\,dY$) are
  well defined while the forms $f(X,Y)\,dY$
  (respectively $g(X,Y)\,dX$) are not. Given
  $Z\in\H(\Omega)$, we can consider the forms
  $U^2\,dX$, $U^2\,dY$, $(Z^a_X)^2\,dX$ and
  $(Z^a_Y)^2\,dY$. For any curve
  $\bar\Gamma=(\bar\X,\bar\Y)\in\C(\Omega)$,
  $(\bar\X,\bar\Y,\bar\Z,\bar\V,\bar\W)=\Theta\in\G(\Omega)$,
  we have (by definition of the integral of a form
  along a curve and the definition of
  $Z\bullet(\bar\X,\bar\Y)$)
  \begin{equation*}
    \int_{\bar\Gamma}(U^2\,dX+U^2\,dY)=2\int_{s_l}^{s_r}\bar\Z_3^2(s)\,ds\quad\text{ (as $\X+\Y=2s$)},
  \end{equation*}
  and
  \begin{equation*}
    \int_{\bar\Gamma}(Z^a_X)^2\,dX=\int_{X_l}^{X_r}\bar\V^a(X)^2dX,\quad\int_{\bar\Gamma}(Z^a_Y)^2\,dY=\int_{Y_l}^{Y_r}\bar\W^a(Y)^2dY.
  \end{equation*}
  We can rewrite
  \begin{equation*}
    \norm{Z\bullet(\bar X,\bar Y)}_{\G(\Omega)}^2=\int_{\bar\Gamma}\left(\frac12U^2\,(dX+dY)+(Z^a_X)^2\,dX+(Z^a_Y)^2\,dY\right).
  \end{equation*}
  Thus, we want to prove that
  \begin{multline}
    \label{eq:L2bdform}
    \int_{\bar\Gamma}\left(\frac12U^2\,(dX+dY)+(Z^a_X)^2\,dX+(Z^a_Y)^2\,dY\right)
    \\\leq
    C\int_{\Gamma}\left(\frac12U^2\,(dX+dY)+(Z^a_X)^2\,dX+(Z^a_Y)^2\,dY\right).
  \end{multline}
  We decompose the proof into three steps.\\[5mm]
  \textbf{Step 1.} We first prove that \eqref{eq:L2bdform} holds for small
  domains. We claim that there exist constants $\delta$ and $C$, which
  depend uniquely on $\norm{(\bar\X,\bar\Y)}_{\C(\Omega)}$ and
  $\tnorm{Z\bullet(\X,\Y)}_{\G(\Omega)}$ such that, for any
  rectangular domain $\Omega=[X_l,X_r]\times[Y_l,Y_r]$ with
  $s_r-s_l\leq\delta$, \eqref{eq:L2bdform} holds.  We denote by $C$ a
  generic increasing function of $\norm{(\bar\X,\bar\Y)}_{\C(\Omega)}$
  and $\tnorm{Z\bullet(\X,\Y)}_{\G(\Omega)}$. By Lemma
  \ref{lem:aprioribd}, we have
  \begin{equation*}
    \norm{U}_{L^\infty(\Omega)}+\norm{Z^a_X}_{L^\infty(\Omega)}+\norm{Z^a_Y}_{L^\infty(\Omega)}\leq C.
  \end{equation*}
  Let
  \begin{equation*}
    A=\sup_{\bar\Gamma}\int_{\Gamma}\left(\frac12U^2\,(dX+dY)+(Z^a_X)^2\,dX+(Z^a_Y)^2\,dY\right)
  \end{equation*}
  where the supremum is taken over all
  $\bar\Gamma=(\bar\X,\bar\Y)\in\C(\Omega)$. Since
  $Z$ is a solution of \eqref{eq:goveq}, we have
  for a.e. $X\in[X_l,X_r]$ and all
  $Y\in[Y_l,Y_r]$, that
  \begin{align*}
    (x_X-\frac12)^2(X,Y)&=(x_X-\frac12)^2(X,\Y(X))+\int_{\Y(X)}^{Y}(x_X-\frac12)x_{XY}\,d\bar Y\\
    &=(x_X-\frac12)^2(X,\Y(X))\\
    &\quad+\int_{\Y(X)}^{Y}\frac{c'(U)}{c(U)}\big((x_X-\frac12)^2(U_X+U_Y)+U_X(x_X-\frac12)+U_Y(x_X-\frac12)\big)\,d\bar Y\\
  \end{align*}
  and
  \begin{equation}
    \label{eq:bdcurv}
    \int_{\bar\Gamma}(x_X-\frac12)^2\,dX\leq\int_{\Gamma}(x_X-\frac12)^2\,dX+C\int_{X_l}^{X_r}\int_{Y_l}^{Y_r} ((Z^a_X)^2+(Z^a_Y)^2)dXdY.
  \end{equation}
  For any $Y\in[Y_l,Y_r]$, the
  integral $\int_{X_l}^{
    X_r}(Z^a_X)^2(X,Y)\,dX$ can be seen as the
  integral of the form $(Z^a_X)^2\,dX$ on the
  piecewise linear path $\Gamma$ going through the
  points $(X_l,Y_l)$, $(X_l,Y)$, $(X_r,Y)$,
  $(X_r,Y_r)$, so that
  $\int_{X_l}^{X_r}(Z^a_X)^2(X,Y)\,dX\leq
  A$. Similarly,
  $\int_{Y_l}^{Y_r}(Z^a_X)^2(X,Y)\,dY\leq A$, for
  any $X\in[X_l,X_r]$. Hence, \eqref{eq:bdcurv}
  yields
  \begin{align*}
    \int_{\bar\Gamma}(x_X-\frac12)^2\,dX&\leq\int_{\Gamma}(x_X-\frac12)^2\,dX+C(Y_r-Y_l+X_r-X_l)A\\
    &\leq\int_{\Gamma}(x_X-\frac12)^2\,dX+C\delta A
  \end{align*}
  as $(Y_r-Y_l+X_r-X_l)=s_r-s_l$.  By treating similarly the other
  components of $Z^a_X$ and $Z^a_Y$, we get
  \begin{equation}
    \label{eq:intZest}
    \int_{\bar\Gamma}(Z^a_X)^2\,dX\leq\int_{\Gamma}(Z^a_X)^2\,dX+6C\delta A
    \quad\text{ and }\quad\int_{\bar\Gamma}(Z^a_Y)^2\,dY\leq\int_{\Gamma}(Z^a_Y)^2\,dY+6C\delta A.
  \end{equation}
  For the component $U$, we have
  \begin{align*}
    U^2(X,Y)&=U^2(X,\Y(X))+\int_{\Y(X)}^{Y}2UU_Y\,d\bar Y\\
    &\leq
    U^2(X,\Y(X))+\int_{\Y(X)}^YU^2\,d\bar
    Y+\int_{\Y(X)}^YU_Y^2\,d\bar Y
  \end{align*}
  and it follows, as before, that
  \begin{equation}
    \label{eq:intUdXest}
    \int_{\bar\Gamma}U^2\,dX\leq\int_{\Gamma}U^2\,dX+C\delta A.
  \end{equation}
  Similarly, we obtain
  \begin{equation}
    \label{eq:intUdYest}
    \int_{\bar\Gamma}U^2\,dY\leq\int_{\Gamma}U^2\,dY+C\delta A.
  \end{equation}
  After adding \eqref{eq:intZest},
  \eqref{eq:intUdXest}, \eqref{eq:intUdYest} and
  recalling that $ds=\frac12(dX+dY)$, we obtain
  \begin{multline*}
    \int_{\bar\Gamma}(\frac12U^2\,(dX+dY)+(Z^a_X)^2\,dX+(Z^a_Y)^2\,dY)\\\leq\int_{\Gamma}(\frac12U^2\,(dX+dY)+(Z^a_X)^2\,dX+(Z^a_Y)^2\,dY)+13C\delta A
  \end{multline*}
  which yields, after taking the supremum over all
  curves $\bar\Gamma$,
  \begin{equation*}
    (1-13C\delta)A\leq\int_{\Gamma}(\frac12U^2\,(dX+dY)+(Z^a_X)^2\,dX+(Z^a_Y)^2\,dY)
  \end{equation*}
  and \eqref{eq:L2bdform} follows.\\[5mm]
  \textbf{Step 2.}  For an arbitrarily large
  rectangular domain
  $\Omega=[X_l,X_r]\times[Y_l,Y_r]$, let us prove
  that \eqref{eq:L2bdform} holds for the curves
  $\bar\Gamma=(\bar\X,\bar\Y)\in\C(\Omega)$ such
  that
  \begin{equation*}
    \bar\Y(s)-\bar\X(s)>\Y(s)-\X(s)\text{ for all }s\in(s_l,s_r),
  \end{equation*}
  that is, the curve $\bar\Gamma$ is above
  $\Gamma$ and intersects $\Gamma$ only at the end
  points.  Similarly one proves that
  \eqref{eq:L2bdform} holds for curves
  $\bar\Gamma=(\bar\X,\bar\Y)\in\C(\Omega)$ such
  that $\bar\Y(s)-\bar\X(s)<\Y(s)-\X(s)$ for
  all $s\in[s_l,s_r]$. For a constant $K>0$ that
  we will determine later, we have for
  a.e. $X\in[X_l,X_r]$, that
  \begin{multline*}
    e^{-K(\bar\Y(X)-X)}(x_X-\frac12)^2(X,\bar\Y(X))-e^{-K(\Y(X)-X)}(x_X-\frac12)^2(X,\Y(X))\\
    =\int_{\Y(X)}^{\bar\Y(X)}-Ke^{-K(Y-X)}(x_X-\frac12)^2\,dY+\int_{\Y(X)}^{\bar\Y(X)}e^{-K(X-Y)}(x_X-\frac12)x_{XY}\,dY
  \end{multline*}
  which implies, since $Z$ is solution and by the
  estimates of Lemma \ref{lem:aprioribd}, that
  \begin{multline}
    \label{eq:formintxX}
    \int_{\bar\Gamma}e^{-K(Y-X)}(x_X-\frac12)^2\,dX-\int_{\Gamma}e^{-K(Y-X)}(x_X-\frac12)^2\,dX\\
    \leq\int_{X_l}^{X_r}\int_{\Y(X)}^{\bar\Y(X)}-Ke^{-K(Y-X)}(x_X-\frac12)^2\,dXdY\\
+C\int_{X_l}^{X_r}\int_{\Y(X)}^{\bar\Y(X)}e^{-K(Y-X)}((Z^a_X)^2+(Z^a_Y)^2)\,dXdY.
  \end{multline}
  Note that \eqref{eq:formintxX} corresponds to an
  application of Stokes's theorem to the domain
  bounded by the curves $\Gamma$ and
  $\bar\Gamma$. We treat in the same way each
  component of $Z^a_X$ and obtain that
  \begin{multline}
    \label{eq:xixsqest}
    \int_{\bar\Gamma}e^{-K(Y-X)}(Z^a_X)^2\,dX-\int_{\Gamma}e^{-K(Y-X)}(Z^a_X)^2\,dX\\
    \leq\int_{X_l}^{X_r}\int_{\Y(X)}^{\bar\Y(X)}-Ke^{-K(Y-X)}(Z^a_X)^2\,dXdY\\+C\int_{X_l}^{X_r}\int_{\Y(X)}^{\bar\Y(X)}e^{-K(Y-X)}((Z^a_X)^2+(Z^a_Y)^2)\,dXdY.
  \end{multline}
  As far as $Z^a_Y$ is concerned, we get
  \begin{multline}
    \label{eq:xiysqest}
    \int_{\bar\Gamma}e^{-K(Y-X)}(Z^a_Y)^2\,dY-\int_{\Gamma}e^{-K(Y-X)}(Z^a_Y)^2\,dY\\
    \leq-\int_{Y_l}^{Y_r}\int_{\bar\X(Y)}^{\X(Y)}Ke^{-K(Y-X)}(Z^a_X)^2\,dXdY\\+C\int_{Y_l}^{Y_r}\int_{\bar\X(Y)}^{\X(Y)}e^{-K(Y-X)}((Z^a_X)^2+(Z^a_Y)^2)\,dXdY.
  \end{multline}
  Let us prove that the sets 
  \begin{equation*}
    \N_1=\left\{
      \begin{aligned}
        &X_l< X< X_r\\
        &\Y(X)<Y<\bar\Y(X)
      \end{aligned}
    \right.
    \quad\text{ and }\quad
    \N_2=\left\{
      \begin{aligned}
        &Y_l< Y< Y_r\\
        &\bar\X(Y)<X<\X(Y)
      \end{aligned}
    \right.
  \end{equation*}
  are equal up to a set of zero measure. Let us consider
  $(X,Y)\in\N_1$. We set $s_1=\X^{-1}(X)$, $s_2=\Y^{-1}(Y)$,
  $s_3=\bar\Y^{-1}(Y)$ and $s_4=\bar\X^{-1}(X)$. Since
  $\Y(X)<Y<\bar\Y(X)$, we get
  \begin{equation*}
    \Y(X)=\Y(s_1)<\Y(s_2)=\bar\Y(s_3)=Y<\bar\Y(s_4)=\bar\Y(X).
  \end{equation*}
  Hence, $s_1<s_2$ and $s_3<s_4$, which implies
  \begin{equation*}
    X=\X(s_1)\leq\X(s_2)=\X(Y),\quad \bar\X(Y)=\bar\X(s_3)\leq \bar\X(s_4)=X
  \end{equation*}
  and therefore $\bar\X(Y)\leq X\leq\X(Y)$. Thus
  we have prove that $\N_1\subset\N_2$ up to a set
  of zero measure. Similarly, one proves the
  reverse inclusion.  Hence, by adding
  \eqref{eq:xixsqest} and \eqref{eq:xiysqest}, we
  get
  \begin{multline}
    \int_{\bar\Gamma}e^{-K(Y-X)}((Z^a_X)^2\,dX+(Z^a_Y)^2\,dY)-\int_{\Gamma}e^{-K(Y-X)}((Z^a_X)^2\,dX+(Z^a_Y)^2\,dY)\\
    \leq-K\int_{\N_1}e^{-K(Y-X)}((Z^a_X)^2+(Z^a_Y)^2)\,dXdY+C\int_{\N_1}e^{-K(Y-X)}((Z^a_X)^2+(Z^a_Y)^2)\,dXdY.
  \end{multline}

As far as $U$ is concerned, we proceed in
  the same way and get
  \begin{align}
    \notag
    \int_{\bar\Gamma}e^{-K(Y-X)}U^2\,dX&-\int_{\Gamma}e^{-K(Y-X)}U^2\,dX\\
    \notag
    &=\int_{X_l}^{X_r}\int_{\Y(X)}^{\bar\Y(X)}e^{-K(Y-X)}(-KU^2+2UU_Y)\,dXdY\\
    \label{eq:usqestx}
    &\leq\int_{\N_1}e^{-K(Y-X)}(-KU^2+U^2+U_Y^2)\,dXdY
  \end{align}
  and
  \begin{align}
    \notag
    \int_{\bar\Gamma}e^{-K(Y-X)}U^2\,dY&-\int_{\Gamma}e^{-K(Y-X)}U^2\,dY\\
    \label{eq:usqesty}
    &\leq\int_{\N_2}e^{-K(Y-X)}(-KU^2+U^2+U_X^2)\,dXdY.
  \end{align}
  Combining \eqref{eq:xixsqest}, \eqref{eq:xiysqest},
  \eqref{eq:usqestx}, \eqref{eq:usqestx}, we get
  \begin{multline}
    \label{eq:xixysqest}
    \int_{\bar\Gamma}e^{-K(Y-X)}(\frac12U^2\,(dX+dY)+(Z^a_X)^2\,dX+(Z^a_Y)^2\,dY)\\-\int_{\Gamma}e^{-K(Y-X)}(\frac12U^2\,(dX+dY)+(Z^a_X)^2\,dX+(Z^a_Y)^2\,dY)\\
    \leq\int_{\N_1}(C-K)e^{-K(Y-X)}(U^2+(Z^a_X)^2+(Z^a_Y)^2)\,dXdY.
  \end{multline}
  We choose $K$ sufficiently large so that the right-hand side in
  \eqref{eq:xixysqest} is negative and we obtain that
  \begin{multline*}
    e^{-K\norm{\bar\X-\bar\Y}_{L^\infty}}\int_{\bar\Gamma}(\frac12U^2\,(dX+dY)+(Z^a_X)^2\,dX+(Z^a_Y)^2\,dY)\\
    \leq
    e^{K\norm{\X-\Y}_{L^\infty}}\int_{\Gamma}(\frac12U^2\,(dX+dY)+(Z^a_X)^2\,dX+(Z^a_Y)^2\,dY)
  \end{multline*}
  and \eqref{eq:L2bdform} follows.\\[5mm]
  \textbf{Step 3}. Given any rectangle
  $\Omega=[X_l,X_r]\times[Y_l,Y_r]$, we consider a
  sequence of rectangular domains
  $\Omega_i=[X_i,X_{i+1}]\times[Y_i,Y_{i+1}]$ for
  $i=0,\ldots,N-1$ where $X_i$ and $Y_i$ are
  increasing and $X_0=X_l$, $Y_0=Y_l$, $X_N=X_r$,
  $Y_N=Y_r$ and such that $(\X,\Y)$,
  $(\bar\X,\bar\Y)$ belong to $\G(\Omega_i)$ for
  $s\in[s_i,s_{i+1}]$. We construct the sequence
  of rectangles such that either
  $s_{i+1}-s_{i}\leq\delta$ (and Step 1 applies)
  or $\bar\Y(s)-\bar\X(s)\leq\Y(s)-\X(s)$ or
  $\Y(s)-\X(s)\leq\bar\Y(s)-\bar\X(s)$ for
  $s\in[s_i,s_{i+1}]$ (and Step 2 applies). Hence,
  \begin{align*}
    \norm{Z\bullet(\bar \X,\bar
      \Y)}_{\G(\Omega)}^2&=\sum_{i=0}^{N-1}\norm{Z\bullet(\bar \X,\bar
      \Y)}_{\G(\Omega_i)}^2\\
    &\leq \sum_{i=0}^{N-1}C\norm{Z\bullet(\X,\Y)}_{\G(\Omega_i)}^2\quad\text{(by steps 1 and 2)}\\
    &\leq C\norm{Z\bullet(\X,\Y)}_{\G(\Omega)}^2.
  \end{align*}
  We can construct the sequence of rectangles as follows. Let $\bar N$
  be an integer such that $\frac{s_l-s_r}{\bar N}\leq\frac{\delta}2$
  and we set $\tilde s_j=s_l+j\frac\delta2$ for $j=0,\ldots,\bar
  N$. We take $s_0=s_l$ and define $s_i$ iteratively: Given $s_i$ and
  $j_i\in\{0,\ldots,\bar N-1\}$ such that $j_i\geq i$,
  $\X(s_i)=\bar\X(s_i)$, $\Y(s_i)=\bar\Y(s_i)$ and $s_i\in[\tilde
  s_{j_i},\tilde s_{j_i+1}]$. If $j_i+1=\bar N$, we set $N=i+1$,
  $s_{i+1}=s_{r}$ and we are done.  Otherwise, there exists an index
  $k\geq j_i+1$ such that $\bar\Y(s)-\bar\X(s)<\Y(s)-\X(s)$ for all
  $s\in[\tilde s_{j+1},\tilde s_{k})$ or
  $\bar\Y(s)-\bar\X(s)>\Y(s)-\X(s)$ for all $s\in[\tilde
  s_{j+1},\tilde s_{k})$ and there exists $s\in[\tilde s_{k},\tilde
  s_{k+1}]$ such that $\bar\Y(s)-\bar\X(s)=\Y(s)-\X(s)$ (which implies
  that $\X(s)=\bar\X(s)$ and $\Y(s)=\bar\Y(s)$). We then set
  $j_{i+1}=k$ and choose $s_{i+1}\in[\tilde s_{k},\tilde s_{k+1}]$
  such that $\X(s_{i+1})=\bar\X(s_{i+1})$ and
  $\Y(s_{i+1})=\bar\Y(s_{i+1})$. Since $j_i\geq i$, the iteration
  stops in a finite number of steps.
  \begin{figure}
    \centering
    \includegraphics[width=7cm]{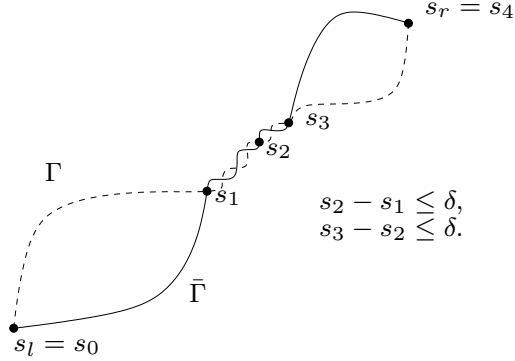}
    \caption{The interval $[s_l,s_r]$ is divided
      into \textit{large} intervals (in this
      example $[s_0,s_1]$ and $[s_3,s_4]$) where
      one curve is over the other and
      \textit{small} intervals of length smaller
      than $\delta$ (in this example $[s_1,s_2]$
      and $[s_2,s_3]$) where the curves can cross
      an arbitrarily number of times.}
    \label{fig:combarg}
  \end{figure}
\end{proof}

Given two solutions $Z$ and $\bar Z$, we want to
compare along curves in $\C$ the forms $Z_X\,dX$
and $Z_XY\,dY$ with $\bar Z_X\,dX$ and $\bar
Z_XY\,dY$, respectively. By using the same
argument as in the proof above, we obtain the
following stability result in $L^2$. This is a
stronger result than the one that could be
established from the fixed point argument in Lemma
\ref{lem:short} as the latter would only hold in
$L^\infty$.

\begin{lemma}[Stability in $L^2$]\label{lem:stabL2} Given $\Omega=[X_l,X_r]\times[Y_l,Y_r]$,
  $Z,\bar Z\in\H(\Omega)$ and
  $(\X,\Y)\in\C(\Omega)$. Then, for any
  $(\bar\X,\bar\Y)\in\C(\Omega)$,
  \begin{equation}
    \label{eq:L2stab}
    \norms{(Z-\bar Z)\bullet(\bar\X,\bar\Y)}_{\G(\Omega)}\leq C\norms{(Z-\bar Z)\bullet(\X,\Y)}_{\G(\Omega)}
  \end{equation}
  where
  $C=C(\norms{(\bar\X,\bar\Y)}_{\C(\Omega)},\tnorm{Z\bullet(\X,\Y)}_{\G(\Omega)}),\tnorm{\bar
    Z\bullet(\X,\Y)}_{\G(\Omega)})$ is a given
  increasing function with respect to both its
  arguments.
\end{lemma}

In the definition below of global solutions we
include a condition about the decay of the
solutions along the diagonal. This
condition is necessary to guarantee that, given a
solution, the curves which correspond to a given
time $T$ belong to $\C$.

\begin{definition}[Global solutions]
  \label{def:H} 
  Let $\H$ be the set of all functions
  $Z\in\Winfloc(\Real^2)$ such that
  \begin{enumerate}
  \item[(i)] $Z\in\H(\Omega)$ for all rectangular
    domains $\Omega$; and
  \item[(ii)] there exists a curve $(\X,\Y)\in\C$ such
    that $Z\bullet(\X,\Y)\in\G$.
  \end{enumerate}
\end{definition}
The condition (ii), which corresponds to a decay
condition, does not depend on the particular curve
for which it holds, as the next lemma shows. In
particular, we can replace condition (iv) in
Definition \ref{def:H} by the requirement that
$Z\bullet(\X_d,\Y_d)\in\G$ for the diagonal
($Y=X$), which is given by $\X_d(s)=\Y_d(s)=s$. We
then denote
\begin{equation*}
  \norm{Z}_{\H}=\norm{Z\bullet(\X_d,\Y_d)}_{\G}\quad\text{ and }\quad\tnorm{Z}_{\H}=\tnorm{Z\bullet(\X_d,\Y_d)}_{\G}.
\end{equation*}
\begin{lemma}
  \label{lem:exany}
  Given $Z\in\H$, we have $Z\bullet(\X,\Y)\in\G$
  \textnormal{for any curve}
  $(\X,\Y)\in\C$. Moreover, the limit
  $\lim_{s\to\infty}J(\X(s),\Y(s))$ is independent
  of the curve $(\X,\Y)\in\C$.
\end{lemma}
In this lemma we denote as before $\Z_4$ by $J$
where
$\Theta=(\X,\Y,\Z,\V,\W)=Z\bullet(\X,\Y)$. Later,
we will see that the limit of $J$ at infinity
corresponds to the total energy and the lemma
would allow us to prove that the total energy is
conserved.
\begin{proof}[Proof of Lemma \ref{lem:exany}]
  For any curve $(\bar\X,\bar\Y)\in\C$, we have to
  prove that
  \begin{equation}
    \label{eq:bdtnorrepa}
    \tnorm{Z\bullet(\bar\X,\bar\Y)}_{\G}<\infty\quad\text{ and }\quad\norm{Z\bullet(\bar\X,\bar\Y)}_{\G}<\infty
  \end{equation}
  and
  \begin{equation}
    \label{eq:limJrepar}
    \lim_{s\to\infty}\bar J(s)=0
  \end{equation}
  where $\bar J=\bar\Z_4$ with
  $(\bar\X,\bar\Y,\bar\Z,\bar\V,\bar\W)=Z\bullet(\bar\X,\bar\Y)$. For
  any real number $\bar s\in\Real$ that will
  eventually tend to infinity and denote
  $\Omega_{\bar s}=[\bar\X(-\bar s),\bar\X(\bar
  s)]\times[\bar\Y(-\bar s),\bar\Y(\bar s)]$. Let
  \begin{equation}
    \label{eq:condcurv1}
    s_{\max}=
    \begin{cases}
      \Y^{-1}(\bar\Y(\bar s))\quad\text{ if
      }\Y(\bar\X(\bar s))\leq
      \bar\Y(\bar s),\\
      \X^{-1}(\bar\X(\bar s))\quad\text{
        otherwise},
    \end{cases}
  \end{equation}
  and 
  \begin{equation}
    \label{eq:condcurv2}
    s_{\min}=
    \begin{cases}
      \Y^{-1}(\bar\Y(-\bar s))\quad\text{ if
      }\Y(\bar\X(-\bar s))\leq
      \bar\Y(-\bar s),\\
      \X^{-1}(\bar\X(-\bar s))\quad\text{
        otherwise},
    \end{cases}
  \end{equation}
  see Figure \ref{fig:twocurves} for an example.
  One can check that by construction
  $s_{\min}\leq-\bar s\leq\bar s\leq s_{\max}$ and
  we denote $\tilde\Omega_{\bar
    s}=[\X(s_{\min}),\X(s_{\max})]\times[\Y(s_{\min}),\Y(s_{\max})]$. We
  have $\Omega_{\bar s}\subset\tilde\Omega_{\bar
    s}$. We construct the curve which consists of
  a (vertical or horizontal) straight line joining
  $(\X(s_{\min}),\Y(s_{\min}))$ and $(\X(\bar
  s),\Y(\bar s))$, the curve
  $(\bar\X(s),\bar\Y(s))$ for $s\in[-\bar s,\bar
  s]$ and another (vertical or horizontal)
  straight line joining $(\bar\X(\bar
  s),\bar\Y(\bar s))$ and
  $(\X(s_{\max}),\Y(s_{\max}))$, see Figure
  \ref{fig:twocurves}. We denote by
  $(\tilde\X,\tilde\Y)$ this curve and we have
  that $(\tilde\X,\tilde\Y)$ and $(\X,\Y)$ belong
  to $\G(\tilde\Omega_{\bar s})$.  By Lemma
  \ref{lem:aprioribd}, we get
  \begin{multline*}
    \tnorm{Z\bullet(\bar\X,\bar\Y)}_{\G(\Omega_{\bar
        s})}\leq\tnorm{Z\bullet(\bar\X,\bar\Y)}_{\G(\tilde\Omega_{\bar
        s})}\\\leq
    C_1(\norms{(\bar\X,\bar\Y)}_{\C(\tilde\Omega_{\bar
        s})},\tnorm{\Theta}_{\G(\tilde\Omega_{\bar
        s})})\leq
    C_1(\norms{(\bar\X,\bar\Y)}_{\C},\tnorm{\Theta}_{\G})
  \end{multline*}
  and by letting $\bar s$ tend to infinity, we get
  $\tnorm{Z\bullet(\bar\X,\bar\Y)}_{\G}<\infty$.  By
  Lemma \ref{lem:gron}, we get
  \begin{equation}
    \label{eq:ZbarXYbars}
    \norm{Z\bullet(\bar\X,\bar\Y)}_{\G(\Omega_{\bar s})}\leq\norm{Z\bullet(\bar\X,\bar\Y)}_{\G(\tilde\Omega_{\bar s})}\leq C\norm{Z\bullet(\X,\Y)}_{\G(\tilde\Omega_{\bar s})}\leq C\norm{Z\bullet(\X,\Y)}_{\G}
  \end{equation}
  where the constant $C$ depends on $\norm{(\bar
    X,\bar Y)}_{\G(\tilde\Omega_{\bar s})}$ and
  $\tnorm{Z\bullet(\X,\Y)}_{\G(\tilde\Omega_{\bar
      s})}$, that is, on $\norm{(\bar X,\bar
    Y)}_{\G}$ and $\tnorm{Z\bullet(\X,\Y)}_{\G}$,
  which are independent on $\bar s$. By letting
  $\bar s$ tend to infinity in
  \eqref{eq:ZbarXYbars}, we get
  $\norm{Z\bullet(\bar\X,\bar\Y)}_{\G}<\infty$. It
  remains to prove \eqref{eq:limJrepar}. We know
  that $\bar J$ is positive. We denote
  $J(s)=\Z_4(s)$ with
  $(\X,\Y,\Z,\V,\W)=Z\bullet(\X,\Y)$ and, slightly
  abusing the notation, we denote also by $J$,
  $J(X,Y)=Z_4(X,Y)$. For any $s\in\Real$, let
  $s_1=\X^{-1}\bar\X(s)$ and
  $s_2=\Y^{-1}\bar\Y(s)$. If $s_1\leq s_2$, then
  $\bar\X(s)=\X(s_1)\leq\X(s_2)$ and
  $\Y(s_1)\leq\bar\Y(s)=\bar\Y(s_2)$. By the monotonicity of
  $J(X,Y)$, we get
  \begin{equation*}
    J(\X(s_1),\Y(s_1))\leq J(\bar\X(s),\bar\Y(s))\leq J(\X(s_2),\Y(s_2)).
  \end{equation*}
  If $s_2\leq s_1$, we get a similar result so
  that, finally,
  \begin{subequations}
    \label{eq:zeleqJs1s2}
    \begin{equation}
      \label{eq:zeleqJs1s2a}
      \min\{J(\X(s_1),\Y(s_1)),J(\X(s_2),\Y(s_2))\}\leq J(\bar \X(s),\bar \Y(s))
    \end{equation}
    and
    \begin{equation}
      \label{eq:zeleqJs1s2b}
      J(\bar \X(s),\bar \Y(s))\leq\max\{J(\X(s_1),\Y(s_1)),J(\X(s_2),\Y(s_2))\}.
    \end{equation}
  \end{subequations}
  Since 
  \begin{equation*}
    \abs{s_1-s}\leq\abs{\X(s_1)-s_1}+\abs{\bar\X(s)-s}\leq\norm{(\X,\Y)}_{\C}+\norm{(\bar\X,\bar\Y)}_{\C},
  \end{equation*}
  we have that $\lim_{s\to\pm\infty}s_1=\pm\infty$ 
  similarly we obtain that
  $\lim_{s\to\pm\infty}s_2=\pm\infty$. Hence,
  \eqref{eq:zeleqJs1s2} yields
  \begin{equation*}
    \lim_{s\to\pm\infty}J(\bar \X(s),\bar \Y(s))=\lim_{s\to\pm\infty} J(\X(s),\Y(s)).
  \end{equation*}
  Thus, these limits are independent of the curve
  $\bar\Gamma$ which is chosen. The existence of
  the limits is guaranteed by the monotonicity and
  boudedness of $J$. The identity
  \eqref{eq:limJrepar} follows from
  \eqref{eq:normalizationJ}.
  \begin{figure}
    \centering
    \includegraphics[width=7cm]{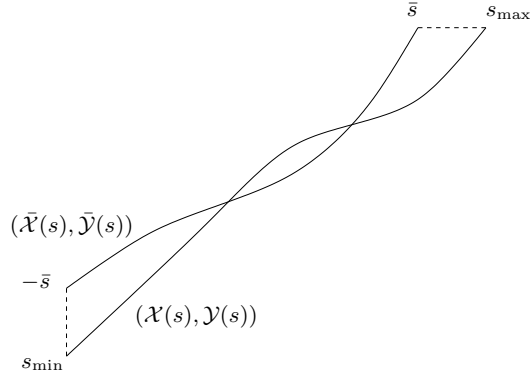}
    \caption{Prolongation of the curve $\bar\Gamma$.}
    \label{fig:twocurves}
  \end{figure}
\end{proof}

From Lemma \ref{lem:globalsol}, we infer the
following global existence theorem for the
equivalent system.

\begin{theorem}[Existence and uniqueness of global
  solutions]
  \label{th:globalsol}
  For any initial data\\ $(\X,\Y,\Z,\V,\W)$ in
  $\G$, there exists a unique solution $Z\in \H$
  such that $\Theta=Z\bullet (\X,\Y)$. We denote this solution mapping
  by $\Sb\colon\G\to\H$.
\end{theorem}

\section{Semigroup of solution $S_T$ in $\F$}
\label{sec:semigroupinG}
From a solution function $Z\in\H$ in the whole
plane, we want to extract the data at a given
time. It is enough to do it at $t=0$, and the
definition below describes how we proceed.

\begin{definition}
  \label{def:Eb}
  Given $Z\in\H$, we define
  \begin{equation}
    \label{eq:defXsE}
    \X(s)=\sup\{X\in\Real\mid t(X',2s-X')<0\quad\text{for all }X'<X\}
  \end{equation}
  and $\Y(s)=2s-\X(s)$. Then, we have
  $(\X(s),\Y(s))\in\C$ and
  $Z\bullet(\X,\Y)\in\G_0$. We denote by $\Eb$ the
  mapping from $\H$ to $\G_0$ that associates to
  any $Z\in\H$ the element
  $Z\bullet(\X,\Y)\in\G_0$ as defined here.
\end{definition}

\begin{proof}[Proof of the well-posedness of Definition \ref{def:Eb}]
  First we prove that $\X$ is increasing. Let
  $X_n$ be a sequence such that $X_n<\X(s)$ and
  $X_n\to \X(s)$. We have $t(X_n,2s-X_n)<0$ for
  any $\bar s>s$. Since $t$ is decreasing with
  respect to the second variable (as $t_Y\leq0$),
  we have that $t(X_n,2\bar s-X_n)<0$ and
  therefore $X_n\leq \X(\bar s)$. After letting
  $n$ tend to infinity, we obtain $\X(s)\leq
  \X(\bar s)$ so that $\X$ is an increasing
  function. Let us prove that $\X$ is Lipschitz
  with Lipschitz coefficient smaller than $2$. Let
  us assume the opposite, i.e., there exists $\bar
  s>s$ such that
  \begin{equation}
    \label{eq:X0notlip}
    \X(\bar s)-\X(s)>2(\bar s-s).
  \end{equation}
  It implies that $\Y(s)>\Y(\bar s)$. Since
  $t_X\geq0$ and $t_Y\leq0$, we have, for any
  $(X,Y)\in[\X(s),\X(\bar s)]\times[\Y(\bar
  s),\Y(s)]$
  \begin{equation*}
    0=t(\X(s),\Y(s))\leq t(X,\Y(s))\leq t(X,Y)\leq t(X,\Y(\bar s))\leq t(\X(\bar s),\Y(\bar s))=0
  \end{equation*}
  and, therefore $t(X,Y)=0$ on
  $\Omega=[\X(s),\X(\bar s)]\times[\Y(\bar
  s),\Y(s)]$. Let us consider the point
  $(X,Y)\in\Omega$ for $Y=\Y(s)$ and $X=2\bar
  s-\Y(s)$. We have $t(X,Y)=0$, $X+Y=2\bar s$ and
  $X<\X(\bar s)$, which contradict the definition
  of $\X$ at $\bar s$. Thus, we have proved that
  $\X$ is Lipschitz. To show that $(\X,\Y)\in\C$,
  it remains to prove that
  $\norm{\X-\Y}_{L^\infty(\Real)}<\infty$. We
  claim that there there exists $\bar L$ such that
  \begin{equation}
    \label{eq:claimt}
    \liminf_{X\to\infty}t(X+L,X)\geq1
  \end{equation}
  for any $L\geq\bar L$. Let us prove this
  claim. By using the fact that $x_X=c(U)t_X$ and
  $c(U)\geq\frac1\kappa$, we get
  \begin{align}
    \notag
    t(X+L,X)&=t(X,X)+\int_{X}^{X+L}t_X(\tilde X,X)\,d\tilde X\\
    \label{eq:txplx}
    &\geq
    -\tnorm{Z}_{\H}+\frac{L}{2\kappa}+\frac{1}{\kappa}\int_{X}^{X+L}(x_X-\frac{1}{2})(\tilde
    X,X)\,d\tilde X.
  \end{align}
  We look at the domain
  $\Omega_{X,L}=[X,X+L]\times[X,X+L]$. We consider
  the curve $\X_d(s)=\Y_d(s)=s$ (the diagonal) and
  the curve $(\bar\X,\bar\Y)$ which consists of a
  horizontal and a vertical segment given by
  \begin{equation*}
    \begin{cases}
      \bar\X(s)=2s-X,\quad\bar\Y(s)=X&\text{ for }s\in[0,X+\frac{L}2]\\
      \bar\X(s)=X+L,\quad\bar\Y(s)=2s-(X+L)&\text{
        for }s\in[X+\frac{L}2,X+L].
    \end{cases} 
  \end{equation*}
  By Lemma \ref{lem:gron} we get
  \begin{equation}
    \label{eq:intxxplxx}
    \int_X^{X+L}(x_X-\frac12)^2(\tilde
    X,X)\,d\tilde X\leq\norm{Z\bullet(\bar
      X,\bar\Y)}_{\G(\Omega_{X,L})}^2 \leq
    C\norm{Z\bullet(\X_d,\Y_d)}_{\G(\Omega_{X,L})}^2
  \end{equation}
  where the constant $C$ depends on $L$ and
  $\tnorm{Z\bullet(\X_d,\Y_d)}_{\G(\Omega_{X,L})}$
  which is bounded by $\tnorm{Z}_{\H}$. We have
  \begin{equation*}
    \lim_{X\to\infty}\norm{Z\bullet(\X_d,\Y_d)}_{\G(\Omega_{X,L})}^2=\lim_{X\to\infty}\int_{X}^{X+L}(U^2(\tilde X,\tilde
    X)+(Z^a_X)^2(\tilde X,\tilde X)+(Z^a_Y)^2(\tilde
    X,\tilde X))\,d\tilde X=0,
  \end{equation*}
  and therefore \eqref{eq:intxxplxx} and
  \eqref{eq:txplx} yield
  \begin{equation*}
    \liminf_{X\to\infty}t(X+L,X)\geq-\norm{Z}_{\H}+\frac{L}{2\kappa}
  \end{equation*}
  which, for $L$ large enough, implies
  \eqref{eq:claimt}. Using the same type of
  argument, we prove that there exists $L$ such
  that
  \begin{align}
    \label{eq:liminfMbd1}
    \liminf_{X\to\infty}t(X+L,X)&\geq1,&\limsup_{X\to\infty}t(X-L,X)&\leq-1,\\
    \label{eq:liminfMbd2}
    \liminf_{X\to-\infty}t(X+L,X)&\geq1,&\liminf_{X\to-\infty}t(X-L,X)&\leq-1.
  \end{align}
  Let us prove that
  $\limsup_{s\to\infty}(\X(s)-s)\leq\frac{L}2$. We
  assume the opposite and then, there exists
  $s\in\Real$ such that
  $t(s+\frac{L}2,s-\frac{L}2)\geq 1$, by
  \eqref{eq:liminfMbd1}, and
  $\X(s)>s+\frac{L}2$. It implies that
  $\Y(s)=2s-\X(s)\leq s-\frac{L}2$, and, using the
  monotonicity of $t$ (that is $t_X\geq0$ and
  $t_Y\leq0$), we get $1\leq
  t(s+\frac{L}2,s-\frac{L}2)\leq
  t(\X(s),\Y(s))=0$, which is a
  contradiction. Similarly one proves that
  $\liminf_{s\to\infty}(\X(s)-s)\geq -\frac{L}2$,
  $\limsup_{s\to-\infty}(\X(s)-s)\leq \frac{L}2$
  and $\liminf_{s\to-\infty}(\X(s)-s)\geq
  -\frac{L}2$ and it follows that
  \begin{equation*}
    \limsup_{s\to\pm\infty}\abs{\X(s)-s}\leq \frac{L}2.
  \end{equation*}
  Hence, the condition \eqref{eq:regXY} is
  satisfied and $(\X,\Y)\in\C$. By Lemma
  \ref{lem:exany}, we have
  $(\X,\Y,\Z,\V,\W)=Z\bullet(\X,\Y)\in\G$ and by
  construction $\Z_1(s)=t(\X(s),\Y(s))=0$ so that
  $Z\bullet(\X,\Y)\in\G_0$.
\end{proof}

\begin{definition} \label{def:tT}
  Given any $T$, let us introduce the mapping
  $\tb_T\colon\H\to\H$ defined as follows. For any
  $Z\in\H$, let $\tb_T(Z)=\bar Z\in\H$ be given by
  \begin{equation*}
    \bar t(X,Y)=t(X,Y)-T
  \end{equation*}
  and
  \begin{align*}
    \bar x(X,Y)&=x(X,Y),&\bar U(X,Y)&=U(X,Y),\\
    \bar J(X,Y)&=J(X,Y),&\bar K(X,Y)&=K(X,Y).
  \end{align*}
  We have
  \begin{equation}
    \label{eq:tranlaw}
    \tb_{T+T'}=\tb_T\circ\tb_{T'}.
  \end{equation}
\end{definition}
We have now defined the following mappings:
\begin{equation}
  \label{eq:diagmappings}
  \xymatrix{\F\ar@<1mm>[r]^{\Cb}
    &\G_0 \ar@<1mm>[l]^{\Db} \ar@<1mm>[r]^{\Sb}&\H \ar@<1mm>[l]^{\Eb} \ar@(dr,ur)[lr]_{\tb_T}}
\end{equation}
The mapping $\tb_T$ is used to extract the
solution at any given time $T$ by only using the
operator $\Eb$, which is designed for time
zero. Indeed, by taking $\Eb\circ \tb_T$, we
recover an element in $\G_0$ which corresponds to
the solution at time $T$. In the following lemma,
we prove that $\F$ and $\H$ are in bijection,
which also justify the introduction of $\F$: It is
a consistent way to parametrize initial data: To
any element in $\F$, there corresponds a unique
solution in $\H$, and vice versa. The $\G_0$ does
not fit that role as $\G_0$ and $\H$ are not in
bijection.
\begin{lemma}
  \label{lem:cdeeqe}
  We have
  \begin{equation}
    \label{eq:cdeeqe}
    \Cb\circ \Db\circ \Eb=\Eb,\quad\Db\circ\Cb=\id
  \end{equation}
  and
  \begin{equation}
    \label{eq:escc}
    \Eb\circ\Sb\circ\Cb=\Cb,\quad\Sb\circ\Eb=\id.
  \end{equation}
  It follows that $\Sb\circ\Cb=(\Db\circ\Eb)^{-1}$
  and the sets $\F$ and $\H$ are in bijection.
\end{lemma}
\begin{proof} 
  \textbf{Step 1.} We prove
  \eqref{eq:cdeeqe}. Given $Z\in\H$, let
  $(\X,\Y,\Z,\V,\W)=\Eb(Z)$,
  $(\psi_1,\psi_2)=\Db(\X,\Y,\Z,\V,\W)$ and
  $(\bar\X,\bar\Y,\bar\Z,\bar\V,\bar\W)=\Cb(\psi_1,\psi_2)$. We
  want to prove that
  $(\bar\X,\bar\Y,\bar\Z,\bar\V,\bar\W)=(\X,\Y,\Z,\V,\W)$. We
  have to prove that $\X=\bar\X$, the rest will
  easily follow.  For any $s\in\Real$, we claim
  that for any couple $(X,Y)$ such that $X<\X(s)$
  and $X+Y=2s$ then we have either
  \begin{equation}
    \label{eq:condx1low}
    x_1(X)<x_1(\X(s))\quad\text{ or }\quad x_2(Y)>x_2(\Y(s)).
  \end{equation}
  Let us assume the opposite, that is, there
  exist $\bar s$, $\bar{X}$ and $\bar{Y}$ such
  that $\bar X<\X(\bar{s})$, $\bar{X}+\bar{Y}=2\bar{s}$ and
  \begin{equation*}
    x_1(\bar{X})=x_1(\X(\bar{s}))=x(\bar{s})=x_2(\Y(\bar{s}))=x_2(\bar{Y}).
  \end{equation*}
  Here, $x(s)$ denotes $\Z_2(s)$, see
  \eqref{eq:x1x2x}. Let $s_0=\X^{-1}(\bar X)$ and
  $s_1=\Y^{-1}(\bar Y)$. Since
  $\bar{X}<\X(\bar{s})$ and $\bar{Y}>\Y(\bar{s})$,
  we have $s_0<\bar{s}<s_1$. We have
  \begin{equation*}
    x(s_0)=x_1(\X(s_0))=x_1(\bar X)=x(\bar{s})
  \end{equation*}
  and, similarly, we obtain that
  $x(s_1)=x(\bar{s})$. We consider the rectangular
  domain
  $\Omega=[\X(s_0),\X(s_1)]\times[\Y(s_0),\Y(s_1)]$. Since
  $x(s)=x(s_0)=x(s_1)$ for all $s\in[s_0,s_1]$, we
  have $\dot x=0$ on $[s_0,s_1]$ because $x$ is
  nondecreasing. We have $\dot
  x=0=\V_2(\X)\dot\X+\W_2(\Y)\dot\Y$ on
  $[s_0,s_1]$, which implies that $\V_2(X)=0$ for
  a.e. $X\in[\X(s_0),\X(s_1)]$ and $\W_2(Y)=0$ for
  a.e. $Y\in[\Y(s_0),\Y(s_1)]$. By
  \eqref{eq:reltx1} and \eqref{eq:relVW} it
  implies $\V_1=\V_2=\V_3=0$ on
  $[\X(s_0),\X(s_1)]$ and $\W_1=\W_2=\W_3=0$ on
  $[\Y(s_0),\Y(s_1)]$. Then, we can check that
  $\tilde Z$ given by
  \begin{equation*}
    \tilde t(X,Y)=0,\quad \tilde x(X,Y)=x(\bar{s}),\quad \tilde U(X,Y)=U(\bar{s})
  \end{equation*}
  and
  \begin{equation*}
    \tilde J(X,Y)=J_1(X)+J_2(Y),\quad \tilde K(X,Y)=K_1(X)+K_2(Y)
  \end{equation*}
  is a solution to \eqref{eq:goveq} in $\Omega$
  (that is, $\tilde Z\in\H(\Omega)$) and $\tilde
  Z\bullet(\X,\Y)=(\X,\Y,\Z,\V,\W)$. By uniqueness
  of the solution, we get $\tilde Z=Z$.  In
  particular, we have $t(\bar X,\bar Y)=0$ such
  that $\bar X+\bar Y=2s$ and $\bar X<\X(\bar s)$,
  which contradicts the definition of $\X(s)$
  given by \eqref{eq:defXsE}. This concludes the
  proof of the claim \eqref{eq:condx1low}.  Since
  $x_1(\X(s))=x_2(2s-\X(s))$ we get, by
  \eqref{eq:x1x2x}, that
  \begin{equation*}
    \bar\X(s)\leq\X(s).
  \end{equation*} 
  We have by the continuity of $x_1$ and $x_2$ that
  \begin{equation}
    \label{eq:egbarX}
    x_1(\bar\X(s))=x_2(\bar\Y(s)).
  \end{equation}
  Let us assume that $\bar\X(s)<\X(s)$, then, by
  the claim \eqref{eq:condx1low} we have proved,
  we have either $x_1(\bar\X(s))<x_1(\X(s))$ or
  $x_2(\bar\Y(s))>x_2(\Y(s))$. If
  $x_1(\bar\X(s))<x_1(\X(s))$, then, as
  $\bar\Y(s)>\Y(s)$
  \begin{equation*}
    x_1(\bar\X(s))<x_1(\X(s))=x_2(\Y(s))\leq x_2(\bar\Y(s))
  \end{equation*}
  which contradicts \eqref{eq:egbarX}. Similarly,
  we check that if $x_2(\bar\Y(s))>x_2(\Y(s))$
  then we obtain a contradiction to
  \eqref{eq:egbarX}. Hence, $\bar\X=\X$ and
  therefore $\bar\Y=\Y$. Then,
  $\bar x(s)=x_2(\bar\X(s))=x_2(\X(s))=x(s)$
  and similarly, we treat the other components of
  $\bar\Z$. From the definitions of $\Cb$ and
  $\Db$, we have that $\bar\V=\V$ and
  $\bar\W=\W$. Hence, we have proved that
  $\Cb\circ \Db\circ \Eb=\Eb$. The fact that
  $\Db\circ\Cb=\id$ follows directly from the
  definitions of $\Cb$ and $\Db$. 

  \textbf{Step 2.} We prove \eqref{eq:escc}.
  Given $(\psi_1,\psi_2)\in\F$, let
  $(\X,\Y,\Z,\V,\W)=\Cb(\psi_1,\psi_2)$,
  $Z=\Sb(\X,\Y,\Z,\V,\W)$ and
  $(\bar\X,\bar\Y,\bar\Z,\bar\V,\bar\W)=\Eb(Z)$. We
  have to prove that $\bar\X=\X$, the rest will
  easily follow. Since $Z\in\H$ is a solution with
  data $(\X,\Y,\Z,\V,\W)\in\G_0$, we have
  $t(\X(s),\Y(s))=0$. Hence, from the definition
  of $\Eb$, we get $\bar\X(s)\leq\X(s)$. Assume
  that there exists $s\in\Real$ such that
  $\bar\X(s)<\X(s)$. By the definition of $\Eb$,
  we have $t(\bar\X(s),\bar\Y(s))=0$. Let
  \begin{equation*}
    s_0=\X^{-1}(\bar\X(s))\quad\text{ and }\quad
    s_1=\Y^{-1}(\bar\Y(s)).    
  \end{equation*}
  We have 
  \begin{equation*}
    \X(s_0)=\bar\X(s)<\X(s)\quad\text{ and }\quad
    \Y(s_1)=\bar\Y(s)>\Y(s), 
  \end{equation*}
  and therefore $s_0<s<s_1$. Due to the
  monotonicity of $t(X,Y)$ (that is, $t_X\geq0$
  and $t_Y\leq0$), since
  $t(\X(s_0),\Y(s_0))=t(\bar\X(s),\bar\Y(s))=t(\X(s_1),\Y(s_1))=0$,
  we get that $t(X,Y)=0$ on the rectangle
  $\Omega=[\X(s_0),\X(s_1)]\times[\Y(s_0),\Y(s_1)]$. It
  implies that $x_X=c(U)t_X=0$ and $x_Y=c(U)t_Y=0$
  on $\Omega$. Hence, the function $x(X,Y)$ is
  constant on $\Omega$ and we have
  \begin{equation*}
    x_1(\bar\X(s))=x_1(\X(s_0))=x(\X(s_0),\Y(s_0))=x(\X(s_1),\Y(s_1))=x_2(\Y(s_1))=x_2(\bar\Y(s)).
  \end{equation*}
  However, the fact that
  $x_1(\bar\X(s))=x_2(\bar\Y(s))$ and
  $\bar\X(s)<\X(s)$ contradicts the definition of
  $\X$ in \eqref{eq:defXs}, and therefore we have
  proved that $\bar X=\X$. Then, $\bar\Y=\Y$ and
  \begin{equation*}
    \bar\Z(s)=Z(\bar\X(s),\bar\Y(s))=Z(\X(s),\Y(s))=\Z(s).
  \end{equation*}
  Similarly, one proves that $\bar\V=\V$ and
  $\bar\W=\W$. Thus we have proved that
  $\Eb\circ\Sb\circ\Cb=\Cb$. The fact that
  $\Sb\circ\Eb=\id$ follows from the uniqueness of
  the solution for a given data.
\end{proof}

\begin{definition}  \label{def:ST}
  For any $T\geq0$, we define the mapping
  $S_T\colon\F\to\F$ by
  \begin{equation}
    \label{eq:defST}
    S_T=\Db\circ \Eb\circ\tb_T\circ \Sb\circ \Cb.
  \end{equation}
\end{definition}
\begin{theorem}
  \label{th:Stsemigroup}
  The mapping $S_T$ is a semigroup.
\end{theorem}

\begin{proof}
  We have
  \begin{align*}
    S_T\circ S_{T'}&=\Db\circ \Eb\circ\tb_T\circ \Sb\circ \Cb\circ \Db\circ \Eb\circ\tb_{T'}\circ \Sb\circ \Cb\\
    &=\Db\circ \Eb\circ\tb_T\circ\tb_{T'}\circ \Sb\circ \Cb&&\text{(by Lemma \ref{lem:cdeeqe})}\\
    &=\Db\circ \Eb\circ\tb_{T+T'}\circ \Sb\circ \Cb&&\text{(by \eqref{eq:tranlaw})}\\
    &=S_{T+T'}.
  \end{align*}
\end{proof}

\section{Returning to the original variables}
\label{sec:backto}
\begin{definition}
  \label{def:M}
  Given $\psi=(\psi_1,\psi_2)\in\F$, we define
  $(u,R,S,\mu,\nu)\in\D$ as
  \begin{subequations}
    \label{eq:defMbpf}
    \begin{equation}
      \label{eq:defMbpfu1}
      u(x)=U_1(X)\text{ if }x_1(X)=x
    \end{equation}
    or, equivalently,
    \begin{equation}
      \label{eq:defMbpfu2}
      u(x)=U_2(X)\text{ if }x_2(X)=x
    \end{equation}
    and\footnote{The push-forward of a measure $\lambda$ by a function $f$ is the measure $f_\#\lambda$ defined by $f_\#\lambda(B)=\lambda(f^{-1}(B))$ for Borel sets $B$.}
    \begin{align}
      \label{eq:defMbpf1}
      \mu&=(x_1)_\#(J_1'(X)\,dX),\\
      \label{eq:defMbpf2}
      \nu&=(x_2)_\#(J_2'(Y)\,dY),\\
      \label{eq:defMbpf3}
      R(x)\,dx&=(x_1)_\#\left(2c(U_1(X))V_1(X)\,dX\right),\\
      \label{eq:defMbpf4}
      S(x)\,dx&=(x_2)_\#\left(-2c(U_2(Y))V_2(Y)\,dY\right).
    \end{align}
  \end{subequations}
  The relations \eqref{eq:defMbpf3} and
  \eqref{eq:defMbpf4} are equivalent to 
  \begin{subequations}
    \label{eq:reluRSF}
    \begin{equation}
      \label{eq:reluRF}
      R(x_1(X))x_1'(X)=2c(U_1(X))V_1(X)
    \end{equation}
    and
    \begin{equation}
      \label{eq:reluSF}
      S(x_2(Y))x_2'(Y)=2c(U_2(Y))V_2(Y)
    \end{equation}
  \end{subequations}
  for a.e. $X$ and $Y$. We denote by
  $\Mb\colon\F\to\D$ the mapping that to any
  $\psi\in\F$ associates $(u,R,S,\mu,\nu)\in\F$ as
  defined above.
\end{definition}
We have to prove that the measures
$(x_1)_\#\left(c(U_1(X))V_1(X)\,dX\right)$ and
$(x_2)_\#\left(-c(U_2(Y))V_2(Y)\,dY\right)$ are
absolutely continuous with respect to the Lebesgue
measure and that $(u,R,S,\mu,\nu)$ belongs to $\D$
so that the definition is well-posed. It will be
done in the proof of the following lemma where an
equivalent definition of the mapping $\Mb$ is
given.
\begin{lemma}
  \label{lem:revufromG}
  Given $\psi=(\psi_1,\psi_2)\in\F$, let
  $(u,R,S,\mu,\nu)=\Mb(\psi_1,\psi_2)$ as defined
  in Definition \ref{def:M}. Then, for any
  $\Theta=(\X,\Y,\Z,\V,\W)\in\G_0$ such that
  $(\psi_1,\psi_1)=\Db\Theta$, we have
  \begin{subequations}
    \label{eq:relufromG}
    \begin{equation}
      \label{eq:relufromGu}
      u(\bar x)=U(s)\text{ if }\bar x=x(s)
    \end{equation}
    and
    \begin{align}
      \label{eq:relufromG1}
      \mu&=x_\#(\V_4(\X(s))\dot \X(s)\,ds),\\
      \label{eq:relufromG2}
      \nu&=x_\#(\W_4(\Y(s))\dot \Y(s)\,ds),\\
      \label{eq:relufromG3}
      R(x)\,dx&=x_\#\left(2c(U(s))\V_3(\X(s))\dot \X(s)\,ds\right),\\
      \label{eq:relufromG4}
      S(x)\,dx&=x_\#\left(-2c(U(s))\W_3(\Y(s))\dot
        \Y(s)\,ds\right).
    \end{align}
  \end{subequations}
  The relations \eqref{eq:relufromG3} and
  \eqref{eq:relufromG4} are equivalent to
  \begin{subequations}
    \label{eq:reluRS}
    \begin{equation}
      \label{eq:reluR}
      R(x(s))\V_2(\X(s))=c(U(s))\V_3(\X(s))\text{ for any }s\text{ such that }\dot\X(s)>0
    \end{equation}
    and
    \begin{equation}
      \label{eq:reluS}
      S(x(s))\W_2(\X(s))=-c(U(s))\W_3(\X(s))\text{ for any }s\text{ such that }\dot\Y(s)>0,
    \end{equation}
  \end{subequations}
  respectively.
\end{lemma}
\begin{proof} We decompose the proof into 5 steps.
  
  \textbf{Step 1.} We prove that
  \eqref{eq:defMbpf} imply
  \eqref{eq:relufromG}. If $\bar x=x_1(X)$, let
  $s=\X^{-1}(X)$. Then, we have $\bar
  x=x_1(\X(s))=x(s)$ and
  $U_1(X)=U_1(\X(s))=U(s)$. Hence,
  \eqref{eq:defMbpfu1} implies
  \eqref{eq:relufromGu}. For any measurable set
  $A$, we have
  \begin{align*}
    \mu(A)&=\int_{x_1^{-1}(A)}J_1'(X)\,dX\\
    &=\int_{(x_1\circ\X)^{-1}(A)}J_1'(\X(s))\dot\X(s)\,ds&\text{
      (after a change of variables)}\\
    &=\int_{x^{-1}(A)}\V_4(\X(s))\dot\X(s)\,ds&\text{
      (by the definition of $\Db$)}
  \end{align*}
  and therefore
  $\mu=x_{\#}(\V_4(\X(s))\dot\X(s)\,ds)$. One
  proves in the same way the other identities in
  \eqref{eq:relufromG}. 

  \textbf{Step 2.}  We prove that $u$ is a
  well-defined function $L^2$ that is Hölder
  continuous with exponent $1/2$. Given $\bar x$
  such that $\bar x=x(s_0)=x(s_1)$. Since $x$ is
  nondecreasing, it implies that $\dot
  x=\V_2(\X)\dot\X+\W_2(\Y)\dot\Y=0$ in
  $[s_0,s_1]$. Hence,
  $\V_2(\X)\dot\X=\W_2(\Y)\dot\Y=0$ as both
  quantities are positive. By \eqref{eq:relVW}, it
  implies that $\V_3(\X)\dot\X=\W_3(\Y)\dot\Y=0$
  and therefore
  \begin{equation*}
    \dot U=\V_3(\X)\dot\X+\W_3(\Y)\dot\Y=0
  \end{equation*}
  in $[s_0,s_1]$ and $U(s_0)=U(s_1)$. The
  definition of $u$ is therefore well-posed. We
  have
  \begin{equation*}
    \int_{\Real}u^2(x)\,dx=\int_{\Real}u^2(x(s))\dot x(s)\,ds\leq\norm{U}_{L^2}^2
  \end{equation*}
  and $u\in L^2(\Real)$. We have
  \begin{equation}
    \label{eq:uholpr}
    u(x(s))-u(x(\bar s))=\int_{\bar s}^{s}\dot U(s)\,ds
    =\int_{\bar
      s}^{s}(\V_3(\X)\dot\X+\W_3(\Y)\dot\Y)\,ds.
  \end{equation}
  Since $\Theta\in\G_0$, we have $t(s)=0$ which
  implies that $\V_1(\X)\dot\X=\W_1(\Y)\dot\Y$
  and, therefore, $\V_2(\X)\dot\X=\W_2(\Y)\dot\Y$,
  by \eqref{eq:reltx1}, and
  \begin{equation}
    \label{eq:dotxteqz}
    \dot x=2\V_2(\X)\dot\X=2\W_2(\Y)\dot\Y.
  \end{equation}
  By using the Cauchy--Schwarz inequality and \eqref{eq:relVW}, we
  get
  \begin{align}
    \notag \int_{\bar
      s}^{s}\abs{\V_3(\X)}\dot\X\,ds&\leq\left(\int_{\bar s}^s\V_3^2(\X)\dot\X\,ds\right)^{1/2}\left(\int_{\bar s}^s\dot\X\,ds\right)^{1/2}\\
    \notag &\leq C\left(\int_{\bar
        s}^s\V_2^2(\X)\dot\X\,ds\right)^{1/2}\\
    \label{eq:holderbd}
    &\leq C\left(\int_{\bar s}^s\dot
      x\,ds\right)^{1/2}=C(x(s)-x(\bar s))^{1/2}
  \end{align}
  where the constant $C$ depends on
  $\tnorm{\Theta}_{\G}$ and $\abs{\X(s)-\X(\bar
    s)}$. Similarly, one proves that $\int_{\bar
    s}^{s}\abs{\W_3(\Y)}\dot\Y\,ds\leq
  C(x(s)-x(\bar s))^{1/2}$. Hence,
  \eqref{eq:uholpr} implies that $u$ is locally
  Hölder continuous with exponent $1/2$. 

\textbf{Step 3.} We show that the measures
    $x_\#\left(c(U(s))\V_3(\X(s))\dot
      \X(s)\,ds\right)$ and\\
    $x_\#\left(-c(U(s))\W_3(\Y(s))\dot
      \Y(s)\,ds\right)$ are absolutely continuous
    and \eqref{eq:reluR} and \eqref{eq:reluS}
    hold. The inequality
  \eqref{eq:holderbd} proves that the measure
  $\V_3(\X)\dot\X\,ds$ is absolutely continuous
  with respect to $\dot x\,ds$. For any set $A$ of
  zero measure, we have $\int_{x^{-1}(A)}\dot
  x\,ds=0$ and therefore
  $\int_{x^{-1}(A)}\V_3(\X)\dot\X\,ds=0$. It
  follows that
  \begin{equation*}
    x_\#\left(2c(U(s))\V_3(\X(s))\dot
      \X(s)\,ds\right)(A)=\int_{x^{-1}(A)}2c(U)\V_3(\X)\dot\X\,ds=0
  \end{equation*}
  and the measure
  $x_\#\left(2c(U(s))\V_3(\X(s))\dot
    \X(s)\,ds\right)$ is absolutely continuous. In
  the same way, one proves that
  $x_\#\left(-2c(U(s))\W_3(\Y(s))\dot
    \Y(s)\,ds\right)$ is absolutely continuous and
  $R(x)$ and $S(x)$ as given by
  \eqref{eq:relufromG3} and \eqref{eq:relufromG4}
  are well-defined. We have
  \begin{equation}
    \label{eq:rewpusfw}
    \int_{x^{-1}(A)}R(x(s))\dot x(s)\,ds=\int_{x^{-1}(A)}2c(U(s))\V_3(\X(s))\dot \X(s)\,ds
  \end{equation}
  for any measurable set $A$. For any measurable
  set $B$, we have the decomposition
  $x^{-1}(x(B))=B\cup(B^c\cap x^{-1}(x(B)))$. Let
  prove that the set $B^c\cap x^{-1}(x(B))$ has
  measure zero with respect to $\dot x(s)\,ds$. We
  consider a point $\bar s\in B^c\cap
  x^{-1}(x(B))$, there exists $s\in B$ such that
  $x(\bar s)=x(s)$. Since $x$ is increasing, it
  implies that $\dot x(\bar s)=0$ and therefore
  $\int_{B^c\cap x^{-1}(x(B))}\dot x(s)\,ds=0$ so
  that the set $B^c\cap x^{-1}(x(B))$ has zero
  measure with respect to $\dot x(s)\,ds$. Since
  $\V_3(\X(s))\dot\X(s)\,ds$ is absolutely
  continuous with respect to $\dot x(s)\,ds$, it
  implies that $\int_{B^c\cap
    x^{-1}(x(B))}\V_3(\X(s))\dot\X(s)\,ds=0$. Hence,
  taking $A=x(B)$ in \eqref{eq:rewpusfw}, we
  obtain
  \begin{equation*}
    \int_{B}R(x(s))\dot x(s)\,ds=\int_{B}2c(U(s))\V_3(\X(s))\dot \X(s)\,ds
  \end{equation*}
  for any Borel set $B$. Hence,
  \begin{equation}
    \label{eq:relRSfromG1}
    R(x(s))\dot x(s)=2c(U(s))\V_3(\X(s))\dot \X(s)
  \end{equation}
  which yields
  \begin{equation}
    \label{eq:relRSfromG2}
    R(x(s))\V_2(\X(s))=c(U(s))\V_3(\X(s))
  \end{equation}
  after simplifying by $\dot\X(s)$. Similarly, we
  obtain \eqref{eq:reluS}. 

\textbf{Step 4.}  We show that $R$ and $S$ belong to $L^2$
    and $u_x=\frac{R-S}{2c}$. We have
  \begin{align*}
    \int_\Real R^2(x)\,dx&=\int_\Real
    R^2(x(s))\dot x\,ds\\
    &=2\int_\Real
    R^2(x(s))\V_2(\X(s))\dot\X\,ds\\
    &=2\int_{\{s\in\Real\mid
      \V_2(\X(s))>0\}}\frac{(R(x(s))\V_2(\X(s)))^2}{\V_2(\X(s))}\dot\X\,ds\\
    &\leq 4\int_\Real\V_4(\X)\dot\X\,ds&\text{(by
      \eqref{eq:reluR} and \eqref{eq:relVW})}\\
    &\leq 4\norm{J}_{L^\infty(\Real)}\leq
    4\tnorm{\Theta}_{\G}
  \end{align*}
  because $\dot
  J(s)=\V_4(\X)\dot\X+\W_4(\Y)\dot\Y$ and $\V_4$
  and $\W_4$ are positive. Hence $R\in L^2$ and,
  similarly, one proves that $S\in L^2$. For any
  smooth function $\phi$ with compact support, we
  have
  \begin{equation*}
    \int_{\Real}u\phi_x\,dx=\int_{\Real}u(x(s))\phi_x(x(s))\dot x(s)\,ds=\int_\Real U(s)\phi(x(s))_s\,ds.
  \end{equation*}
  After integrating by parts, it yields
  \begin{align*}
    \int_{\Real}u\phi_x\,dx=-\int_{\Real}\dot U(s)\phi(x(s))\,ds&=\int_\Real(\V_3(\X)\dot\X+\W_3(\Y)\dot\Y)\phi(x(s))\,ds\\
    &=\int_\Real\frac1{c(U)}(R(x(s))\V_2(\X)\dot\X-S(x(s))\W_2(\Y)\dot\Y)\phi(x(s))\,ds\\
    &=\int_\Real\frac1{2c(U)}(R(x(s))-S(x(s))\phi(x(s))\dot
    x\,ds\quad\text{ (by \eqref{eq:dotxteqz})}\\
    &=\int_\Real\frac{R-S}{2c(u)}\phi\,dx,
  \end{align*}
  after a change of variables. Hence,
  $u_x=\frac{R-S}{2c(u)}$ in the sense of
  distribution. 

  \textbf{Step 5.} We show that
  $\muac=\frac14 R^2dx$ and
  $\nuac=\frac14 S^2dx$. Let
  \begin{equation}
    \label{eq:defsetA}
    A=\{s\mid \V_2(\X(s))>0\}\quad\text{ and }B=(x(A^c))^c.
  \end{equation}
  We have
  \begin{equation*}
    \meas(B^c)=\int_{A^c}\dot x\,ds=0
  \end{equation*}
  because $\dot x=0$ almost everywhere on $A^c$,
  by \eqref{eq:dotxteqz}. The set $B$ has
  therefore full measure. We have
  $x^{-1}(B)\subset A$. Indeed, for any $s\in
  x^{-1}(B)$, we have $x(s)\neq x(\bar s)$ for all
  $\bar s\in A^c$. For any measurable subset
  $B'\subset B$, we have, by definition of $\mu$,
  that
  \begin{equation}
    \label{eq:mubpr}
    \mu(B')=\int_{x^{-1}(B')}\V_4(\X(s))\dot\X(s)\,ds.
  \end{equation}
  Hence, since $x^{-1}(B')\subset A$,
  \begin{align*}
    \mu(B')&=\int_{x^{-1}(B')}\frac{\V_4(\X)\V_2(\X)}{\V_2(\X)}\dot\X\,ds\\
    &=\int_{x^{-1}(B')}\frac{c^2(U)\V_3^2(\X)}{2\V_2(\X)}\dot\X\,ds&\text{ (by \eqref{eq:relVW}) }\\
    &=\frac14\int_{x^{-1}(B')}R^2(x(s))\dot x(s)\,ds&\text{ (by \eqref{eq:reluR} and \eqref{eq:dotxteqz}) }\\
    &=\frac14\int_{B'}R^2\,dx
  \end{align*}
  and therefore
  $\muac=\frac14 R^2\,dx$. Similarly, one proves
  that $\nuac=\frac14 S^2\,dx$.
\end{proof}
\begin{lemma}
  \label{lem:relURS}
  Given $(u_0,R_0,S_0,\mu_0,\nu_0)\in\D$, let us
  denote $(u,R,S,\mu,\nu)(T)=
  \Mb\circ S_T\circ\Lb (u_0,R_0,S_0,\mu_0,\nu_0)$ and
  $Z=\Sb\circ\Lb(u_0,R_0,S_0,\mu_0,\nu_0)$. Then,
  we have
  \begin{equation}
    \label{eq:reluU}
    u(t(X,Y),x(X,Y))=U(X,Y)
  \end{equation}
  for all $(X,Y)\in\Real^2$, and 
  \begin{subequations}
    \label{eq:relRUXY}
    \begin{align}
      \label{eq:relRUXY1}
      R(t(X,Y),x(X,Y))x_X(X,Y)&=c(U(X,Y))U_X(X,Y),\\
      \label{eq:relRUXY2}
      S(t(X,Y),x(X,Y))x_Y(X,Y)&=-c(U(X,Y))U_Y(X,Y),
    \end{align}
  \end{subequations}
  for almost all $(X,Y)\in\Real^2$ such that
  $x_X(X,Y)>0$ and $x_Y(X,Y)>0$. We have
  \begin{equation}
    \label{eq:weakderu}
    u_t=\frac12(R+S)\quad\text{ and }\quad u_x=\frac1{2c(u)}(R-S)
  \end{equation}
  in the sense of distributions.
\end{lemma}
\begin{proof}
  We consider a solution $Z\in\H$. Given
  $(X,Y)\in\Real^2$, let us denote $\bar t=t(X,Y)$
  and $\bar x=x(X,Y)$. Let
  $(\X,\Y,\Z,\V,\W)=\Eb\circ\tb_{\bar t}(Z)$. By
  definition, we have $t(\X(s),\Y(s))=\bar t$,
  and, slightly abusing the notation,
  $x(s)=\Z_2(s)=x(\X(s),\Y(s))$ and
  $U(s)=\Z_3(s)=U(\X(s),\Y(s))$ for all
  $s\in\Real$. By Lemma \ref{lem:revufromG}, we
  have $u(\bar t,\bar x)=U(s)$ for any $s$ such
  that $x(s)=\bar x$. It implies that, for any
  $\bar s$ such that
  \begin{subequations}
    \label{eq:tXbart}
    \begin{equation}
      \label{eq:tXbart1}
      t(\X(\bar s),\Y(\bar s))=\bar t=t(X,Y)
  \end{equation}
  and
  \begin{equation}
    \label{eq:tXbart2}
    x(\X(\bar s),\Y(\bar s))=\bar x=x(X,Y),
  \end{equation}
  \end{subequations}
  we have
  \begin{equation*}
    u(\bar t,\bar x)=U(\X(\bar s),\Y(\bar s)).
  \end{equation*}
  Then, \eqref{eq:reluU} will be proved once we
  have proved that
  \begin{equation}
    \label{eq:Uconst}
    U(\X(\bar s),\Y(\bar
    s))=U(X,Y).    
  \end{equation}
  Let us prove that when \eqref{eq:tXbart} hold,
  then either $(X,Y)=(\X(\bar s),\Y(\bar s))$ or
  \begin{equation}
    \label{eq:derxUvan}
    x_X=x_Y=U_X=U_Y=0,
  \end{equation}
  in the rectangle with corners at $(X,Y)$ and
  $(\X(\bar s),\Y(\bar s))$, so that
  \eqref{eq:Uconst} holds in both cases. We
  consider first the case where $\X(\bar s)\leq X$
  and $\Y(\bar s)\leq Y$. Since the function $x$
  is increasing in the $X$ and $Y$ directions, we
  must have $x_X=x_Y=0$ in the rectangle $[\X(\bar
  s),X]\times [\Y(\bar s),Y]$ and, by
  \eqref{eq:energrel2}, $U_X=U_Y=0$ in the same
  rectangle so that $U$ is constant and we have
  proved \eqref{eq:Uconst}. In the case where
  $\X(\bar s)\leq X$ and $\Y(\bar s)\geq Y$, since
  the function $t$ is increasing in the $X$
  direction and decreasing in the $Y$ direction,
  it follows that $t_X=t_Y=0$ in the rectangle
  $[\X(\bar s),X]\times [Y,\Y(\bar s)]$. Hence,
  $x_X=x_Y=0$ and, as before, we prove
  \eqref{eq:Uconst}. The other cases can be
  treated in the same way and this concludes the
  proof of \eqref{eq:Uconst} and therefore
  \eqref{eq:reluU} holds. Let us prove
  \eqref{eq:relRUXY1}. By \eqref{eq:reluR} and the
  definition of $\Eb$, we get
  \begin{equation*}
    R(\bar t,x(s))x_X(\X(s),\Y(s))=c(U(\X(s),\Y(s)))U_X(\X(s),\Y(s))
  \end{equation*}
  so that
  \begin{equation}
    \label{eq:RtxbarX}
    R(t(X,Y),x(X,Y))x_X(\X(\bar s),\Y(\bar s))=c(U(\X(\bar s),\Y(\bar s)))U_X(\X(\bar s),\Y(\bar s))
  \end{equation}
  for any $\bar s$ such that \eqref{eq:tXbart}
  holds. We have proved that when
  \eqref{eq:tXbart} is satisfied, then either
  $(X,Y)=(\X(\bar s),\Y(\bar s))$ or
  \eqref{eq:derxUvan} holds. Hence,
  \eqref{eq:relRUXY1} follows from
  \eqref{eq:RtxbarX}. Similarly, one proves
  \eqref{eq:relRUXY2}. For any smooth function
  $\phi(t,x)$ with compact support, we have
  \begin{multline}
    \label{eq:intuphit}
    \int_{\Real^2}u(t,x)\phi_t(t,x)=\\\int_{\Real^2}u(t(X,Y),x(X,Y))\phi_t(t(X,Y),x(X,Y))(t_Xx_Y-x_Xt_Y)\,dXdY
  \end{multline}
  where we have used \eqref{eq:reltx} and
  \eqref{eq:reluU}.  By differentiating the
  function $\phi(t(X,Y),x(X,Y))$ with respect to
  $X$ and $Y$, we get that
  \begin{equation*}
    \phi(t(X,Y),x(X,Y))_Xx_Y-\phi(t(X,Y),x(X,Y))_Yx_X=\phi_t(t(X,Y),x(X,Y))(t_Xx_Y-t_Yx_X).
  \end{equation*}
  We insert this identity in \eqref{eq:intuphit} and
  obtain, after integrating by parts,
  \begin{align*}
    \int_{\Real^2}u(t,x)\phi_t(t,x)\,dtdx&=-\int_{\Real^2}((Ux_Y)_X-(Ux_X)_Y)(X,Y)\phi(t(X,Y),x(X,Y))\,dXdY\\
    &=-\int_{\Real^2}((U_Xx_Y-U_Yx_X)(X,Y)\phi(t(X,Y),x(X,Y))\,dXdY.
  \end{align*}
  We use \eqref{eq:relRUXY} and get
  \begin{align*}
    \int_{\Real^2}u(t,x)\phi_t(t,x)\,dtdx&=-\int_{\Real^2}\Big(\big(\frac{R+S}{c(u)}\phi\big)\circ(t,x)\,x_Xx_Y\,dXdY\\
    &=-\int_{\Real^2}\big(\frac12(R+S)2\phi\big)\circ(t,x)\,(t_Xx_Y-tx_Yx_X)\,dXdY\\
    &=-\int_{\Real^2}\frac12(R+S)(t,x)\phi(t,x)\,dtdx.
  \end{align*}
  This proves the first identity in
  \eqref{eq:weakderu}; the second one is proven in the
  same way.
\end{proof}
We can now define the semigroup mapping $\bar S_T$
on $\D$, the original set of variables.
\begin{definition}
  \label{def:bST}
  For any $T>0$, let $\bar S_T:\D\to\D$ be defined
  as
  \begin{equation*}
    \bar S_T=\Mb\circ S_T\circ \Lb
  \end{equation*}
\end{definition}
Given $(u_0,R_0,S_0,\mu_0,\nu_0)\in\D$, let us
denote $(u,R,S,\mu,\nu)(t)=\bar
S_t(u_0,R_0,S_0,\mu_0,\nu_0)$. In the theorem that
follows, we prove that $u(t,x)$ is a weak solution
of the nonlinear wave equation. However it is not
clear if $\bar S_T$ is a semigroup. Indeed, we have
\begin{equation*}
  \bar S_{T}\circ\bar S_{T'}=\Mb\circ S_T\circ \Lb\circ\Mb\circ S_{T'}\circ \Lb.
\end{equation*}
By the semigroup property of $S_T$, it would
follow immediately that $\bar S_{T}$ is also a
semigroup if we had $\Lb\circ\Mb=\id$, but this
identity does not hold in general. It is the aim
of the last section to show that $\bar S_{T}$ is a
semigroup.

\section{Relabeling symmetry}
\label{sec:relasym}

We consider the set of transformations of the
$\Real^2$-plane given by
\begin{equation*}
  (X,Y)\mapsto(f(X),g(Y))
\end{equation*}
for any $(f,g)\in\Gr^2$, where $\Gr$ is the group
of diffeomorphisms on the line, see Definition
\ref{def:Gr}. It is a subgroup of the group of
diffeomorphisms of $\Real^2$. Such transformations
let the characteristics lines invariant. Indeed,
vertical and horizontal lines, which correspond to
the characteristics  in our new sets of
coordinates, remain vertical and horizontal lines
through this mapping. In this section, we show
that the subgroup $\Gr^2$ plays an essential role
by exactly capturing the degree of freedom we have
introduced when changing coordinates and
introduced the equivalent system
\eqref{eq:goveq}. Given $f$ and $g$ in $\Gr$, the
$\Real^2$ plane is stretched in the $X$ and $Y$
direction by the transformations $X\mapsto f(X)$
and $Y\mapsto g(Y)$. The solutions of
\eqref{eq:goveq} are preserved and we can define
an action of $\Gr^2$ on the set of solutions $\H$.

\begin{definition}
  \label{def:actionH} For any $Z\in\H$, $f$ and
  $g$ in $\Gr$, we define $\bar Z\in\H$ as
  \begin{equation}
    \label{eq:defactHZ}
    \bar Z(X,Y)=Z(f(X),g(Y)).
  \end{equation}
  The mapping from $\H\times\Gr^2$ to $\H$ given
  by $Z\times(f,g)\mapsto\bar Z$ defines an action of
  the group $\Gr^2$ on $\H$ and we denote $\bar
  Z=Z\act(f,g)$.
\end{definition}
\begin{proof}[Proof of well-posedness of Definition \ref{def:actionH}]
  For any $\Omega$, given $Z\in\H(\Omega)$ and
  $(f,g)\in\Gr^2$, let $\bar X_l=f(X_l)$,
  $X_r=f(X_r)$, $\bar Y_l=g(Y_l)$, $\bar
  Y_r=g(Y_r)$ and
  \begin{equation*}
    \bar\Omega=[\bar X_l,\bar X_r]\times[\bar Y_l,\bar Y_r].
  \end{equation*}
  Let us prove that $\bar Z\in\H(\bar\Omega)$. We have
  \begin{equation*}
    \bar Z_X(X,Y)=f'(X)Z_X(f(X),g(Y)),\quad\bar Z_Y(X,Y)=g'(Y)Z_X(f(X),g(Y)).
  \end{equation*}
  and
  \begin{equation*}
    \bar Z_{XY}(X,Y)=f'(X)g'(Y)Z_{XY}(f(X),g(Y)).
  \end{equation*}
  By using the linearity of the mapping $F(Z)$ in
  \eqref{eq:condgoveq}, we get
  \begin{align*}
    \bar Z_{XY}&=f'g'Z_{XY}(f,g)\\
    &=f'g'F(Z(f,g))\big(Z_X(f,g),Z_Y(f,g)\big)\\
    &=F(Z(f,g))\big(f'Z_X(f,g),g'Z_Y(f,g)\big)&\text{(by
      the linearity of $F(Z)$)}\\
    &=F(\bar Z)(\bar Z_X,\bar Z_Y)
  \end{align*}
  and $\bar Z$ is a solution of \eqref{eq:goveq}.
  Since $f$ and $g$ belong to $\Gr$, there exists
  $\delta>0$ such that $f'(X)>\delta$ for
  a.e. $X\in\Real$ and $g'(Y)>\delta$ for
  a.e. $Y\in\Real$, see Lemma \ref{lem:charH}. We
  have to check that $\bar Z$ fulfills
  \eqref{eq:relpres}. It is not difficult to do so
  once one has observed that the equalities and
  inequalities in \eqref{eq:relpres} enjoy the
  required homogeneity properties. For example, we
  have
  \begin{equation*}
    2\bar J_X\bar x_X=2f'^2J_X(f,g)x_X(f,g)=f'^2(c(U(f,g))U_X(f,g))^2=(c(\bar U)\bar U_X)^2
  \end{equation*}
  and 
  \begin{equation*}
    \bar x_X=f'x_X(f,g)\geq0.
  \end{equation*}
  We will prove that $\bar Z$ fulfills the
  condition (ii) in Definition \ref{def:H}
  after we have introduced the action of $\Gr^2$
  on $\G$.
\end{proof}
We can define an action on $\C$ as follows. This
action corresponds to a stretching of the curve in
the $X$ and $Y$ direction.
\begin{definition}
  \label{def:actionC}
  Given $(\X,\Y)\in\C$, we define
  $(\bar\X,\bar\Y)\in\C$ such that
  \begin{align}
    \label{eq:defactXY}
    \bar\X=f^{-1}\circ\X\circ
    h&&\bar\Y=g^{-1}\circ\Y\circ h
  \end{align}
  where $h$ is the re-normalizing function which
  yields $\bar\X+\bar\Y=2\id$, that is,
  \begin{equation}
    \label{eq:defhren}
    (f^{-1}\circ\X+g^{-1}\circ\Y\big)\circ h=2\id.
  \end{equation}
  We denote $(\bar\X,\bar\Y)=(\X,\Y)\act(f,g)$.
\end{definition}
\begin{proof}[Proof of wellposedness of Definition \ref{def:actionC}]
  Let us denote
  $v=f^{-1}\circ\X+g^{-1}\circ\Y$. We have
  $v-\id\in\Winf(\Real)$ because $f^{-1}-\id$,
  $g^{-1}-\id$, $\X-\id$ and $\Y-\id$ all belong
  to $\Winf(\Real)$. There exists $\delta>0$ such
  that $(f^{-1})'\geq\delta$ and
  $(g^{-1})'\geq\delta$ a.e. and therefore
  $\dot{v}\geq\delta(\dot\X+\dot\Y)=2\delta$. Hence,
  by Lemma \ref{lem:charH}, we have that $v$ is
  invertible so that $h$ exists and
  $h-\id\in\Winf(\Real)$. One proves then easily
  that $(\bar\X,\bar\Y)\in\C$.
\end{proof}
We define the action on $\G$ so that it commutes
with the $\bullet$ operation and the actions on
$\H$ and $\C$, that is,
\begin{equation}
  \label{eq:comactbule}
  (Z\bullet\Gamma)\act\phi=(Z\act\phi)\bullet(\Gamma\act\phi)
\end{equation}
for any $Z\in\H$, $\Gamma=(\X,\Y)\in\C$ and
$\phi=(f,g)\in\Gr^2$. We obtain the following
definition.
\begin{definition}
  \label{def:actionG}
  For any $\Theta=(\X,\Y,\Z,\V,\W)\in\G$ and
  $f,g\in\Gr$, we define $\bar\Theta=(\bar \X,\bar
  \Y,\bar\Z,\bar \V,\bar \W)\in\G$ as follows
  \begin{equation*}
    (\bar\X,\bar\Y)=(\X,\Y)\act(f,g)
  \end{equation*}
  and
  \begin{align}
    \label{eq:defactGVW}
    \bar\V(X)=f'(X)\V(f(X))&&\bar\W(Y)=g'(Y)\W(g(Y))
  \end{align}
  and
  \begin{equation}
    \label{eq:defactGz}
    \bar\Z=\Z\circ h
  \end{equation}
  where $h$ is given by \eqref{eq:defhren}. The
  mapping from $\G\times\Gr^2$ to $\G$ given by
  $\Theta\times(f,g)\mapsto\bar\Theta$ defines an
  action of the group $\Gr^2$ on $\G$ that we
  denote $\bar\Theta= \Theta\act(f,g)$.
\end{definition}
To check that this definition is well-posed, we
have to check that $\bar\Theta\in\G$. This can be
done without any special difficulty and we omit
the details here. Let us however prove
\eqref{eq:comactbule} in details as we will use it
several times in the following. For any $Z\in\H$,
$\Gamma=(\X,\Y)\in\C$ and $\phi=(f,g)\in\Gr^2$, we
denote $\Theta=(\X,\Y,\Z,\V,\W)=Z\bullet\Gamma$,
$\bar\Theta=(\bar\X,\bar\Y,\bar\Z,\bar\V,\bar\W)=\Theta\act\phi$,
$\bar\Gamma=(\bar\X,\bar\Y)=\Gamma\act\phi$ and
$\tilde\Theta=\bar Z\act\bar\Gamma$. We want to
prove that $\tilde\Theta=\bar\Theta$. We have
$\Z(s)=Z(\X(s),\Y(s))$ and, by
\eqref{eq:defactGz}, $\bar\Z(s)=Z(\X\circ
h(s),\Y\circ h(s))$. Hence, by
\eqref{eq:defactXY},
\begin{equation*}
  \tilde\Z=\bar Z(\bar\X,\bar\Y)=Z(f\circ\bar\X,g\circ\bar\Y)=Z(\X\circ h,\Y\circ h)=\bar\Z.
\end{equation*}
We have $\V(\X(s))=Z_X(\X(s),\Y(s))$ and, by
\eqref{eq:defactGVW}, 
\begin{multline*}
  \bar\V(\bar\X(s))=f'(\bar\X(s))\V(f\circ\bar\X(s))=f'(\bar\X(s))\V(\X\circ
  h(s))\\=f'(\bar\X(s))Z_X(\X\circ h(s),\Y\circ
  h(s))
\end{multline*}
and
\begin{multline*}
  \tilde\V(\bar\X(s))=\bar
  Z_X(\bar\X(s),\bar\Y(s))=f'(\X(s))Z_X(f\circ\bar\X(s),g\circ\bar\Y(s))\\=f'(\bar\X(s))Z_X(\X\circ
  h(s),\Y\circ h(s)).
\end{multline*}
Hence, $\tilde\V=\bar\V$. Similarly one proves
that $\tilde\W=\bar\W$, which concludes the proof
of \eqref{eq:comactbule}. 
\begin{proof}[End of proof of well-posedness of
  Definition \ref{def:actionH}]
  For any $\phi\in\Gr^2$ and $Z\in\H$, it remains
  to prove that $\bar Z$, as defined by
  \eqref{eq:defactHZ}, fulfills the condition
  (ii) in Definition \ref{def:H}. By this
  same condition, for any $Z\in\H$, there exists a
  curve $\Gamma=(\X,\Y)\in\G$ such that
  $Z\bullet\Gamma\in\G$. From
  \eqref{eq:comactbule}, it follows that, for the
  curve $\bar\Gamma=\Gamma\act\phi$, we have $\bar
  Z\bullet\bar\Gamma\in\G$ because
  $(Z\bullet\Gamma)\act\phi\in\G$ and therefore
  $\bar Z$ fulfills the condition (ii) in
  Definition \ref{def:H}.
\end{proof}
\begin{definition}
  \label{def:actionF}
  For any $\psi=(\psi_1,\psi_2)\in\F$ and
  $f,g\in\Gr$, we define
  $\bar\psi=(\bar\psi_1,\bar\psi_2)\in\F$ as
  follows
  \begin{align*}
    \bar x_1(X)&=x_1(f(X)),&\bar U_1(X)&=U_1(f(X)),&\bar J_1(X)&=J_1(f(X)),\\
    \bar x_2(Y)&=x_2(g(Y)),&\bar U_2(Y)&=U_2(g(Y)),&\bar J_2(Y)&=J_2(g(Y)),
  \end{align*}
  and
  \begin{align*}
    \bar V_1(X)=V_1(f(X))f'(X),&&    \bar V_2(Y)=V_2(g(Y))g'(Y).
  \end{align*}
  The mapping from $\F\times\Gr^2$ to $\F$ given
  by $\psi\times(f,g)\mapsto\bar\psi$ defines an
  action of the group $\Gr^2$ on $\F$, and we
  denote
  \begin{equation*}
    \bar\psi=\psi\act
\phi.
  \end{equation*}
\end{definition}
\begin{proof}[Proof of well-posedness of
  Definition \ref{def:actionF}] We have to check
  that $\bar\psi=(\bar\psi_1,\bar\psi_2)\in\F$. We
  will only check that the identities
  \eqref{eq:condcurvF} in the definition
  \ref{def:F} of $\F$ are fulfilled, as the other
  properties can be checked without
  difficulty. For any curve $(\bar\X,\bar\Y)\in\C$
  such that $\bar x_1(\bar\X)=\bar x_2(\bar\Y)$,
  let $(\X,\Y)=(\bar\X,\bar\Y)\act(f,g)$. We have
  \begin{equation*}
    x_1(\X(s))=\bar x_1\circ f^{-1}\circ\X(s)=\bar x_1\circ\bar\X\circ h^{-1}(s)=\bar x_2\circ\bar\Y\circ h^{-1}(s)=\bar x_2\circ f^{-1}\circ\Y(s)=x_2(\Y(s))
  \end{equation*}
  and therefore, since $\psi\in\F$,
  $U_1(\X(s))=U_2(\Y(s))$ for all $s\in\Real$,
  which implies
  \begin{equation*}
    \bar U_1(\bar\X)=U_1\circ\X\circ h=U_2\circ\Y\circ h=\bar U_2(\bar\Y)
  \end{equation*}
  and this proves \eqref{eq:U1U2equal} for
  $\bar\psi$. Similarly, one proves that
  \eqref{eq:condU1U_2} holds for $\bar\psi$.
\end{proof}
In the following lemma, we show that all the
mappings given in \eqref{eq:diagmappings} are
equivariant with respect to the action of the
group $\Gr^2$.
\begin{lemma}
  The mappings $\Eb$, $\tb_T$, $\Cb$, $\Db$ and
  $\Cb$ are $\Gr^2$-equivariant, that is, for all
  $\phi=(f,g)\in\Gr^2$,
  \begin{subequations}
    \label{eq:equivariants}
    \begin{equation}
      \label{eq:equivariants1}
      \Eb(Z\cdot\phi)=\Eb(Z)\cdot\phi,
    \end{equation}
    \begin{equation}
      \label{eq:equivariants2}
      \tb_T(Z\cdot\phi)=\tb_T(Z)\cdot\phi
    \end{equation}
    for all $Z\in\H$ and
    \begin{equation}
      \label{eq:equivariants3}
      \Sb(\Theta\cdot\phi)=\Sb(\Theta)\cdot\phi
    \end{equation}
    for all $\Theta\in\G$ and
    \begin{equation}
      \label{eq:equivariants4}
      \Db(\Theta\cdot\phi)=\Db(\Theta)\cdot\phi
    \end{equation}
    for all $\Theta\in\G_0$ and
    \begin{equation}
      \label{eq:equivariants5}
      \Cb(\psi\cdot\phi)=\Cb(\psi)\cdot\phi
    \end{equation}
  \end{subequations}
  for $\psi\in\F$. Therefore $S_T$ is
  $\Gr^2$-equivariant, that is,
  \begin{equation}
    \label{eq:equivariantsST}
    S_T(\psi\cdot\phi)=S_T(\psi)\cdot\phi
  \end{equation}
  for all $\psi\in\F$.
\end{lemma}
\begin{proof}
  Let us prove \eqref{eq:equivariants1}. We denote
  $\Theta=(\X,\Y,\Z,\V,\W)=\Eb(Z)$, $\bar
  Z=Z\act\phi$,
  $\bar\Theta=(\bar\X,\bar\Y,\bar\Z,\bar\V,\bar\W)=\Eb(\bar
  Z)$ and
  $\tilde\Theta=(\tilde\X,\tilde\Y,\tilde\Z,\tilde\V,\tilde\W)=\Theta\act\phi$. We
  want to prove that
  $\tilde\Theta=\bar\Theta$. First we prove that
  $\tilde\Gamma=\bar\Gamma$ where
  $\tilde\Gamma=(\tilde\X,\tilde\Y)$ and
  $\bar\Gamma=(\bar\X,\bar\Y)$. By definition
  \eqref{eq:defXsE}, we have
  \begin{equation*}
    \bar\X(s)=\sup\{X\in\Real\mid t(f(X'),g(Y'))<0\quad\text{for all }X'<X\text{ and }Y'\text{ such that }X'+Y'=s\}.
  \end{equation*}
  We have 
  \begin{align*}
    t(f(\tilde\X(s)),g(\tilde\Y(s)))=t(\X\circ
    h(s),\Y\circ h(s))=0,
  \end{align*}
  by the definition of $(\X,\Y)$ given by
  \eqref{eq:defXsE} which implies that
  $t(\X(s),\Y(s))=0$ for all $s\in\Real$. Hence,
  \begin{equation}
    \label{eq:barxinftilx}
    \bar\X\leq\tilde\X.
  \end{equation}
  Assume that $\bar\X(s)<\tilde\X(s)$ for some
  point $s$. We have
  $t(f(\bar\X(s)),g(\bar\Y(s)))=0$ and due to the
  monotonicity of $t$, it implies that $t(X,Y)=0$
  for all
  $(X,Y)\in[f\circ\bar\X(s),f\circ\tilde\X(s)]\times[g\circ\tilde\Y(s),g\circ\bar\Y(s)]$. If
  $f\circ\bar\X(s)\leq2h(s)-g\circ\bar\X(s)$, we
  obtain a contradiction. Indeed, if we set
  $X'=2h(s)-g\circ\bar\Y(s)$ and
  $Y'=g\circ\bar\Y(s)$, then we have
  \begin{equation*}
    X'<2h(s)-g\circ\tilde\Y(s)=2h(s)-\Y\circ h(s)=\X\circ h(s)
  \end{equation*}
  so that $t(X',Y')=0$ and $X'+Y'=2h(s)$, which
  contradicts the definition \eqref{eq:defXsE} of
  $(\X,\Y)$ at $h(s)$. If
  $f\circ\bar\X(s)>2h(s)-g\circ\bar\X(s)$, then we
  set $X'=f\circ\bar\X(s)<\X\circ
  h(s)=f\circ\tilde\X(s)$ and
  $Y'=2h(s)-f\circ\bar\X(s)$. We have $t(X',Y')=0$
  and $X'+Y'=2h(s)$, which also leads to a
  contradiction of \eqref{eq:defXsE}. Hence, we
  have proved that $\bar\X=\tilde\X$ and therefore
  $\bar\Gamma=\tilde\Gamma$. It means that
  \begin{equation*}
    \Eb(Z\cdot\phi)=(Z\act\phi)\bullet\bar\Gamma=(Z\act\phi)\bullet\tilde\Gamma=(Z\act\phi)\bullet(\Gamma\act\phi).
  \end{equation*}
  Hence, by \eqref{eq:comactbule}, it yields
  \begin{equation*}
    \Eb(Z\cdot\phi)=(Z\bullet\Gamma)\act\phi=\Eb(Z)\act\phi,
  \end{equation*}
  and we have proved \eqref{eq:equivariants1}. The
  identity \eqref{eq:equivariants2} follows
  directly from the definition of $\tb_T$. Let us
  prove \eqref{eq:equivariants3}. For any
  $\Theta=(\X,\Y,\Z,\V,\W)$, we denote
  $Z=S(\Theta)$, $\bar Z=S(\Theta\act\phi)$,
  $\tilde Z=Z\act\phi$. We want to prove that
  $\tilde Z=\bar Z$. By definition of the solution
  operator $S$, we have
  \begin{equation*}
    \bar Z\bullet(\Gamma\act\phi)=\Theta\act\phi,
  \end{equation*}
  and
  \begin{equation*}
    \tilde Z\bullet(\Gamma\act\phi)=(Z\act\phi)\bullet(\Gamma\act\phi)=(Z\bullet\Gamma)\act\phi=\Theta\act\phi,
  \end{equation*}
  by \eqref{eq:comactbule}. Hence, $\bar Z$ and
  $\tilde Z$ are solutions that match the same
  data on a curve. Since the solution is unique by
  Theorem \ref{th:globalsol}, we get $\bar
  Z=\tilde Z$. The property
  \eqref{eq:equivariants4} follows directly from
  the definitions. Let us prove
  \eqref{eq:equivariants5}. For any
  $\psi=(\psi_1,\psi_2)\in\F$, $\phi\in\Gr^2$, we
  denote
  $\bar\psi=(\bar\psi_1,\bar\psi_2)=\psi\act\phi$,
  $\Theta=(\X,\Y,\Z,\V,\W)=\Cb(\psi)$,
  $\bar\Theta=(\bar\X,\bar\Y,\bar\Z,\bar\V,\bar\W)=\Cb(\bar\psi)$
  and $\tilde\Theta=\Theta\act\phi$. We want to
  prove that $\tilde\Theta=\bar\Theta$. First, we
  prove that $\tilde\Gamma=\bar\Gamma$ where
  $\tilde\Gamma=(\tilde\X,\tilde\Y)$ and
  $\bar\Gamma=(\bar\X,\bar\Y)$. By definition
  \eqref{eq:defXsE}, we have
  \begin{multline}
    \label{eq:defXbar}
    \bar\X(s)=\sup\{X\in\Real\mid x_1\circ
    f(X')<x_2\circ g(Y'),\\\text{for all
    }X'<X\text{ and }Y'\text{ such that
    }X'+Y'=s\},
  \end{multline}
  and
  \begin{equation*}
    \X(s)=\sup\{X\in\Real\mid x_1(X')<x_2(Y'),\text{for all
    }X'<X\text{ and }Y'\text{ such that
    }X'+Y'=s\}.
  \end{equation*}
  By continuity, we have $x_1(\X(s))=x_2(\Y(s))$
  so that $x_1\circ\X\circ h=x_2\circ\Y\circ
  h$. Hence, $x_1\circ f\circ\tilde\X=x_2\circ
  g\circ\tilde\Y$ and it implies, by
  \eqref{eq:defXbar}, that
  \begin{equation*}
    \bar\X\leq\tilde\X.
  \end{equation*}
  The proof then resembles to what was done above
  after \eqref{eq:barxinftilx}. Let us assume that
  $\bar\X(s)<\tilde\X(s)$ for some point $s$. We
  have $f(\bar X(s))<f(\tilde\X(s))$ and $g(\tilde
  Y(s))<g(\bar\Y(s))$. By using the monotonicity
  of $x_1$ and $x_2$, we get
  \begin{equation*}
    x_1\circ f\circ\bar\X(s)\leq x_1\circ f\circ\tilde\X(s)=x_2\circ g\circ\tilde\Y(s)\leq x_2\circ g\circ\bar\Y(s)=x_1\circ f\circ\bar\X(s).
  \end{equation*}
  Hence, $x_1\circ f\circ\bar\X(s)=x_1\circ
  f\circ\tilde\X(s)$ and $x_2\circ
  g\circ\tilde\Y(s)=x_2\circ
  g\circ\bar\Y(s)$. Since $x_1$ and $x_2$ are
  decreasing, it follows that $x_1$ and $x_2$ are
  constant on
  $[f\circ\bar\X(s),f\circ\tilde\X(s)]$ and
  $[g\circ\tilde\Y(s),g\circ\bar\Y(s)]$,
  respectively. If
  $f\circ\bar\X(s)\leq2h(s)-g\circ\bar\X(s)$, then
  we obtain a contradiction. Indeed, let us set
  $X'=2h(s)-g\circ\bar\Y(s)$ and
  $Y'=g\circ\bar\Y(s)$, we have
  \begin{equation*}
    X'<2h(s)-g\circ\tilde\Y(s)=2h-\Y\circ h(s)=\X\circ h(s)
  \end{equation*}
  so that $x_1(X')=x_2(Y')=0$ and $X'+Y'=2h(s)$,
  which contradicts the definition
  \eqref{eq:defXs} of $(\X,\Y)$ at $h(s)$. If
  $f\circ\bar\X(s)>2h(s)-g\circ\bar\X(s)$, then we
  have
  \begin{equation*}
    x_1(f\circ\bar\X(s))=x_2(2h(s)-f\circ\bar\X(s)),
  \end{equation*}
  which, in the same way, leads to a contradiction
  of \eqref{eq:defXs}. Thus we have proved that
  $\tilde\Gamma=\bar\Gamma$. We then prove that
  $\tilde\Z=\bar\Z$, $\tilde\V=\bar\V$ and
  $\tilde\W=\bar\W$. It is just a matter of
  applying directly the definitions. For example,
  we have
  \begin{equation*}
    \bar U(s)=\bar U_1\circ\bar\X(s)=U_1\circ f\circ\bar\X(s)=U_1\circ f\circ\tilde\X(s)=U_1\circ\X\circ h(s)=U\circ h(s)=\tilde U(s)
  \end{equation*}
  and
  \begin{equation*}
    \tilde\V_1(X)=f'(X)\V_1(f(X))=\frac{f'(X)x_1'(f(X))}{2c(U_1(f(X)))}=\frac{\bar x_1'(X)}{2c(\tilde U_1(X))}=\frac{\bar x_1'(X)}{2c(\bar U_1(X))}=\bar\V_1(X).
  \end{equation*}
  The equivariance property
  \eqref{eq:equivariantsST} of $S_T$ follows
  directly from the definition of $S_T$ and the
  equivariance properties \eqref{eq:equivariants}.
\end{proof}

\begin{definition}
  We define by $\Fquot$ the quotient of $\F$ with
  respect to the action of the group $\Gr^2$ on
  $\F$, that is,
  \begin{equation*}
    \psi\sim\bar\psi\text{ if there exists $\phi\in\Gr^2$ such that }\bar\psi=\psi\act\phi.
  \end{equation*}
\end{definition}

\begin{definition}
  Let
  \begin{equation*}
    \F_0=\{\psi=(\psi_1,\psi_2)\in\F\mid x_1+J_1=\id\text{ and }x_2+J_2=\id\}
  \end{equation*}
  and $\Pib\colon\F\to\F_0$ be the projection on
  $\F_0$ given by
  $\bar\psi=(\bar\psi_1,\bar\psi_2)=\Pib(\psi)$
  where $\bar\psi\in\F_0$ is defined as
  follows. Let
  \begin{equation}
    \label{eq:deffgpi}
    f(X)=x_1(X)+J_1(X),\quad g(Y)=x_2(Y)+J_2(Y),
  \end{equation}
  we set
  \begin{equation*}
    \bar\psi=\psi\act\phi^{-1}.
  \end{equation*}
\end{definition}

\begin{lemma} The following statements hold: \\
  \begin{enumerate}
  \item[(i)] We have
    \begin{equation}
      \label{eq:eqeqprojpi}
      \psi\sim\bar\psi\quad\text{ if and only if }\quad\Pib(\psi)=\Pib(\bar\psi)
    \end{equation}
    so that the sets $\Fquot$ and $\F_0$ are in bijection.\\[2mm]
  \item[(ii)] We have
    \begin{equation}
      \label{eq:mpim}
      \Mb\circ\Pib=\Mb
    \end{equation}
    and
    \begin{equation}
      \label{eq:linvm}
      \Lb\circ\Mb|_{\F_0}=\id|_{\F_0}\quad\text{ and }\quad\Mb\circ\Lb=\id
    \end{equation}
    so that the sets $\D$, $\F_0$ and $\Fquot$ are
    in bijection.
  \item[(iii)] We have
    \begin{equation}
      \label{eq:picomSt}
      \Pib\circ S_T\circ\Pib=\Pib\circ S_T.
    \end{equation}
  \end{enumerate}
\end{lemma}
Note that the first identity in \eqref{eq:linvm}
is equivalent to
\begin{equation}
  \label{eq:linvmeq}
  \Lb\circ\Mb\circ\Pib=\Pib.
\end{equation}
\begin{proof}
  \textbf{Step 1.}  We prove
  \eqref{eq:eqeqprojpi}. If $\bar\psi\sim\psi$,
  then there exists $\tilde\phi\in\Gr^2$ such that
  $\bar\psi=\psi\act\tilde\phi$. Let $\phi=(f,g)$
  and $\bar\phi=(\bar f,\bar g)$ be given by
  \eqref{eq:deffgpi} for $\psi$ and $\bar\psi$,
  respectively. One can check that
  $\bar\phi=\phi\circ\tilde\phi$ and therefore
  \begin{equation*}
    \Pib(\bar\psi)=\bar\psi\act(\bar\phi)^{-1}=(\psi\act\tilde\phi)\act(\phi\circ\tilde\phi)^{-1}=\psi\act(\tilde\phi\circ(\phi\circ\tilde\phi)^{-1})=\psi\act\phi^{-1}=\Pib(\psi).
  \end{equation*}
  Conversely, if $\Pib(\bar\psi)=\Pib(\psi)$ then
  $\bar\psi\act\bar\phi^{-1}=\psi\act\phi^{-1}$ so
  that
  $\bar\psi=(\psi\act\phi^{-1})\act\bar\phi=\psi\act(\phi^{-1}\circ\phi)$
  and $\bar\psi$ and $\psi$ are equivalent.

  \textbf{Step 2.}  We prove that
  $\Lb\circ\Mb=\id_{\F_0}$. Given
  $\psi=(\psi_1,\psi_2)\in\F_0$, let us consider
  $(u,R,S,\mu,\nu)=\Lb(\psi_1,\psi_2)$ and
  $\bar\psi=(\bar\psi_1,\bar\psi_2)=\Mb(u,R,S,\mu,\nu)$. We
  want to prove that $\bar\psi=\psi$. Let
  \begin{equation}
    \label{eq:gdef}
    g(x)=\sup\{X\in\Real\mid x_1(X)<x\}.
  \end{equation}
  It is not hard to prove, using the fact that
  $x_1$ is increasing and continuous, that
  \begin{equation}
    \label{eq:ygxeqx}
    x_1(g(x))=x
  \end{equation}
  for all $x\in\Real$ and
  $x_1\inv((-\infty,x))=(-\infty,g(x))$. For any
  $x\in\Real$, we have, by \eqref{eq:defMbpf1},
  that
  \begin{equation}
    \label{eq:J1geqmu}
    \mu((-\infty,x))=\int_{x_1\inv((-\infty,x))}J_1'(X)\,dX
    =\int_{-\infty}^{g(x)}J_1'(X)\,dX
    =J_1(g(x))
  \end{equation}
  because $J_1(-\infty)=0$. Since $\psi\in\F_0$,
  $x_1+J_1=\id$ and we get, by \eqref{eq:ygxeqx}
  and \eqref{eq:J1geqmu}, that
  \begin{equation}
    \label{eq:mupxeqg}
    \mu((-\infty,x))+x=g(x).
  \end{equation}
  From the definition of $\bar x_1$, we then obtain
  that
  \begin{equation}
    \label{eq:ybardef}
    \bar x_1(X)=\sup\{x\in\Real\mid g(x)<X\}.
  \end{equation}
  For any given $X\in\Real$, let us consider an
  increasing sequence $z_i$ tending to $\bar
  x_1(X)$ such that $g(z_i)<X$; such sequence
  exists by \eqref{eq:ybardef}. Since $x_1$ is
  increasing and using \eqref{eq:ygxeqx}, it
  follows that $z_i\leq x_1(X)$. Letting $i$ tend
  to $\infty$, we obtain $\bar x_1(X)\leq
  x_1(X)$. Assume that $\bar x_1(X)<x_1(X)$. Then,
  there exists $x$ such that $\bar
  x_1(X)<x<x_1(X)$ and \eqref{eq:ybardef} then
  implies that $g(x)\geq X$. On the other hand,
  $x=x_1(g(x))<x_1(X)$ implies $g(x)<X$ because
  $x_1$ is increasing, which gives us a
  contradiction. Hence, we have $\bar x_1=x_1$. It
  follows directly from the definitions, since
  $x_1+J_1=\id$, that $\bar J_1=J_1$ and $\bar
  U_1=U_1$. It follows from the definition
  \eqref{eq:defV1V2} and \eqref{eq:reluRSF} that
  $\bar V_1=V_1$ and $\bar V_2=V_2$. Thus we
  have proved that $\bar\psi_1=\psi_1$. In the
  same way, we prove that $\bar\psi_2=\psi_2$,
  which concludes the proof that $L\circ
  M=\id_{\F_0}$.
  
  \textbf{Step 3.} We prove that $\Mb\circ\Lb=\id$.
  Given $(u,R,S,\mu,\nu)\in\D$, let
  $\psi=(\psi_1,\psi_2)=\Lb(u,R,S,\mu,\nu)$ and
  $(\bar u,\bar R,\bar S,\bar \mu,\bar
  \nu)=\Mb(\psi)$. We want to prove that $(\bar
  u,\bar R,\bar S,\bar \mu,\bar
  \nu)=(u,R,S,\mu,\nu)$. Let $g$ be the function
  defined as before by \eqref{eq:gdef}. The same
  computation that leads to \eqref{eq:mupxeqg} now
  gives
  \begin{equation}
    \label{eq:mubarg}
    \bar \mu((-\infty,x))+x=g(x).
  \end{equation}
  Given $X\in\Real$, we consider an increasing
  sequence $x_i$ which converges to $x_1(X)$ and
  such that $\mu((-\infty,x_i))+x_i<X$. Passing to
  the limit and since $x\mapsto\mu((-\infty,x))$ is
  lower semi-continuous, we obtain
  $\mu((-\infty,x_1(X)))+x_1(X)\leq X$. We take
  $X=g(x)$ and get
  \begin{equation}
    \label{eq:musg}
    \mu((-\infty,x))+x\leq g(x).
  \end{equation}
  From the definition of $g$, there exists an
  increasing sequence $X_i$ which converges to
  $g(x)$ such that $x_1(X_i)<x$. The definition
  \eqref{eq:defx1} of $x_1$ tells us that
  $\mu((-\infty,x))+x\geq X_i$. Letting $i$ tend
  to infinity, we obtain $\mu((-\infty,x))+x\geq
  g(x)$ which, together with \eqref{eq:musg},
  yields
  \begin{equation}
    \label{eq:mueqg}
    \mu((-\infty,x))+x=g(x).
  \end{equation}
  Comparing \eqref{eq:mueqg} and \eqref{eq:mubarg}
  we get that $\bar\mu=\mu$. Similarly, one proves
  that $\bar\nu=\nu$. It is clear from the
  definitions that $\bar u=u$. The fact that $\bar
  R=R$ and $\bar S=S$ follow from
  \eqref{eq:defV1V2} and
  \eqref{eq:reluRSF}. Hence, we have proved that
  $(\bar u,\bar R,\bar S,\bar \mu,\bar
  \nu)=(u,R,S,\mu,\nu)$ and $M\circ L=\id_{\D}$.
  \textbf{Step 4.} We prove
  \eqref{eq:picomSt}. For any
  $\psi=(\psi_1,\psi_2)\in\F$, we denote
  $\psi_T=S_T\psi$. Let $\phi=(f,g)\in\Gr^2$ and
  $\phi_T=(f_T,g_T)\in\Gr^2$ be defined as in
  \eqref{eq:deffgpi} so that
  $\Pib\psi=\psi\act\phi^{-1}$ and
  $\Pib\psi_T=\psi_T\act\phi_T^{-1}$. By using
  \eqref{eq:equivariantsST}, we get
  \begin{equation*}
    S_T\circ\Pib(\psi)=S_T(\psi\act\phi^{-1})=S_T(\psi)\act\phi^{-1}
  \end{equation*}
  and therefore $S_T\circ\Pib(\psi)$ and
  $S_T(\psi)$ are equivalent. Then,
  \eqref{eq:picomSt} follows from
  \eqref{eq:eqeqprojpi}.
\end{proof}

We now come to our main theorem.
\begin{theorem} \label{th:main}
  Given $(u_0,R_0,S_0,\mu_0,\nu_0)\in\D$, let us
  denote $(u,R,S,\mu,\nu)(t)=\bar
  S_t(u_0,R_0,S_0,\mu_0,\nu_0)$. Then $u$
  is  a weak solution of the nonlinear variational
  wave equation \eqref{eq:nvw}, that is,
  \begin{equation}
    \label{eq:soldist}
    \int_{\Real^2}(\phi_t-(c(u)\phi)_x)R\,dxdt+\int_{\Real^2}(\phi_t+(c(u)\phi)_x)S\,dxdt=0
  \end{equation}
  for all smooth functions $\phi$ with compact
  support and where
  \begin{equation}
    \label{eq:weakderu1}
    R=u_t+c(u)u_x,\quad S=u_t-c(u)u_x.
  \end{equation}
  Moreover, the measures $\mu(t)$ and $\nu(t)$
  satisfy the following equations in the sense of
  distribution
  \begin{subequations}
    \label{eq:meassol}
    \begin{equation}
      \label{eq:meassol1}
      (\mu+\nu)_t-(c(\mu-\nu))_x=0
    \end{equation}
    and
    \begin{equation}
      \label{eq:meassol2} 
      (\frac{1}{c}(\mu-\nu))_t-(\mu+\nu)_x=0.
    \end{equation}
  \end{subequations}
  The mapping $\bar S_T:\D\to\D$ is a semigroup, that is,
  \begin{equation*}
    \bar S_{t+t'}=\bar S_t\circ \bar S_{t'}
  \end{equation*}
  for all positive $t$ and $t'$.
\end{theorem}
\begin{proof}
  From Lemma \ref{lem:relURS}, we know that
  \eqref{eq:weakderu1} is fulfilled. On can check
  that \eqref{eq:soldist} is equivalent to
  \begin{equation}
    \label{eq:soldist1}
    R_t-(c(u)R)_x+S_t+(c(u)S)_x+c'(u)\frac{(R-S)^2}{2c(u)}=0
  \end{equation}
  in the sense of distributions, which makes
  senses, as $R$, $S$ belong to $L^2$ and $c(u)$,
  $c'(u)$ are bounded. After a change of
  variables, we have
  \begin{align*}
    \int_{\Real^2}(R\phi_t-(c(u)R)\phi_x)(t,x)\,dtdx&=\int_{\Real^2}(R(\phi_t-c(u)\phi_x))(t,x)(t_Xx_Y-t_Yx_X)\,dXdY\\
    &=2\int_{\Real^2}(c(u)R(\phi_t-c(u)\phi_x))(t,x)t_Xx_Y\,dXdY\\
    &=-2\int_{\Real^2}c(U)U_X\phi(t,x)_Y\,dXdY
  \end{align*}
  by \eqref{eq:relRUXY1} and because
  $\phi(t,x)_Y=\phi_t(t,x)t_Y+\phi_x(t,x)x_Y
=-\frac{x_Y}{c}(\phi_t-c\phi_x)(t,x)$. We
  integrate by parts and obtain
  \begin{align}
    \notag
    \int_{\Real^2}(R\phi_t-(c(u)R)\phi_x)(t,x)\,dtdx&=2\int_{\Real^2}(c(U)U_X)_Y\phi(t,x)\,dXdY\\
    \label{eq:exprRtuxy}
    &=2\int_{\Real^2}(c(U)U_{XY}+c'(U)U_XU_Y)\phi(t,x)\,dXdY.
  \end{align}
  In the same way, one proves that 
  \begin{equation}
    \label{eq:exprStuxy}
    \int_{\Real^2}(S\phi_t+(c(u)S)\phi_x)(t,x)\,dtdx=2\int_{\Real^2}(c(U)U_{XY}+c'(U)U_XU_Y)\phi(t,x)\,dXdY.
  \end{equation}
  We have, after a change of variables,
  \begin{equation}
    \label{eq:Rsqmin}
    \int_{\Real^2}\frac{R^2-2RS+S^2}{2c(u)}c'(u)\phi\,dtdx=2\int_{\Real^2}\Big(\frac{R^2-2RS+S^2}{2c(u)^2}c'(u)\phi\Big)(t,x)x_Xx_Y\,dXdY.
  \end{equation}
  We introduce the set $A=A_1\cup A_2$ where
  \begin{equation}
    \label{eq:defA1}
    A_1=\{(X,Y)\in\Real^2\mid x_X(X,Y)=0,\quad x_Y(X,Y)>0\text{ and }c'(U)(X,Y)\neq0\}
  \end{equation}
  and
  \begin{equation}
    \label{eq:defA2}
    A_2=\{(X,Y)\in\Real^2\mid x_Y(X,Y)=0,\quad x_X(X,Y)>0\text{ and }c'(U)(X,Y)\neq0\}.
  \end{equation}
  We claim that 
  \begin{equation}
    \label{eq:measclaim}
    \meas(A)=\meas(A_1)=\meas(A_2)=0.
  \end{equation}
  We prove this claim later. By using
  \eqref{eq:measclaim}, we get
  \begin{align}
    \notag
    \int_{\Real^2}\frac{R^2-2RS+S^2}{2c(u)}c'(u)\phi\,dtdx&=2\int_{A^c}\Big(\frac{R^2-2RS+S^2}{2c(u)^2}c'(u)\phi\Big)(t,x)x_Xx_Y\,dXdY\\
    \notag
    &=\int_{A^c}\Big(\frac{U_X^2}{x_X}x_Y+2U_XU_Y+\frac{U_Y^2}{x_Y}x_X\Big)c'(U)\phi(t,x)\,dXdY\\
    \label{eq:Rsqm2rs0}
    &=\int_{A^c}\Big(2\frac{J_Xx_Y}{c^2(U)}+2U_XU_Y+\frac{J_Yx_X}{c^2(U)}\Big)c'(U)\phi(t,x)\,dXdY\\
    \label{eq:Rsqm2rs}
    &=\int_{\Real}\Big(2\frac{J_Xx_Y}{c^2(U)}+2U_XU_Y+\frac{J_Yx_X}{c^2(U)}\Big)c'(U)\phi(t,x)\,dXdY.
  \end{align}
  Note that \eqref{eq:measclaim} is necessary to
  get \eqref{eq:Rsqm2rs} from \eqref{eq:Rsqm2rs0}
  as the integrand in \eqref{eq:Rsqm2rs0} does not
  vanish on $A$. After combining
  \eqref{eq:exprRtuxy}, \eqref{eq:exprStuxy} and
  \eqref{eq:Rsqm2rs}, and using the governing
  equations \eqref{eq:goveq}, we get
  \begin{align*}
    \int_{\Real^2}
    (R+S)\phi_t&-(c(u)(R-S))\phi_x-c'(u)\frac{(R-S)^2}{2c(u)}\phi\,dtdx\\
    &=-\int_{\Real^2}\Big(4c(U)U_{XY}-\frac{2c'}{c^2}(J_Xx_Y+J_Yx_X)+2c'(U)U_YU_X\Big)\phi(t,x)\,dXdY\\
    &=0,
  \end{align*}
  which proves \eqref{eq:soldist1} and therefore
  \eqref{eq:soldist} holds. It remains to prove
  the claim \eqref{eq:measclaim}. Let us introduce
  the set
  \begin{equation*}
    A_1(X)=\{Y\in\Real\mid (X,Y)\in A_1\}.
  \end{equation*}
  Let us prove that, for almost every $X\in\Real$,
  $\meas(A_1(X))=0$ and therefore, by Fubini's
  theorem, $\meas(A_1)=0$. We consider a point
  $Y_0\in A_1(X)$ and a rectangle
  $\Omega=[X_l,X_r]\times[Y_l,Y_r]$ which contains
  $(X,Y_0)$. Since $Z\in\H(\Omega)$, there exist
  $\delta>0$ such that $(x_X+J_X)(X,Y)>\delta$ for
  almost every $X\in\Real$ and all
  $Y\in\Real$. Since $x_X\in\WY(\Omega)$, the
  function $x_X$ is continuous with respect to $Y$
  for almost any given $X\in\Real$.  Formally the
  argument goes as follows: We consider a fixed
  given $X\in\Real$ and denote
  $f(Y)=x_X(X,Y)$. For any $Y_0\in A_1(X)$, we
  have, by definition, $x_X(X,Y_0)=f(Y_0)=0$. By
  using \eqref{eq:goveqx}, we get
  \begin{equation*}
    f'(Y_0)=x_{XY}(X,Y)=0
  \end{equation*}
  because $x_X=0$ implies $U_X=0$, see
  \eqref{eq:energrel2}. We do not have enough
  regularity to differentiate \eqref{eq:goveqx};
  but if nevertheless we do so, then we formally obtain
  \begin{align*}
    f''(Y_0)=x_{XYY}(X,Y_0)&=\frac{c'}{2c}(U_Yx_{XY}+U_{XY}x_Y)(X,Y_0)\\
    &=\frac{c'^2}{4c^4}(J_Xx_Y^2)(X,Y_0)
  \end{align*}
  where we have used again the fact
  $x_X(X,Y_0)=U_X(X,Y_0)=0$. We have
  $J_X(X,Y_0)=(x_X+J_X)(X,Y_0)\geq\delta$ and
  $x_Y(X,Y_0)>0$, $c'^2(U(X,Y_0))>0$ because
  $(X,Y_0)\in A_1(X)$. Hence, $f''(Y_0)>0$ and it
  implies that $f(Y)>0$ for all $Y$ different from
  $Y_0$ in a neighborhood of $Y_0$, so that the
  points in $A_1(X)$ are isolated. Let us now
  prove this result rigorously. Again, we consider
  $Y_0\in A_1(X)$ and a rectangle
  $\Omega=[X_l,X_r]\times[Y_l,Y_r]$ which contains
  $(X,Y_0)$. Without loss of generality, we assume
  that $Y_0$ is a Lebesgue point for the function
  $Y\mapsto x_Y(X,Y)$ and, therefore, since
  $x_Y(X,Y_0)>0$, there exists $\delta>0$ such
  that
  \begin{equation}
    \label{eq:lowxY}
    \int_{Y_0}^Yx_Y(\bar Y)\,d\bar Y>\delta'(Y-Y_0)
  \end{equation}
  in a neighborhood of $Y_0$. We can choose
  $\delta'>0$ such that, in addition,
  \begin{equation*}
    c'(U(X,Y))>\delta\quad\text{ and }\quad J_X(X,Y)>\delta
  \end{equation*}
  in a neighborhood of $Y_0$ (we recall that $U$ is
  continuous). We have, after using the governing
  equations \eqref{eq:goveq},
  \begin{align*}
    x_X&=\int_{Y_0}^{Y}\frac{c'}{2c}(U_Yx_X+U_Xx_Y)\,d\bar Y\\
    &=\int_{Y_0}^{Y}\frac{c'}{2c}(U_Y\int_{Y_0}^{\bar
      Y}x_{XY}\,d\tilde Y+x_Y\int_{Y_0}^{\bar Y}U_{XY}d\tilde Y)\,d\bar Y\\
    &=\int_{Y_0}^{Y}\frac{c'}{2c}\Big(U_Y\int_{Y_0}^{\bar
      Y}x_{XY}\,d\tilde Y+x_Y\int_{Y_0}^{\bar
      Y}\big(\frac{c'}{2c^3}(x_YJ_X+J_Yx_X-\frac{c'}{2c}U_YU_X)\big)d\tilde
    Y\Big)\,d\bar Y.
  \end{align*}
  Since $x_{XY}$ and $U_{XY}$ are bounded (by
  \eqref{eq:goveq}) and $x_X(X,Y_0)=U_X(X,Y_0)=0$,
  we have that $x_X(X,Y)\leq C\abs{Y-Y_0}$ and
  $U_X(X,Y)\leq C\abs{Y-Y_0}$ in a neighborhood of
  $Y_0$ for a constant $C$ which depends only on
  $\tnorm{Z}_{\H(\Omega)}$. Hence,
  \begin{align*}
    \abs{\int_{Y_0}^{\bar Y}x_{XY}\,d\tilde
    Y}&=\abs{\int_{Y_0}^{\bar
      Y}\frac{c'}{2c}(U_Yx_X+U_Xx_Y)\,d\tilde Y}\\
    &\leq C\int_{Y_0}^{\bar Y}\abs{\tilde Y-Y_0}\,d\tilde
    Y\\
    &\leq C(Y-Y_0)^2.
  \end{align*}
  Thus,
  \begin{equation}
    \label{eq:lowbdxX1}
    \abs{\int_{Y_0}^{Y}\frac{c'}{2c}(U_Y\int_{Y_0}^{\bar
      Y}x_{XY}\,d\tilde Y)\,d\bar Y}\leq C\abs{Y-Y_0}^3.
  \end{equation}
  In the same way, one proves that 
  \begin{equation*}
    \abs{\int_{Y_0}^{Y}\frac{c'}{2c}x_Y\int_{Y_0}^{\bar
        Y}\Big(\frac{c'}{2c^3}(J_Yx_X-\frac{c'}{2c}U_YU_X)\Big)d\tilde
      Y)\,d\bar Y}\leq C\abs{Y-Y_0}^3.
  \end{equation*}
  For $Y>Y_0$, we have
  \begin{align*}
    \int_{Y_0}^{Y}\frac{c'}{2c}x_Y\Big(\int_{Y_0}^{\bar
      Y}\frac{c'}{2c^3}x_YJ_Xd\tilde Y\Big)\,d\bar
    Y&\geq\frac{\kappa^4\delta^3}{4}\int_{Y_0}^{Y}x_Y\Big(\int_{Y_0}^{\bar
      Y}x_Yd\tilde Y\Big)\,d\bar Y\\
    &=\frac{\kappa^4\delta^3}{4}\Big(\int_{Y_0}^{Y}x_Y\,d\bar
    Y\Big)^2\\ &\geq
    \frac{\kappa^4\delta^5}4(Y-Y_0)^2
  \end{align*}
  in a neighborhood of $Y_0$. We can check that
  the same inequality holds for $Y<Y_0$. Finally,
  we obtain that, in a neighborhood of $Y_0$,
  \begin{equation}
    \label{eq:lowxXsq}
    x_X(X,Y)\geq\frac{\kappa^4\delta^5}4(Y-Y_0)^2-C\abs{Y-Y_0}^3\geq \frac{\kappa^4\delta^5}8(Y-Y_0)^2.
  \end{equation}
  To complete the argument, we consider the sets
  \begin{multline*}
    A_1^k(X)=\{Y_0\in A_1(X)\cap[Y_l,Y_r]\mid
    x_X(X,Y_0)=0\\\text{ and }x_X(X,Y)>0\text{ for
      all }Y\in[Y_0-\frac1k,Y_0+\frac1k]\setminus
    \{Y_0\}\}
  \end{multline*}
  for any integer $k$. By \eqref{eq:lowxXsq}, we
  have
  \begin{equation*}
    A_1(X)\cap[Y_l,Y_r]=\cup_{k>0}A_1^k(X).
  \end{equation*}
  At the same time, since $A_1^k(X)$ consists of
  points separated by a distance of at least
  $\frac1k$, we have $\meas(A_1^k(X))=0$. Hence,
  after taking sequences of $Y_l$ and $Y_r$ which
  tend to plus and minus infinity, respectively,
  we get $\meas(A_1(X))=0$ so that $\meas(A_1)=0$
  and the proof of the claim \eqref{eq:measclaim}
  is complete. Let us prove \eqref{eq:meassol1},
  that is,
  \begin{equation*}
    \int_{\Real^2}(\phi_t-c(u)\phi_x)\,d\mu dt+\int_{\Real^2}(\phi_t+c(u)\phi_x)\,d\nu dt=0
  \end{equation*}
  for all smooth function $\phi$ with compact
  support. We have, after a change of variables,
  that
  \begin{multline*}
    \int_{\Real}\Big(\int_\Real(\phi_t-c(u)\phi_x)\,d\mu(t)\Big)dt\\=\int_{\Real}\Big(\int_\Real(\phi_t(t,x(t,s))-c(u(t,x(t,s)))\phi_x(t,x(t,s)))\V_4(t,\X(t,s))\X_s(t,s)\,ds\Big)dt,
  \end{multline*}
  where we have added the dependence in $t$ of the
  values of $\Theta(t)=\Lb(u,R,S,u,\mu,\nu)(t) $
  (which gives $x(t,s)$, $\V_4(t,X)$, $u(t,s)$ and
  $\X(t,s)$ in the equation above). We proceed to
  the change of variables
  $(X,Y)\mapsto(t(X,Y),s=\frac12(X+Y))$ whose
  Jacobian is equal to $\frac{x_X+x_Y}{2c(u)}$ and get
  \begin{multline}
    \label{eq:JXXs}
    \int_{\Real}\Big(\int_\Real(\phi_t-c(u)\phi_x)\,d\mu(t)\Big)dt\\
    =\int_{\Real^2}(\phi_t(t,x)-c(u(t,x))\phi_x(t,x))J_X(X,Y)\X_s(t,s)\frac{(x_X+x_Y)(X,Y)}{2c(U(X,Y))}\,dXdY.
  \end{multline}
  Since $t(\X(t,s),\Y(t,s))=t$, by definition, we
  get $t_X\X_s+t_Y\Y_s=0$ and, since
  $\X(s)+\Y(s)=2s$, we have $\X_s+\Y_s=2$. Hence,
  $(x_X+x_Y)\X_s(t,s)=2x_Y$ and \eqref{eq:JXXs}
  implies
  \begin{multline*}
    \int_{\Real}\Big(\int_\Real(\phi_t-c(u)\phi_x)\,d\mu(t)\Big)dt\\
    =\int_{\Real^2}(\phi_t(t,x)-c(u(t,x))\phi_x(t,x))J_X(X,Y)\frac{x_Y}{c(U(X,Y))}\,dXdY.
  \end{multline*}
  Since
  $\phi(t,x)_Y=-\frac{x_Y}{c(u)}(\phi_t-c(u)\phi_x)(t,x)$, it yields
  \begin{equation*}
    \int_{\Real}\Big(\int_\Real(\phi_t-c(u)\phi_x)\,d\mu(t)\Big)dt=-\int_{\Real^2}\phi(t,x)_YJ_X\,dXdY.
  \end{equation*}
  Similarly, one proves that
  \begin{equation*}
    \int_{\Real}\Big(\int_\Real(\phi_t+c(u)\phi_x)\,d\nu(t)\Big)dt=\int_{\Real^2}\phi(t,x)_XJ_Y\,dXdY
  \end{equation*}
  so that
  \begin{equation}
    \label{eq:phiyJx}
    \int_{\Real^2}(\phi_t-c(u)\phi_x)\,d\mu dt+\int_{\Real^2}(\phi_t+c(u)\phi_x)\,d\nu dt=\int_{\Real^2}(-\phi(t,x)_YJ_X+\phi(t,x)_XJ_Y)\,dXdY=0,
  \end{equation}
  by integration by parts, as the support of
  $\phi$ is compact. Similarly one proves
  \eqref{eq:meassol2}. Note that the integrand in
  \eqref{eq:phiyJx} is equal to the exact form
  $d(\phi dJ)$ and equation \eqref{eq:meassol1} is
  actually equivalent to $ddJ=0$ while
  \eqref{eq:meassol2} is equivalent to
  $ddK=0$. The proof of the semigroup property
  follows in a straightforward manner from the
  results that have been established in this
  section. We have
  \begin{align*}
    \bar S_{T}\circ \bar S_{T'}&=\Mb\circ S_T\circ \Lb\circ\Mb\circ S_{T'}\circ \Lb\\
    &=\Mb\circ\Pib\circ S_T\circ \Lb\circ\Mb\circ\Pib\circ S_{T'}\circ \Lb&\text{ by \eqref{eq:mpim} }\\
    &=\Mb\circ\Pib\circ S_T\circ\Pib\circ S_{T'}\circ \Lb&\text{ by \eqref{eq:linvmeq} }\\
    &=\Mb\circ\Pib\circ S_T \circ S_{T'}\circ \Lb&\text{ by \eqref{eq:picomSt} }\\
    &=\Mb\circ S_T \circ S_{T'}\circ \Lb&\text{ by \eqref{eq:mpim} }\\
    &=\Mb\circ S_{T+T'}\circ \Lb&\text{ by Theorem \ref{th:Stsemigroup} }\\
    &=\bar S_{T+T'}.
  \end{align*}
\end{proof}

The semigroup of solution we have constructed is
conservative in the sense given by the following
theorem.
\begin{theorem}
  \label{th:energconcentration}
  Given $(u_0,R_0,S_0,\mu_0,\nu_0)\in\D$, let us
  denote $(u,R,S,\mu,\nu)(t)=\bar
  S_t(u_0,R_0,S_0,\mu_0,\nu_0)$. We have
  \begin{enumerate}
  \item[(i)] For all $t\in\Real$
    \begin{equation}
      \label{eq:prestotenerg}
      \mu(t)(\Real)+\nu(t)(\Real)=\mu_0(\Real)+\nu_0(\Real).
    \end{equation}
  \item[(ii)] For almost every $t\in\Real$, the
    singular part of $\mu(t)$ and $\nu(t)$ are
    concentrated on the set where $c'(u)=0$.
  \end{enumerate}
\end{theorem}

This theorem corresponds to
Theorem 3 in \cite{BreZhe:06}. We use a different
proof based on the coarea formula.

\begin{proof}
  Let us prove (i). We consider a given time that
  we denote $\tau$ (to avoid any confusion with
  the function $t(X,Y)$). As in the proof of the
  previous theorem, we add the dependence in time
  of the values of
  $\Theta(\tau)=\Lb(u,R,S,u,\mu,\nu)(\tau)$. In
  particular, the curve $(\X(\tau,s),\Y(\tau,s))$
  corresponds to the curve where time is constant
  and equal to $\tau$, that is,
  $t((\X(\tau,s),\Y(\tau,s)))=\tau$. By definition
  (see \eqref{eq:relufromG1}, and Definition
  \ref{def:Eb}), we have, for any Borel set $B$,
  \begin{align}
    \label{eq:mutauB1}
    \mu_\tau(B)&=\int_{\{s\in\Real\ |\  x(\tau,s)\in B\}} \V_4(\X(\tau,s))\dot\X(\tau,s)\,ds\\
    \label{eq:mutauB2}
    &=\int_{\{s\in\Real\ |\  x(\X(\tau,s),\Y(\tau,s))\in B\}} J_X(\X(\tau,s),\Y(\tau,s))\dot\X(\tau,s)\,ds.
  \end{align}
  Here, the notation may be confusing as $x$
  denotes two different but of course very related
  functions. In \eqref{eq:mutauB1}, we have
  $x(\tau,s)=\Z_2(\tau,s)$, which corresponds to
  the space variable $\Z_2$ parametrized by $s$ at
  time $\tau$ while, in \eqref{eq:mutauB2},
  $x(X,Y)=Z_2(X,Y)$, corresponds to the value of
  the space variable $Z_2$, where $Z(X,Y)$ is the
  solution of \eqref{eq:goveq} on the whole
  $\Real^2$ plane. We have
  $\Z_2(\tau,s)=Z_2(\X(\tau,s),\X(\tau,s))$ by
  \eqref{eq:defextdatZ} abd Definition
  \ref{def:Eb}. Correspondingly, we have
  \begin{equation*}
    \nu_\tau(B)=\int_{\{s\in\Real\ |\  x(\X(\tau,s),\Y(\tau,s))\in B\}} J_Y(\X(\tau,s),\Y(\tau,s))\dot\Y(\tau,s)\,ds.
  \end{equation*}
  Hence,
  \begin{align*}
    \mu_\tau(\Real)+\nu_\tau(\Real)&=\int_\Real J_X(\X(\tau,s),\Y(\tau,s))\dot\X(\tau,s)+J_Y(\X(\tau,s),\Y(\tau,s))\dot\Y(\tau,s)\,ds\\
    &=\lim_{s\to\infty} J(\X(\tau,s),\Y(\tau,s))\,ds\\
    &=\lim_{s\to\infty} J(\X(0,s),\Y(0,s))\,ds,\quad\text{ by Lemma \ref{lem:exany},}\\
    &=\mu_0(\Real)+\nu_0(\Real)
  \end{align*}
  
  Let us prove (ii). Let
  $\mu_\tau=(\mu_{\tau})_{\text{ac}}+(\mu_{\tau})_{\text{sing}}$
  be the Radon-Nykodin decomposition of
  $\mu_\tau$.  We want to prove that, for allmost
  every time $\tau\in\Real$, we have
  \begin{equation}
    \label{eq:musC}
    (\mu_{\tau})_{\text{sing}}(\{x\in\Real\ |\  c'(u(\tau,x))\neq 0\})=0.
  \end{equation}
  Let us introduce
  \begin{equation*}
    A_\tau=\{s\in\Real\ |\ x_X(\X(\tau,s),\Y(\tau,s))>0\}.
  \end{equation*}
  The set $A_\tau$ corresponds to $A$ in
  \eqref{eq:defsetA} in the proof of Lemma
  \ref{lem:revufromG}.  In this same proof, we
  obtain that, for any Borel set $B$,
  \begin{equation*}
    (\mu_\tau)_{\text{ac}}(B)=\mu_\tau(B\cap (x(\tau,(A_\tau)^c))^c)
  \end{equation*}
  so that 
  \begin{equation*}
    (\mu_\tau)_{\text{sing}}(B)=\mu_\tau(B\cap (x(\tau,(A_\tau)^c))),
  \end{equation*}
  because $\meas(x(\tau,(A_\tau)^c))=0$. Hence,
  \begin{equation}
    \label{eq:mutausinge}
    (\mu_{\tau})_{\text{sing}}(B)=\int_{\{s\in\Real\ |\ x(\X(\tau,s),\Y(\tau,s))\in B\cap (A_\tau)^c\}}J_X(\X(\tau,s),\Y(\tau,s))\dot\X(s)\,ds.
  \end{equation}
  We introduce the set
  \begin{equation*}
    E=\{(X,Y)\in\Real^2\ |\ x_X(X,Y)=0\text{ and } c'(U(X,Y))\neq 0\}.
  \end{equation*}
  By using \eqref{eq:mutausinge}, we get
  \begin{equation}
    \label{eq:mutaucp}
    \mu_\tau(\{x\in\Real\ |\  c'(u(\tau,x))\neq 0\})=\int_{\{s\in\Real\ |\ (\X(\tau,s),\Y(\tau,s))\in E\}}J_X(\X(\tau,s),\Y(\tau,s))\dot\X(s)\,ds
  \end{equation}
  By the coarea formula, see \cite{Ambrosio}, we
  get
  \begin{equation*}
    \int_{\Real}\H^1(E\cap t^{-1}(\tau))\,d\tau=\int_E\sqrt{t_X^2+t_Y^2}\,dXdY=\int_E\frac{x_Y}{c(U)}\,dXdY=0
  \end{equation*}
  because of \eqref{eq:measclaim}. Here, $\H^1$
  denotes the one-dimensional Hausdorff
  measure. Hence, we have that, for allmost every
  time $\tau\in\Real$, the set $E\cap
  t^{-1}(\tau)$ has zero one-dimensional Hausdorff
  measure. We claim that, if
  $\mu_\tau(\{x\in\Real\ |\ c'(u(\tau,x))\neq
  0\})>0$, then $\H^1(E\cap
  t^{-1}(\tau))>0$. Indeed, let us define, for a
  given $\tau$, the mapping $\Gamma_\tau:s\mapsto
  (\X(\tau,s),\Y(\tau,s))$ from $\Real$ to
  $\Real^2$. We rewrite \eqref{eq:mutaucp} as
  \begin{equation*}
    \mu_\tau(\{x\in\Real\ |\ c'(u(\tau,x))\neq
    0\})=\int_{\Gamma_\tau^{-1}(E)}J_X(\X(\tau,s),\Y(\tau,s))\dot\X(s)\,ds.
  \end{equation*}
  In particular it implies that, if
  $\mu_\tau(\{x\in\Real\ |\ c'(u(\tau,x))\neq
  0\})>0$, then $\meas(\Gamma_\tau^{-1}(E))>0$. By
  the area formula, we have
  \begin{equation*}
    \H^1(E)\geq\H^1(\Gamma_\tau\circ\Gamma_\tau^{-1}(E))=\int_{\Gamma_\tau^{-1}(E)}(\X_s^2+\Y_s^2)^{1/2}\,ds\geq\meas(\Gamma_\tau^{-1}(E))
  \end{equation*}
  because
  $(\X_s^2+\Y_s^2)^{1/2}\geq\frac12(\X_s+\Y_s)=1$. Hence,
  our claim is proved and it follows that
  $\mu_\tau(\{x\in\Real\ |\ c'(u(\tau,x))\neq
  0\})>0$ for at most almost every $\tau\in\Real$
  and we have proved \eqref{eq:musC}.
\end{proof}

\begin{figure}[h]
  \centering
  \includegraphics[scale=0.8]{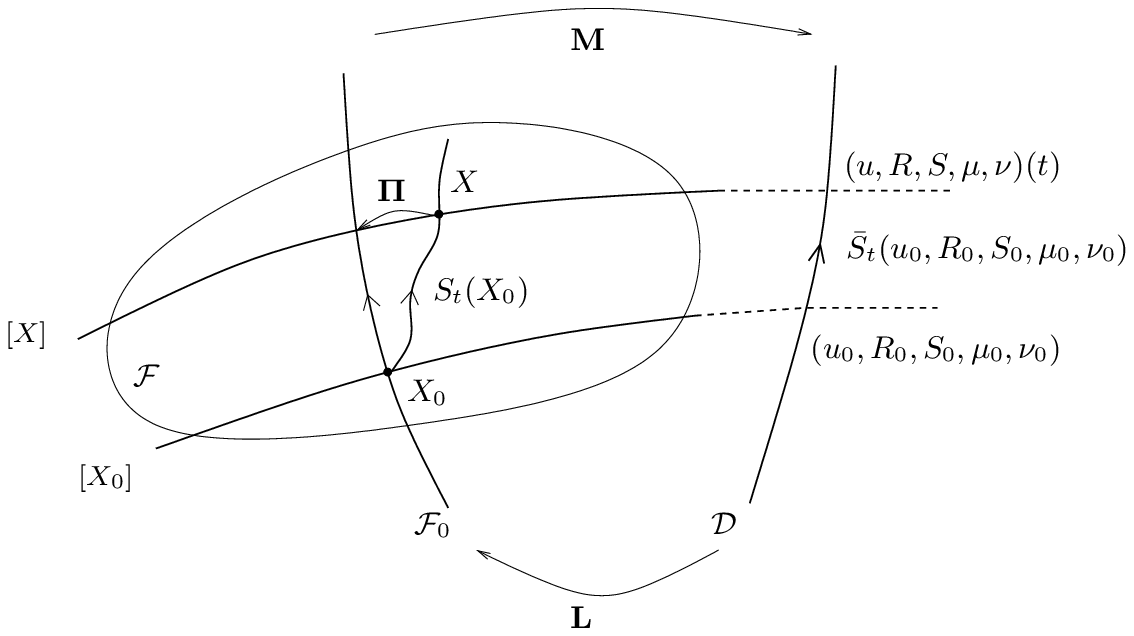}
  \caption{The semigroup $\bar S_t$.}
  \label{fig:semigroup}
\end{figure}

\begin{theorem}[Finite speed of propagation] \label{thm:finite}
  Given $\tb\geq0$ and $\xb\in\Real$, for any two
  initial datas
  $\zeta_0=(u_0,R_0,S_0,\mu_0,\nu_0)$ and
  $\bar\zeta_0=(\bar u_0,\bar R_0,\bar
  S_0,\bar\mu_0,\bar\nu_0)$ in $\D$, we consider
  the corresponding solutions $(u,R,S,\mu,\nu)(t)$
  and $(u,R,S,\mu,\nu)(t)$ given by Definition
  \ref{def:bST}. If the restrictions of $\zeta_0$
  and $\bar\zeta_0$ are equal on
  $[\xb-\kappa\tb,\xb+\kappa\tb]$, then the two
  solutions coincide at $(\tb,\xb)$, that is,
  $u(\tb,\xb)=\bar u(\tb,\xb)$.
\end{theorem}

\begin{proof}
  For a given $\zeta_0=(u_0,R_0,S_0,\mu_0,\nu_0)$,
  we define $\bar\zeta_0$ equal to $\zeta_0$ on
  $[\xb-\kappa\tb,\xb+\kappa\tb]$ and zero
  otherwise, i.e.,
  \begin{equation*}
    (\bar u_0,\bar R_0,\bar
    S_0)(x)=
    \begin{cases}
      (u_0,R_0,S_0)(x)&\text{ if
      }x\in[\xb-\kappa\tb,\xb+\kappa\tb]\\
      0&\text{otherwise}
    \end{cases}
  \end{equation*}
  and
  \begin{equation*}
    \bar\mu_0(E)=\mu_0(E\cap[\xb-\kappa\tb,\xb+\kappa\tb]),\quad \bar\nu_0(E)=\nu_0(E\cap[\xb-\kappa\tb,\xb+\kappa\tb])
  \end{equation*}
  for any Borel set $E$. It is enough to prove
  that the theorem holds for this particular
  $\bar\zeta_0$. We have to compute the solutions
  for $\zeta_0$ and $\bar\zeta_0$. Let us denote
  $x_l=\xb-\kappa\tb$, $x_r=\xb+\kappa\tb$. We set
  $\psi=\Cb(\zeta_0)$ and
  $\bar\psi=\Cb(\bar\zeta_0)$.
  
  \textbf{Step 1.} We want to compute $\bar\psi$
  as a function of $\psi$. We denote $X_l=x_l$,
  $Y_l=x_l$, $X_r=x_r+\mu_0([x_l,x_r])$,
  $Y_r=x_r+\nu_0([x_l,x_r])$ and
  $\Omega=[X_l,X_r]\times[Y_l,Y_r]$. Let us prove
  that
  \begin{equation}
    \label{eq:compbarx1}
    \bar x_1(X)=
    \begin{cases}
      X&\text{ if }X\leq X_l\\
      x_1(X+\mu_0(-\infty,x_l))&\text{ if }X_l<X\leq X_r\\
      X-\mu_0([x_l,x_r])&\text{ if }X_r<X
    \end{cases}
  \end{equation}
  and
  \begin{equation}
    \label{eq:compbarx2}
    \bar x_2(Y)=
    \begin{cases}
      Y&\text{ if }Y\leq Y_l\\
      x_2(Y+\nu_0(-\infty,x_l))&\text{ if }Y_l<Y\leq Y_r\\
      Y-\nu_0([x_l,x_r])&\text{ if }Y_r<Y.
    \end{cases}
  \end{equation}
  From the definition \eqref{eq:defx1}, we have
  \begin{equation}
    \label{eq:defx1bar}
    \bar x_1(X)=\sup\{x'\in\Real\mid x'+\bar\mu_0(-\infty,x')<X\}
  \end{equation}
  First case: $X\leq x_l$. For any $x'$ such that
  $x'+\bar\mu_0(-\infty,x')<X$ we have
  $x'<X$. Hence, $x'<x_l$ and
  $\bar\mu_0(-\infty,x')=0$. It follows that $\bar
  x_1(X)=X$. Second case: $X_l<X\leq X_r$. For any
  $x'$ such that $x'+\bar\mu_0(-\infty,x')<X$, we
  have $x'\leq x_r$. Let us assume the opposite,
  i.e., $x'> x_r$, then
  $\bar\mu_0(-\infty,x')=\mu_0((-\infty,x')\cap[x_l,x_r])=\mu_0([x_l,x_r])$
  and therefore $x'+\mu_0([x_l,x_r])<X\leq
  x_r+\mu_0([x_l,x_r])$, which gives a
  contradiction. We can assume without loss of
  generality that $x'\geq x_l$ because for
  $x'=x_l$, $\bar\mu_0(-\infty,x')=0$ and we have
  $x'+\bar\mu_0(-\infty,x')=x'=x_l<X$. Thus we
  have $x'\in[x_l,x_r]$ and \eqref{eq:defx1bar}
  rewrites
  \begin{equation}
    \label{eq:barx1inter}
    \bar x_1(X)=\sup\{x'\in[x_l,x_r]\mid x'+\mu_0[x_l,x')<X\}.
  \end{equation}
  We now want to prove that, for $X_l\leq X\leq
  X_r$,
  \begin{equation}
    \label{eq:x1inter}
    x_1(X+\mu_0(-\infty,x_l))=\sup\{x'\in[x_l,x_r]\mid x'+\mu_0[x_l,x')<X\}.
  \end{equation}
  For any $x'$ such that
  $x'+\mu_0(-\infty,x')<X+\mu_0(-\infty,x_l)$, we
  have $x'\leq x_r$. Let us assume the opposite,
  i.e., $x'>x_r$, then $x_r+\mu_0([x_l,x_r])\leq
  x'+\mu_0([x_l,x'))<X$ implies a contradiction
  with the assumption that $X\leq X_r$. For
  $x'=x_l$, we have $x'+\mu_0(-\infty,x')<
  X+\mu_0(-\infty,x_l)$ so that we can assume
  without loss of generality that $x'\geq
  x_l$. Hence, 
  \begin{align*}
    x_1(X+\mu_0(-\infty,x_l))&=\sup\{x'\in\Real\mid
    x'+\mu_0(-\infty,x')<X+\mu_0(-\infty,x_l)\}\\
    &=\sup\{x'\in[x_l,x_r]\mid
    x'+\mu_0(-\infty,x')<X+\mu_0(-\infty,x_l)\}
  \end{align*}
  and \eqref{eq:x1inter} follows. By comparing
  \eqref{eq:barx1inter} and \eqref{eq:x1inter}, we
  get $\bar x_1(X)=x_1(X+\mu_0(-\infty,x_l))$ for
  $X_l<X<X_r$. Third case: $X_r<X$. For $x'=x_r$,
  we have $x'+\bar\mu_0(-\infty,x')\leq
  X_r=x_r+\mu_0[x_l,x_r]<X$. Hence, $\bar
  x_1(X)=\sup\{x'\in[x_r,\infty) \mid
  x'+\bar\mu_0(-\infty,x')<X\}$. Since, for
  $x'>x_r$,
  $\bar\mu_0(-\infty,x')=\mu_0([x_l,x_r])$, it
  follows that $\bar
  x_1(X)=X-\mu_0([x_l,x_r])$. This concludes the
  proof of proved \eqref{eq:compbarx1}. One proves
  in the same way \eqref{eq:compbarx2}. Let
  $\phi=(f,g)\in\Gr^2$ where $f:X\mapsto
  X+\mu_0(-\infty,x_l)$ and $g:Y\mapsto
  Y+\nu_0(-\infty,x_r)$. We denote
  $\tilde\psi=\psi\act\phi$. We have proved that
  \begin{equation}
    \label{eq:compbarx1b}
    \bar x_1(X)=\tilde x_1(X)\text{ for }X_l<X\leq X_r
  \end{equation}
  and
  \begin{equation}
    \label{eq:compbarx2b}
    \bar x_2(Y)=\tilde x_2(Y)\text{ for }X_l<Y\leq
    X_r.
  \end{equation}
  We denote $\bar\theta=\Cb(\bar\psi)$ and
  $\tilde\theta=\Cb(\tilde\psi)$.

  \textbf{Step 2.} We prove that
  \begin{equation}
    \label{eq:equXYabrti}
    \bar\X(s)=\tilde\X(s)\quad\text{and}\quad
    \bar\Y(s)=\tilde\Y(s)     
  \end{equation}
  for $s\in[s_l,s_r]$ where $s_l=\frac12(X_l+Y_l)$
  and $s_r=\frac12(X_r+Y_r)$. By using the
  definitions of $x_1$ and $x_2$, we obtain that,
  for any $x\in\Real$,
  \begin{equation}
    \label{eq:x1xbr}
    x_1(x+\mu_0(-\infty,x))=x_1(x+\mu_0(-\infty,x])=x
  \end{equation}
  and that the corresponding statement for $x_2$
  holds. Let us now prove that
  \begin{equation}
    \label{eq:Xsl}
    \tilde\X(s_l)=X_l\quad\text{ and }\quad\tilde\Y(s_l)=Y_l.
  \end{equation}
  It follows from \eqref{eq:x1xbr} that 
  \begin{equation}
    \label{eq:x1Xlx2Yr}
    \tilde x_1(X_l)=\tilde x_2(Y_l)=x_l
  \end{equation}
  as we have $\tilde
  x_1(X_l)=x_1(x_l+\mu_0(-\infty,x_l))=x_l=x_2(x_l+\nu_0(-\infty,x_l))=\tilde
  x_2(X_l)$. For any $X<X_l$, we have $\tilde
  x_1(X)\leq \tilde x_1(X_l)=x_l$. Let us prove
  that $\tilde x_1(X)<\tilde x_1(X_l)$. We assume
  the opposite, i.e., that $\tilde x_1(X)=\tilde
  x_1(X_l)=x_l$. Then, there exists an increasing
  sequence $x_i$ such that
  $\lim_{i\to\infty}x_i=x_l$ and
  $x_i+\mu_0(-\infty,x_i)<X+\mu_0(-\infty,x_l)$. It
  implies that $x_l+\mu_0(-\infty,x_l)\leq
  X+\mu_0(-\infty,x_l)$ because of the lower
  semicontinuity of
  $x\mapsto\mu_0(-\infty,x)$. Hence, $x_l\leq
  X_l$, which is a contradiction. Thus we have
  proved that, for any $X<X_l$, $\tilde
  x_1(X)<\tilde x_1(X_l)=\tilde x_2(Y_l)\leq\tilde
  x_2(2s-X)$. Hence, $\tilde X(s_l)=X_l$ and
  \eqref{eq:Xsl} holds.  By using similar
  arguments, one also proves that
  \begin{equation}
    \label{eq:Ysl}
    \tilde\X(s_r)=X_r,\quad\tilde\Y(s_r)=Y_r\quad\text{ and }\quad\tilde x_1(X_r)=\tilde x_2(Y_r)=x_r.
  \end{equation}
  and the corresponding results for $\bar\X$ and
  $\bar\Y$, that is,
  \begin{equation}
    \label{eq:barXsl}
    \bar\X(s_l)=X_l,\quad\bar\Y(s_l)=Y_l,\quad\bar x_1(X_l)=\bar x_2(Y_l)=x_l
  \end{equation}
  and
  \begin{equation}
    \label{eq:barYsl}
    \bar\X(s_r)=X_r,\quad\bar\Y(s_r)=Y_r,\quad\bar x_1(X_r)=\bar x_2(Y_r)=x_r.
  \end{equation}
  In particular, we have proved
  \eqref{eq:equXYabrti} for $s=s_l$ and
  $s=s_r$. For any $s\in(s_l,s_r)$, either
  $X_l<\bar\X(s)\leq X_r$ or
  $Y_l\leq\bar\Y(s)<Y_r$. We consider only the
  case where $X_l<\bar\X(s)\leq X_r$ as the other
  case can be treated similarly. By definition of
  $\bar\X$, there exists an increasing sequence
  $X_i$ such that $\lim_{i\to\infty}X_i=\bar\X(s)$
  and $\bar x_1(X_i)<\bar x_2(Y_i)$ where
  $Y_i=2s-X_i$. For $i$ large enough, we have
  $X_l<X_i\leq X_r$ and, by \eqref{eq:compbarx1b},
  we get
  \begin{equation}
    \label{eq:x1Xi}
    \bar x_1(X_i)=\tilde x_1(X_i)<\bar x_2(Y_i)
  \end{equation}
  If $Y_i\leq Y_r$ then $\bar x_2(Y_i)=\tilde
  x_2(Y_i)$ and 
  \begin{equation}
    \label{eq:strineqx1x2}
    \tilde x_1(X_i)<\tilde x_2(Y_i).
  \end{equation}
  If $Y_i>Y_r$ then \eqref{eq:strineqx1x2} holds
  also. Indeed, let us assume the opposite. By the
  monotonicity of $\tilde x_1$ and $\tilde x_2$,
  we get
  \begin{equation}
    \label{eq:x1Xrx1Xi} 
    \tilde x_1(X_r)\geq\tilde x_1(X_i)\geq\tilde x_2(Y_i)\geq\tilde x_2(Y_r).
  \end{equation}
  By \eqref{eq:Ysl}, we have $\tilde
  x_1(X_r)=\tilde x_1(\X(s_r))=\tilde
  x_2(\Y(s_r))=\tilde x_2(Y_r)$ and therefore
  \eqref{eq:x1Xrx1Xi} implies that $\tilde
  x_2(Y_i)=\tilde x_2(Y_r)=x_r$. From the
  definitions of $x_2$ and $\tilde x_2$, we know
  that there exists a decreasing sequence $x_j$
  such that $\lim_{j\to\infty}x_j=\tilde x_2(Y_i)$
  and $x_j+\nu_0(-\infty,x_j)\geq
  Y_i+\nu_0(-\infty,x_l)$. Letting $j$ tend to
  infinity, we get $\tilde
  x_2(Y_i)+\nu_0(-\infty,x_2(Y_i)]\geq
  Y_i+\nu_0(-\infty,x_l)$. Hence, as $\tilde
  x_2(Y_i)=\tilde x_2(Y_r)=x_r$,
  \begin{equation*}
    Y_r=x_r+\nu_0[x_l,x_r]\geq Y_i
  \end{equation*}
  which is a contradiction and we have proved that
  \eqref{eq:strineqx1x2} holds. If
  $Y_l<\bar\Y(s)$, we get
  \begin{equation}
    \label{eq:equlx1Xs}
    \tilde
    x_1(\bar\X(s))=\bar x_1(\bar\X(s))=\bar
    x_2(\bar\Y(s))=\tilde x_2(\bar\Y(s))   
  \end{equation}
  from \eqref{eq:compbarx1} and
  \eqref{eq:compbarx2}. If $Y_l=\bar\Y(s)$,
  $\tilde x_2(\bar\Y(s))=x_l=\bar x_2(\bar\Y(s))$,
  by \eqref{eq:Ysl} and \eqref{eq:barYsl} so that
  \eqref{eq:equlx1Xs} also holds. Then, it follows
  from \eqref{eq:strineqx1x2} and
  \eqref{eq:equlx1Xs} that $\bar\X(s)=\tilde\X(s)$
  and the proof of \eqref{eq:equXYabrti} is
  complete.

  \textbf{Step 3.} Let $\tilde Z=\Sb\tilde\Theta$
  and $\bar Z=\Sb\bar\Theta$. We prove that
  \begin{equation}
    \label{eq:equaltxU}
    \bar t(X,Y)=\tilde t(X,Y),\quad\bar x(X,Y)=\tilde x(X,Y),\quad\bar U(X,Y)=\tilde U(X,Y)
  \end{equation}
  for all $(X,Y)\in\Omega$. Since $\bar x_1=\tilde
  x_1$ on $[X_l,X_r]$ and $\bar x_2=\tilde x_2$ on
  $[Y_l,Y_r]$, we get, from the definition of
  $\Lb$, that
  \begin{equation*}
    \bar U_1=\tilde U_1,\quad\bar V_1=\tilde V_1,\quad \tilde J_1=\bar J_1+\tilde J_1(X_l),\quad \tilde K_1=\bar K_1+\tilde K_1(X_l)
  \end{equation*}
  on $[X_l,X_r]$ and 
  \begin{equation*}
    \bar U_2=\tilde U_2,\quad\bar V_2=\tilde V_2,\quad\tilde J_2=\bar J_2+\tilde J_2(Y_l),\quad \tilde K_2=\bar K_2+\tilde K_2(Y_l)
  \end{equation*}
  on $[Y_l,Y_r]$. Since, by \eqref{eq:equXYabrti},
  the two paths $(\bar\X,\bar\Y)$ and
  $(\tilde\X,\tilde\Y)$ in $\C(\Omega)$ are equal,
  one can check, by using the definition of the
  mapping $\Cb$, that it implies that
  \begin{equation*}
    \bar t(s)=\tilde t(s),\quad\bar x(s)=\tilde x(s),\quad\bar U(s)=\tilde U(s),\quad \tilde J(s)=\bar J(s)+\tilde J(s_l),\quad \tilde K(s)=\bar K(s)+\tilde K(s_l)
  \end{equation*}
  for $s\in[s_l,s_r]$ and
  \begin{equation*}
    \bar\V=\tilde\V,\quad\bar\W=\tilde\W
  \end{equation*}
  on $[X_l,X_r]$ and $[Y_l,Y_r]$,
  respectively. The elements $\tilde\Theta$ and
  $\bar\Theta$ are equal in $\Omega$ except that
  the energy potentials $J$ and $K$ differ up to a
  constant. However one can check that the
  governing equation \eqref{eq:goveq} is invariant
  with respect to addition of a constant to the
  energy potentials. Hence, by the uniqueness
  result of Lemma \ref{lem:globalsol} which holds
  on finite domains, we get \eqref{eq:equaltxU}.
  
  \textbf{Step 4.} We prove that there exists
  $(X_0,Y_0)\in\Omega$ such that 
  \begin{equation}
    \label{eq:ttbxxb}
    \bar
    t(X_0,Y_0)=\tb\quad\text{ and }\quad\bar x(X_0,Y_0)=\xb.    
  \end{equation}
  We have
  \begin{equation*}
    \bar x_1(X_l)=\bar x_2(Y_l)=\xb-\kappa\tb\quad\text{ and }\quad \bar x_1(X_r)=\bar x_2(Y_r)=\xb+\kappa\tb
  \end{equation*}
  so that
  \begin{equation*}
    \bar x(X_l,Y_l)=x_l\quad\text{ and }\quad \bar x(X_r,Y_r)=x_r.
  \end{equation*}
  Let $P=(X_r,Y_l)$ denote the right-corner of
  $\Omega$. We have
  \begin{equation*}
    \bar x(P)-x_l=\int_{X_l}^{X_r}\bar x_X(X,Y_l)\,dX=\int_{X_l}^{X_r}c(\bar U)\bar t_X(X,Y_l)\,dX
  \end{equation*}
  and
  \begin{equation*}
    \bar t(P)=\int_{X_l}^{X_r}\bar t_X(X,Y_l)\,dX.
  \end{equation*}
  Hence, using the positivity of $\bar t_X$ and
  the assumption that $\frac1\kappa<c<\kappa$, we
  get
  \begin{equation}
    \label{eq:boundxpl}
    \bar x(P)-x_l\leq\kappa \bar t(P).
  \end{equation}
  Similarly, one proves that $x_r-\bar
  x(P)\leq\kappa\bar t(P)$, which added to
  \eqref{eq:boundxpl}, yields
  $x_l-x_r\leq2\kappa\bar t(P)$ or, after plugging
  the definition of $x_l$ and $x_r$,
  \begin{equation}
    \label{eq:uppbdt}
    \tb\leq\bar t(P).
  \end{equation}
  The mapping $(X,Y)\mapsto (\bar t(X,Y),\bar
  x(X,Y))$ is surjective from $\Real^2$ to
  $\Real^2$ and there exists
  $P_0=(X_0,Y_0)\in\Real^2$, which may not be
  unique, such that \eqref{eq:ttbxxb} is
  fulfilled. Let us assume that
  \begin{equation}
    \label{eq:contttb}
    (\bar t(X,Y),\bar x(X,Y))\neq(\tb,\xb)
  \end{equation}
  for all $(X,Y)\in\Omega$. By using
  \eqref{eq:uppbdt} and the monotonicity of the
  function $t$ and $x$ in the $X$ and $Y$
  directions, we infer that either $X_0>X_r$ and
  $Y_l\leq Y$ or $Y<Y_l$ and $X_0\leq X_r$. We
  treat only the first case as the other case can
  be treated similarly. We have $Y_0\leq Y_r$ as,
  otherwise, $x(X_0,Y_0)\geq x(X_r,Y_r)=x_r$. We
  introduce the point
  $P_1=(X_r,Y_0)\in\Omega$. Let us assume
  $x(P_1)\geq x(P_0)=\xb$. By the monotonicity of
  $x$, we get that $x(P_1)=x(P_0)$ and
  $x_X(X,Y_0)=0$ for $X\in[X_r,X_0]$. It implies
  that $t_X(X,Y_0)=0$ for $x\in[X_l,X_r]$ and
  therefore $t(P_1)=t(P_0)=\tb$. However this
  contradicts the original assumption
  \eqref{eq:contttb} and we must have that
  $x(P_1)<\xb$. By following the same type of
  computation that lead to \eqref{eq:boundxpl}, we
  now get
  \begin{equation*}
    \xb>x(P_1)=x_r+\int_{Y_r}^{Y_0}x_Y(X_r,Y)\,dY\geq x_r-\kappa t(P_1)\geq x_r-\kappa\tb\geq\xb,
  \end{equation*}
  which is a contradiction. Hence,
  \eqref{eq:contttb} cannot hold and we have
  proved \eqref{eq:ttbxxb}.

  \textbf{Step 5.} We now conclude the
  argument. By definition, we have $\bar
  u(\tb,\xb)=\bar U(X_0,Y_0)$ for any $(X_0,Y_0)$
  such that \eqref{eq:ttbxxb} holds. By
  \eqref{eq:equaltxU}, it follows that $\tilde
  U(X_0,Y_0)=\bar U(X_0,Y_0)=\bar u(\tb,\xb)$ and
  $\tilde t(X_0,Y_0)=\tb$ and $\tilde
  x(X_0,Y_0)=\xb$. It gives $U(f(X_0),g(Y_0))=\bar
  u(\tb,\xb)$ and $t(f(X_0),g(Y_0))=\tb$ and
  $x(f(X_0),g(Y_0))=\xb$, so that $u(\tb,\xb)=\bar
  u(\tb,\xb)$, by \eqref{eq:reluU}.
  \begin{figure}[h]
    \centering
    \includegraphics[width=15cm]{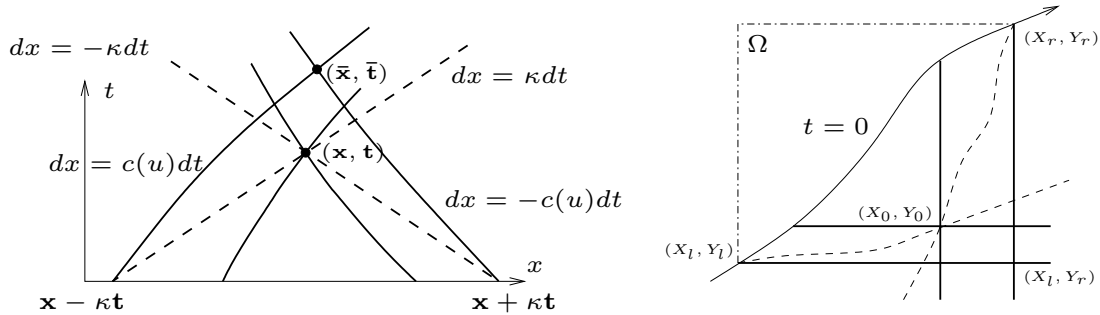}
    \caption{We have $\tb=t(X_0,Y_0)$ and
      $\xb=x(X_0,Y_0)$. In the new set of
      coordinates $(X,Y)$, the domain of dependence
      is given by rectangles. We define the points
      $(X_l,Y_l)$ and $(X_r,Y_r)$ so that they
      correspond to the points $(\xb-\kappa\tb,0)$
      and $(\xb+\kappa\tb,0)$. It then follows, from
      the boundedness of the function $c(u)$ that
      $(X_0,Y_0)$ is contained in $\Omega$.}
    \label{fig:finitespeed}
  \end{figure}  

\end{proof}

\section{Examples}
\label{sec:example}

There is a lack of explicit solutions for any choice of $c$ except the
trivial case of the linear wave equation for which $c$ is constant.
We here discuss two examples; first the linear case with general initial data, and second, a nonlinear case with very simple initial data.  

\subsection{The linear wave equation} \label{sec:linearWE}

In the case of the linear wave equation, the
coefficient $c$ is constant and the equivalent
system \eqref{eq:goveq} rewrites as 
\begin{equation}
  \label{eq:eqsyslin}
  Z_{XY}=0.
\end{equation}
We consider general initial data
$(u_0,R_0,S_0,\mu_0,\nu_0)\in\D$, let
$(\psi_1,\psi_2)=\Lb(u_0,R_0,S_0,\mu_0,\nu_0)$ and
$\Theta=(\X,\Y,\Z,\V,\W)=\Cb(\psi_1,\psi_2)$. From
\eqref{eq:eqsyslin}, we get that
\begin{equation*}
  Z_X(X,Y)=\V(X)\text{ and }Z_Y(X,Y)=\W(X).
\end{equation*}
Given a point $(X,Y)\in\Real^2$, we denote
$s_0=\Y^{-1}(Y)$ and $s_1=\X^{-1}(X)$ so that
\begin{equation}
  \label{eq:defs0s1lin}
  \Y(s_0)=Y\text{ and }\X(s_1)=X.
\end{equation}
We have
\begin{equation*}
  Z(X,Y)=\Z(s_0)+\int_{\X(Y)}^XZ_X(\bar X,Y)\,d\bar X=\Z(s_0)+\int_{\X(Y)}^X\V(\bar X)\,d\bar X
\end{equation*}
and
\begin{equation*}
  Z(X,Y)=\Z(s_1)+\int_{\Y(X)}^YZ_Y(X,\bar Y)\,d\bar Y=\Z(s_1)+\int_{\Y(X)}^Y\W(\bar Y)\,d\bar Y.
\end{equation*}
By averaging these two equations, we get
\begin{equation*}
  Z(X,Y)=\frac12(\Z(s_0)+\Z(s_1))+\frac12(\int_{\X(Y)}^X\V(\bar X)\,d\bar X+\int_{\Y(X)}^Y\W(\bar Y)\,d\bar Y).
\end{equation*}
After a change of variables, it yields
\begin{equation}
  \label{eq:ZXYlinear}
  Z(X,Y)=\frac12(\Z(s_0)+\Z(s_1))+\frac12\int_{s_0}^{s_1}(\V(\X(s))\dot\X(s)-\W(\Y(s))\dot\Y(s))\,ds.
\end{equation}
We recall that for $\Theta\in\G$, we have
$\V_2(\X(s))\dot\X(s)=\W_2(\Y(s))\dot\Y(s)$, see,
for example, \eqref{eq:V2W2equal}.  For the first
component $Z_1(X,Y)$ that we denote $t(X,Y)$, we
have $\Z_1(s)=t(s)=0$ for all $s\in\Real$ and,
after using \eqref{eq:reltx1}, we get
\begin{align}
  \notag
  t(X,Y)&=\frac1{2c}\int_{s_0}^{s_1}(\V_2(\X(s))\dot\X(s)+\W_2(\Y(s))\dot\Y(s))\,ds\\
  \notag
  &=\frac1{2c}(\Z_2(s_1)-\Z_2(s_0)),&\text{ by \eqref{eq:relZVWG} }\\
  \label{eq:tlinear}
  &=\frac1{2c}(x_1(X)-x_2(Y)), &\text{ by
    \eqref{eq:barZmapC2}}.
\end{align}
As far as the second component $Z_2(X,Y)=x(X,Y)$
is concerned, it follows directly from
\eqref{eq:ZXYlinear} and \eqref{eq:barZmapC2} that
\begin{equation}
  \label{eq:xlinear}
  x(X,Y)=\frac12(x_1(X)+x_2(Y)).
\end{equation}
For the third component $Z_3(X,Y)=U(X,Y)$, we have
$\Z_3(s_0)=u_0(x_1(X))$ and
$\Z_3(s_1)=u_0(x_2(Y))$. After using
\eqref{eq:defV1V2} and \eqref{eq:V1V3V2W3}, we
get, after a change of variables, that
\begin{equation*}
  \int_{s_0}^{s_1}(\V_3(\X(s))\dot\X(s)=\int_{s_0}^{s_1}(R_0(x_1(\X(s)))x_1'(\X(s))\dot\X(s)\,ds=\frac1{2c}\int_{x_2(Y)}^{x_1(X)}R_0(x)\,dx.
\end{equation*}
We use the fact that
$x_1(\X(s_1))=x_2(\Y(s_1))=x_2(Y)$, which follows
from \eqref{eq:x1eqx2} and
\eqref{eq:defs0s1lin}. Similarly, we obtain that
$\int_{s_0}^{s_1}W(\Y(s))\dot\Y(s)\,ds=-\frac1{2c}\int_{x_1(X)}^{x_2(Y)}S_0(x)\,dx$. Hence,
\eqref{eq:ZXYlinear} yields
\begin{equation}
  \label{eq:Ulinear}
  U(X,Y)=\frac12(u_0(x_1(X))+u_0(x_2(Y)))+\frac1{4c}\int_{x_1(X)}^{x_2(Y)}(R_0+S_0)\,dx.
\end{equation}
From \eqref{eq:tlinear} and \eqref{eq:xlinear}, it
follows that $x_1(X)=x(X,Y)-ct(X,Y)$ and
$x_2(Y)=x(X,Y)+ct(X,Y)$. Therefore, after using
\eqref{eq:reluU}, we recover d'Alembert's formula
from \eqref{eq:Ulinear}, i.e.,
\begin{equation*}
  u(t,x)=\frac12(u_0(x-ct)+u_0(x+ct))+\frac{1}{4c}\int_{x-ct}^{x+ct}(R_0+S_0)\,dx
\end{equation*}
for the solution of the linear wave equation. Let
us now look at the energy. We use the same
notation as in the proof of Theorem
\ref{th:energconcentration}. For a given time $t$,
$(\X(t,s),\Y(t,s))$ denotes the curve
corresponding to a given time, that is,
$t(\X(t,s),\Y(t,s))=t$ (Beware of the notation,
$t(\cdot,\cdot)$ denotes a function while $t$,
without argument, denotes a constant). For any
point $x$, we have
\begin{equation*}
  \mu(t)(-\infty,x)=\int_{x(\X(t,s),\Y(t,s))<x}J_X(\X(t,s),\X(t,s))\X_s(t,s)\,ds.
\end{equation*}
From \eqref{eq:tlinear} and \eqref{eq:xlinear}, we
get that $x(\X(t,s),\Y(t,s))<x$ if and only if
$x_1(\X(t,s))<x+ct$. Since
$J_X(X,Y)=\V_4(X)=J_1'(X)$, we get
\begin{equation*}
  \mu(t)(-\infty,x)=\int_{x_1(\X(t,s))<x+ct}J_1'(\X(t,s))\X_s(t,s)\,ds.
\end{equation*}
After a change of variables, it yields
\begin{equation*}
  \mu(t)(-\infty,x)=\int_{x_1(X)<x+ct}J_1'(X)\,dX=\mu_0(-\infty,x+ct).
\end{equation*}
Hence, for any Borel set $B$, we have
\begin{equation*}
  \mu(t)(B)=\mu_0(B+ct).
\end{equation*}
Similarly, we get
\begin{equation*}
  \nu(t)(B)=\nu_0(B-ct).
\end{equation*}

\subsection{An example with singular initial
  data} \label{sec:nonlin} Let
\begin{subequations}
  \label{eq:initnum}
  \begin{equation}
    \label{eq:initnumu}
    u_0(x)=1,\ R_{0}(x)=S_{0}(x)=0
  \end{equation}
  for all $x\in\Real$ and
  \begin{equation}
    \label{eq:initnummunu}
    \nu_0=2\mu_0=2\delta
  \end{equation}
\end{subequations}
where $\delta$ denotes the Dirac delta
function. Our intention is to consider initial
data for which all the energy is concentrated in a
set of zero measure (in this case the origin) and
that is why we choose $u_0$ equal to a constant.\footnote{If
we choose $u_0=0$ (the only constant in $L^2(\Real)$) then, since $c'(0)=0$, one can check
from the governing equations \eqref{eq:goveq} that
there is no evolution of the solution, and we have
that $u(t,x)=0$, $2\mu(t)=\nu(t)=2\delta$ is the
conservative solution.} Since $u_0$ does not belong
to $L^2(\Real)$, the theory we have developed
does not apply directly. However, we can consider
the sequence of solutions
$(u_N,R_N,S_N,\mu_N,\nu_N)$ given by the semigroup
$\bar S_t$ for the following initial data
\begin{equation*}
u_0^N(x)=
\begin{cases}
  1&\text{ for }x\in[-N,N]\\
  0&\text{ otherwise }
\end{cases}
\end{equation*}
and $R_{0}^N(x)=S_{0}^N(x)=0$,
$\nu_0=2\mu_0=2\delta$. Given a compact domain in
time and space, we know that for $N$ large enough
the solutions will coincide on this compact
domain due to the finite time of propagation, see Theorem \ref{thm:finite}. Thus we can 
define the solution of
\eqref{eq:nvw} for the initial data
\eqref{eq:initnum} as the limit of the solutions
$u^N$ when $N$ tends to $\infty$. We see that,
by using the same type of construction, we can
actually construct solutions for any initial data
such that $u_0, R_0, S_0$ belong to
$L^2_{\text{loc}}(\Real)$ and $\mu$, $\nu$ are
(not necessarily finite) Radon measures (note
that, by definition, a Radon measure is finite on
compacts). 

For the initial data $(u_0,R_0,S_0,\mu_0,\nu_0)$
given by \eqref{eq:initnum}, let us denote
$(\psi_1,\psi_2)=\Lb(u_0,R_0,S_0,\mu_0,\nu_0)$ as
defined in Definition \ref{def:mappingL} and
$\Theta=(\X,\Y,\Z,\V,\W)=\Cb(\psi_1,\psi_2)$ defined by Definition~\ref{def:C}.
We first find
 \begin{align}
    x_1(X)=
    \begin{cases}
      X&\text{ if }X<0,\\
      0&\text{ if }0\leq X\leq 1,\\
      X-1&\text{ if } X >1,
    \end{cases} \qquad 
  x_2(Y)=
    \begin{cases}
      Y&\text{ if }Y<0,\\
      0&\text{ if }0\leq Y\leq 2,\\
      Y-2&\text{ if } Y >2,
    \end{cases}
  \end{align}
which yields
\begin{align*}
  \Gamma_0 & =\{(X,Y)\mid x_1(X)=x_2(Y)\} \\
&= \{(X,X)\mid X\le 0\}\cup
\big([0,1]\times[0,2]\big) \cup \{(X,X+1) \mid X\ge 1\}.
\end{align*}
Furthermore
\begin{align*}
   J_1(X)&=
    \begin{cases}
      0&\text{ if }X<0,\\
      X&\text{ if }0\leq X\leq 1,\\
      1&\text{ if } X >1,
    \end{cases} & 
  J_2(Y)&=
    \begin{cases}
      0&\text{ if }Y<0,\\
      Y&\text{ if }0\leq Y\leq 2,\\
      2&\text{ if } Y >2,
    \end{cases} \\
 U_1(X)&=1, & U_2(Y)&=1, \\
V_1(X)&=0, & V_2(Y)&=0, \\
  K_1(X)&=
    \begin{cases}
      0&\text{ if }X<0,\\
      X/c(1)&\text{ if }0\leq X\leq 1,\\
      1/c(1)&\text{ if } X >1,
    \end{cases} & 
  K_2(Y)&=
    \begin{cases}
      0&\text{ if }Y<0,\\
      -Y/c(1)&\text{ if }0\leq Y\leq 2,\\
      -2/c(1)&\text{ if } Y >2.
    \end{cases} 
\end{align*}
Next, we obtain
\begin{subequations}
  \label{eq:initnumgoveq}
  \begin{align}
    \X(s)=
    \begin{cases}
      s&\text{ if }s<0,\\
      0&\text{ if }0\leq s<1,\\
      2s-2&\text{ if }1\leq s<3/2,\\
      s-1/2&\text{ if }3/2\leq s,
    \end{cases}&& \Y(s)=
    \begin{cases}
      s&\text{ if }s<0,\\
      2s&\text{ if }0\leq s<1,\\
      2&\text{ if }1\leq s<3/2,\\
      s+1/2&\text{ if }3/2\leq s,
    \end{cases}
  \end{align}
  and
  \begin{align}
    x(s)=
    \begin{cases}
      s&\text{ if }s<0,\\
      0&\text{ if }0\leq s<3/2,\\
      s-3/2&\text{ if }3/2\leq s,
    \end{cases}
  \end{align}
  and $U(s)=1$ and
  \begin{align}
    J(s)=
    \begin{cases}
      0&\text{ if }s<0,\\
      2s&\text{ if }0\leq s<3/2,\\
      3&\text{ if }3/2\leq s,
    \end{cases}
    && K(s)=
    \begin{cases}
      0&\text{ if }s<0,\\
      -\frac{2s}{c(1)}&\text{ if }0\leq s<1,\\
      \frac{2(s-2)}{c(1)}&\text{ if }1\leq s<\frac32,\\
      -\frac{1}{c(1)}&\text{ if }3/2\leq s,
    \end{cases}
  \end{align}
  \begin{align}
    c(1)\V_1(X)=\V_2(X)&=
    \begin{cases}
      1/2&\text{ if }X<0,\\
      0&\text{ if }0\leq X<1,\\
      1/2&\text{ if }1\leq X,
    \end{cases}
    \\ c(1)\W_1(Y)=-\W_2(Y)&=
    \begin{cases}
      1/2&\text{ if }Y<0,\\
      0&\text{ if }0\leq Y<2,\\
      1/2&\text{ if }2\leq Y,
    \end{cases}
  \end{align}
  and $\V_3=\W_3=0$ and
  \begin{align}
    c(1)\V_5(X)=\V_4(X)&=
    \begin{cases}
      0&\text{ if }X<0,\\
      1&\text{ if }0\leq X<1,\\
      0&\text{ if }1\leq X,
    \end{cases}
    \\c(1)\W_5(Y)=-\W_4(Y)&=
    \begin{cases}
      0&\text{ if }Y<0,\\
      1&\text{ if }0\leq Y<2,\\
      0&\text{ if }2\leq Y.
    \end{cases}
  \end{align}
\end{subequations}

In the case of the linear wave equation, the
solution is explicit. In Figure
\ref{fig:initandsollin}, we plot the curve
$(\X,\Y)$ and the curve $t(X,Y)=T$, for a given
$T$. In Figure \ref{fig:initandsollin}, the
letters A to F denote the regions which are
delimited by the neighboring solid or dashed
black lines.  The values of $Z$ in these different
regions are given in Table~\ref{tab:valueZ}. \renewcommand{\arraystretch}{2}

\begin{figure}
  \centering
  \includegraphics{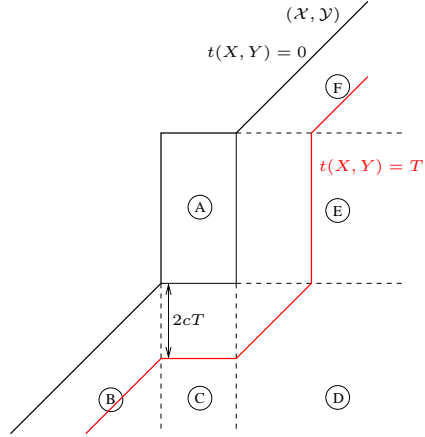}
  \caption{Plot of the initial data curve and the curve for a given time $T$.}
  \label{fig:initandsollin}
\end{figure}

\begin{table}[h]
\begin{center}
  \begin{tabular}[h]{|c|c|c|c|c|c|c|}
    \hline
    &A&B&C&D&E&F\\
    \hline
    $t(X,Y)$&$0$&$\frac{X-Y}{2c}$&$-\frac{Y}{2c}$&$\frac{X-Y-1}{2c}$&$\frac{X-1}{2c}$&$\frac{X-Y+1}{2c}$\\
    \hline
    $x(X,Y)$&$0$&$\frac{X+Y}2$&$\frac{Y}2$&$\frac{X+Y-1}2$&$\frac{X-1}{2}$&$\frac{X+Y-3}{2}$\\
    \hline
    $U(X,Y)$&$1$&$1$&$1$&$1$&$1$&$1$\\
    \hline
    $J(X,Y)$&$X+Y$&$0$&X&$1$&$1+Y$&$3$\\
    \hline
    $K(X,Y)$&$\frac{X-Y}{c}$&$0$&$\frac{X}c$&$\frac{1}c$&$\frac{1-Y}c$&$-\frac{1}c$\\
    \hline
  \end{tabular}
\end{center}
\caption{The values of the solution $Z=(t,x,U,J,K)$ of the linear wave equation for the initial data given by \eqref{eq:initnum} in the different domains of the plane (see Figure \ref{fig:initandsollin}, the regions A-F are delimited by the dashed and solid dark lines).}
\label{tab:valueZ}
\end{table}

If we consider the choice
\begin{equation}
  \label{eq:valc}
  c(u)^2=\beta\cos^2 u+\alpha\sin^2 u
\end{equation}
where $\alpha$ and $\beta$ are strictly positive
constants, then no explicit solutions are
available for the initial data
\eqref{eq:initnum}. Due to the finite speed of
propagation, we know that
\begin{equation}
  \label{eq:exSOL}
 \text{$U(X,Y)=1$ on 
$(-\infty,0\,]\times(-\infty,0\,]
\bigcup\,  [\,0,1\,]\times[\,0,2\,] \bigcup\, [\,1,\infty)\times[\,2,\infty)$.}
\end{equation}
The measures $\nu$ and $\mu$ will become regular
measures for $t$ nonzero.  We take $\alpha=0.2$
and $\beta=0.1$. The solution is illustrated on
Figs.~\ref{fig:difftimes},
\ref{fig:isotimes}--\ref{fig:paramsurface}.  Here we have
used the numerical method described in Section
\ref{sec:num}.

\section{A numerical method for conservative solutions}
\label{sec:num}

Next we describe a general numerical approach to obtain conservative
solutions of the NVW equation. Traditional finite difference methods
will not yield conservative solutions, and we here use the full
machinery of the analytical approach to derive an efficient numerical method for conservative solutions.

We discretize the problem as follows. Given $N$,
$\smin$ and $\smax$, we set $h=(\smax-\smin)/N$
and $s_i=\smin+ih$ for $i=0,\ldots,N$. Let
\begin{align*}
  X_i=\X(s_i),&&Y_j=\Y(s_j),&&P_{i,j}=[X_i,X_j]
\end{align*}
for $i=0,\ldots,N$ and $j=0,\ldots,N$. We compute
the solution of \eqref{eq:goveq} on the domain
$\Omega=[X_0,X_N]\times[Y_0,Y_N]$. The algorithm
follows the same type of iteration as in the proof
of Lemma \ref{lem:globalsol}, and we use the same
notation here.  We approximate the form
$Z_X(X,Y_{j})\,dX$ on the interval $[X_{i-1},X_i]$
by the constant $V_{i,j}$ and the form
$Z_Y(X_{i},Y)\,dY$ on the interval $[Y_j,Y_{j+1}]$
by the constant $W_{i,j}$. We denote by
$Z_{i,j}^h$ and $Z_{i,j}^v$ the approximation of
$Z$ on the segments $P_{i-1,j}-P_{i,j}$ and
$P_{i,j}-P_{i,j+1}$, respectively. The initial
curve is approximated on the piecewise horizontal
and vertical line $\bigcup_{i=1}^{N-1}
([P_{i,i},P_{i,i+1}]\cup[P_{i,i+1},P_{i+1,i+1}])$
and we set
\begin{equation*}
  Z_{i,i}^h=\Z(s_i),\quad\text{ and }\quad
  V_{i,i}=\frac1{X_{i}-X_{i-1}}\int_{X_{i-1}}^{X_{i}}\V(X)\,dX\quad\text{ for }i=1,\ldots,N,
\end{equation*}
\begin{equation*}
  Z_{i,i}^v=\Z(s_i),\quad\text{ and }\quad  W_{i,i}=\frac1{Y_{i+1}-Y_i}\int_{Y_i}^{Y_{i+1}}\W(X)\,dX\quad\text{ for }i=0,\ldots,N-1,
\end{equation*}
where $\Z$, $\V$ and $\W$ are given by
\eqref{eq:initnumgoveq}. If $X_{i}-X_{i-1}$
(respectively $Y_{j+1}-Y_j$) is equal to zero, then
we set $V_{i,i}$ (respectively $W_{i,i}$) to zero
or an arbitrary value (this value will not have
any impact on the computed solution). We compute
the solution iteratively on vertical and
horizontal strips: Given $n\in\{0,\ldots,N\}$, we
assume that the values of
\begin{subequations}
  \label{eq:compvalues}
  \begin{alignat}{3}
    &Z_{i,j}^h,\, V_{i,j}&&\quad\text{ for }1\leq i\leq n,&&\quad 0\leq j\leq n,\\
    &Z_{i,j}^v,\, W_{i,j}&&\quad\text{ for }0\leq i\leq n,&&\quad 0\leq j\leq n,
  \end{alignat}
\end{subequations}
have been computed. Then, we set iteratively, for
$j=n+1,\ldots,1$, 
\begin{align*}
  Z_{n+1,j-1}^h&=Z_{n+1,j}^h-(Y_{j}-Y_{j-1})W_{n,j-1},\\
  V_{n+1,j-1}&=V_{n+1,j}-(Y_{j}-Y_{j-1})F(\frac12(Z_{n+1,j}^h+Z_{n,j-1}^v))(V_{n+1,j},W_{n,j-1}),\\
  Z_{n+1,j-1}^v&=Z_{n,j-1}^v+(X_{n+1}-Y_{n})V_{n+1,j},\\
  W_{n+1,j-1}&=W_{n,j-1}+(X_{n+1}-X_n)F(\frac12(Z_{n+1,j}^h+Z_{n,j-1}^v))(V_{n+1,j},W_{n,j-1}),
\end{align*}
and, for $i=n+1,\ldots,2$,  
\begin{align*}
  Z_{i-1,n+1}^h&=Z_{i-1,n}^h+(Y_{n+1}-Y_{n})W_{i-1,n},\\
  V_{i-1,n+1}&=V_{i-1,n}+(Y_{n+1}-Y_{n})F(\frac12(Z_{i-1,n}^h+Z_{i-1,n}^v))(V_{i-1,n},W_{i-1,n}),\\
  Z_{i-1,n+1}^v&=Z_{i,n+1}^v-(X_{i}-X_{i-1})V_{i,n+1},\\
  W_{i-1,n+1}&=W_{i,n+1}-(X_{i}-X_{i-1})F(\frac12(Z_{i,n+1}^h+Z_{i,n+1}^v))(V_{i,n+1},W_{i,n+1}),
\end{align*}
and
\begin{align*}
  Z_{0,n+1}^v&=Z_{1,n+1}^v-(X_{1}-X_{0})V_{1,n+1},\\
  W_{0,n+1}&=W_{1,n+1}-(X_{1}-X_{0})F(\frac12(Z_{1,n+1}^h+Z_{1,n+1}^v))(V_{1,n+1},W_{1,n+1}).
\end{align*}
We have defined the quantities in
\eqref{eq:compvalues} for $n$ replaced by
$n+1$. By induction we have computed the solution
on the whole domain $\Omega$. To compute the
solution at a given time $T$, we have to extract a
curve $(\X,\Y)\in\C$ such that $t(\X(s),\Y(s))=T$
for all $s\in\Real$. We proceed by iteration and
compute a set of grid points that approximates
well such a curve, for example, by taking
\begin{align*}
  i(k)&=\sup\{ i\in\{0,\ldots,N\} \mid  t(X_{i,k},Y_{i,k})<T\},\\
  j(k)&=k,
\end{align*}
for $k=0,\ldots,N$.  For a given $T$, the function
$u(T,x)$ can be seen as the curve $(x,u(T,x))$ in
$\Real^2$ which is parametrized by $x\in\Real$, and
we approximate this curve by the points
\begin{equation*}
  (x(X_{i(k),j(k)},Y_{i(k),j(k)}),U(X_{i(k),j(k)},Y_{i(k),j(k)}))
\end{equation*}
for $k=0,\ldots,N$. This method has been used to
produce the results presented in Figure~\ref{fig:difftimes}.

\begin{figure}[h]
  \centering
  \includegraphics[width=10cm]{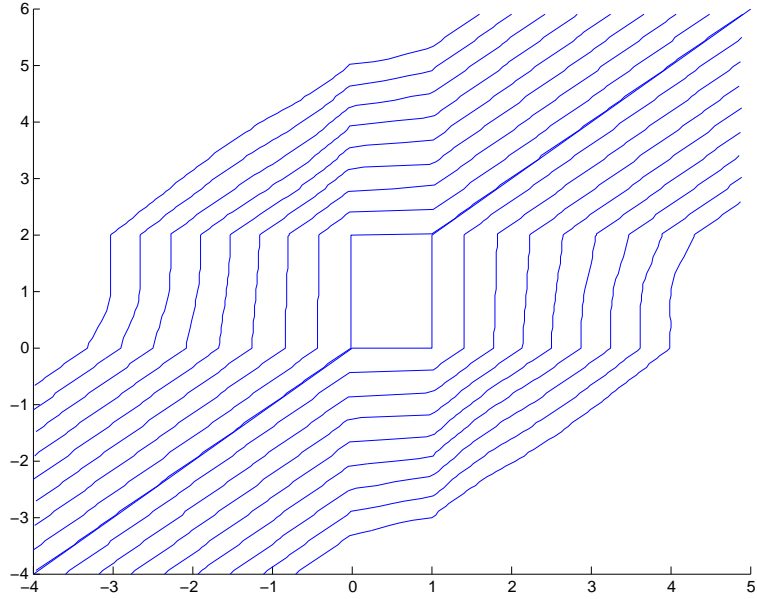}
  \caption{Plot of the isotimes, that is, the
    curves for which $t(X,Y)$ is a constant. Note
    the box in the middle in which $t$ is constant
    and equal to zero.}
  \label{fig:isotimes}
\end{figure}

\begin{figure}[h]
  \centering
  \includegraphics[width=10cm]{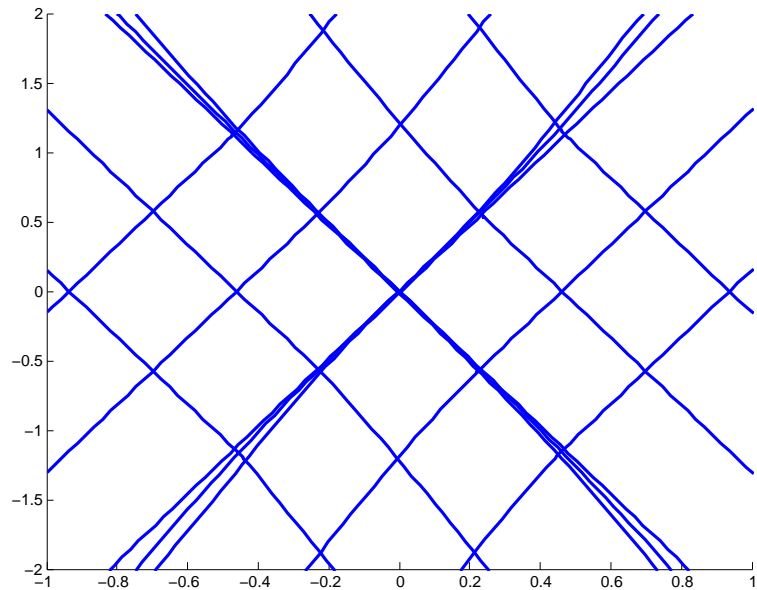}
  \caption{Forward and backward characteristics
     in the $(x,t)$ plane. Both families have
    a point of intersection at zero. It
    corresponds to the point where the measures
    $\mu$ and $\nu$ become singular.}
  \label{fig:charact}
\end{figure}

\begin{figure}[h]
  \centering
  \includegraphics[width=11.5cm]{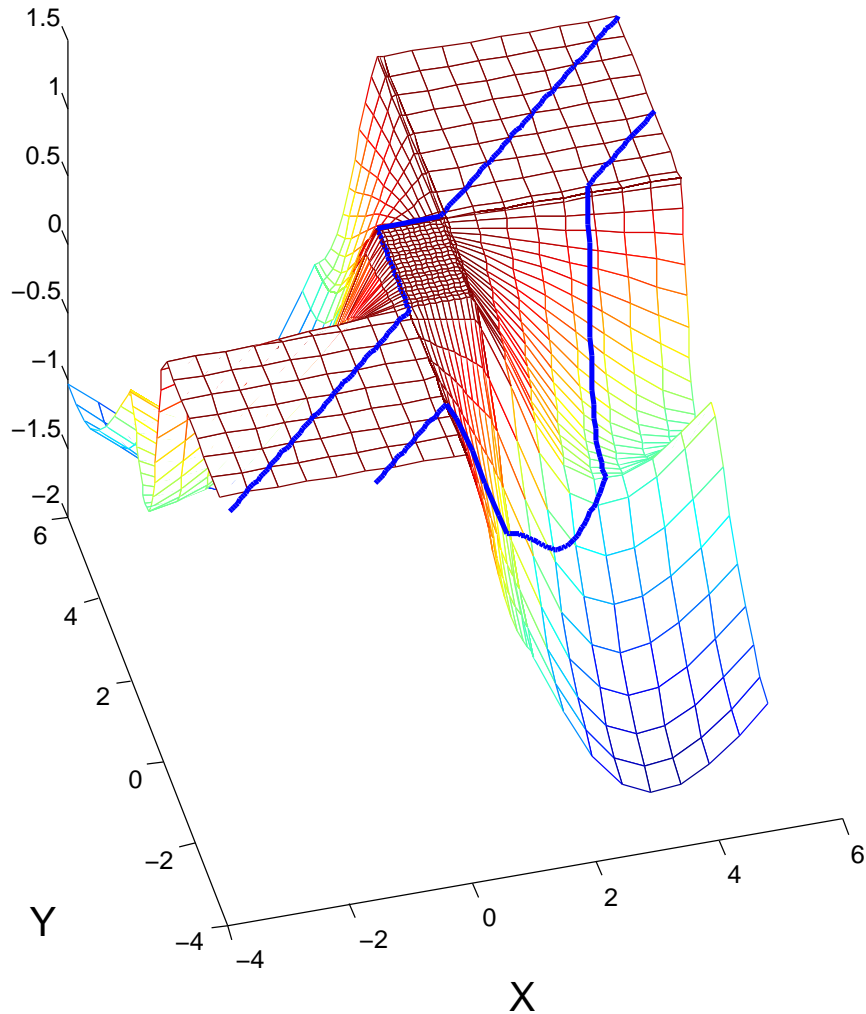}
  \caption{Plot of $U(X,Y)$. The blue curves 
   single out the solution $U(X,Y)$ for 
    times $t=0$ and $t=3$.}
  \label{fig:plotU}
\end{figure}

\begin{figure}[h]
  \centering
  \includegraphics[width=11.5cm]{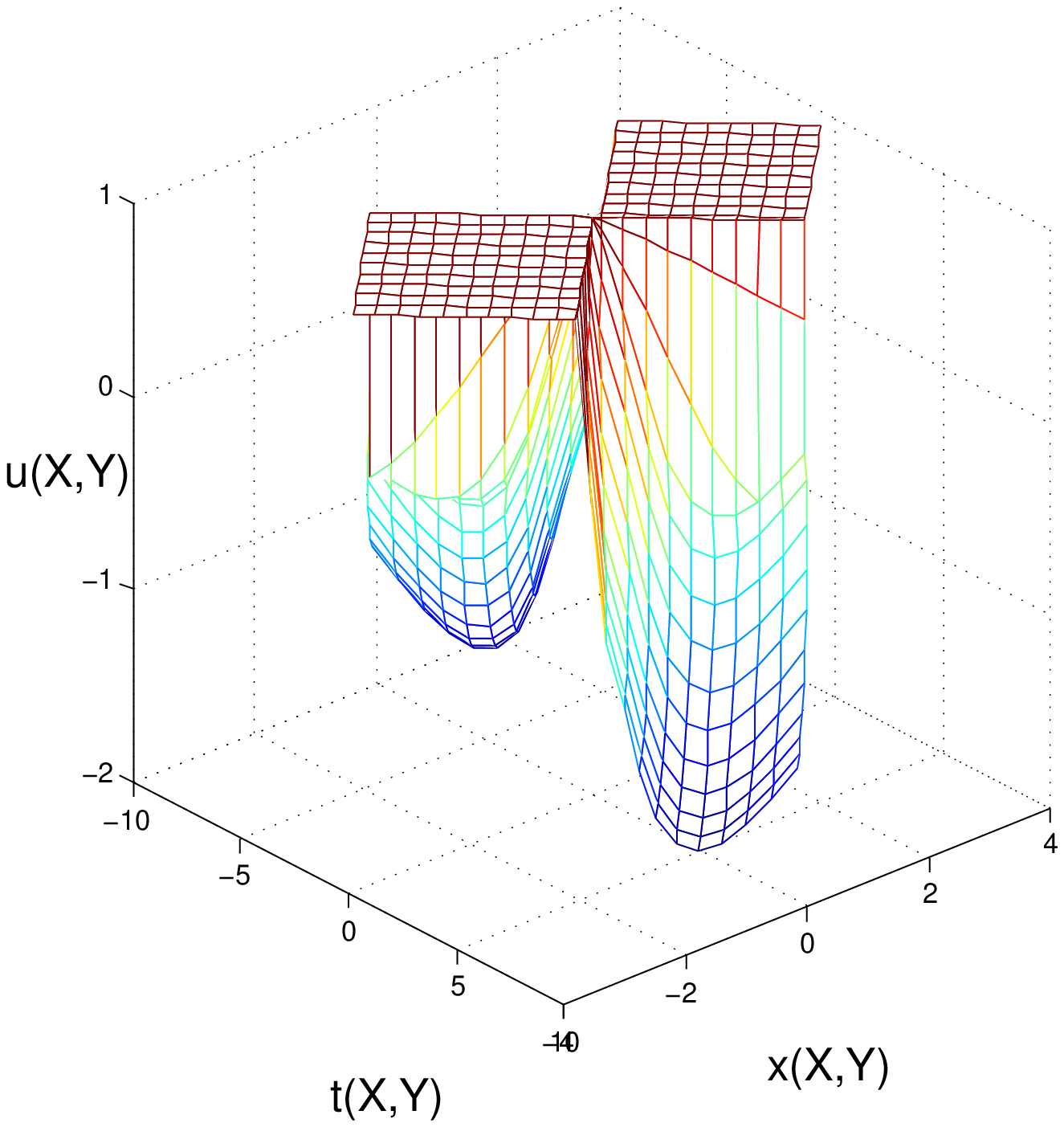}
  \caption{Plot of the surface
    $(t(X,Y),x(X,Y),U(X,Y))$ parametrized by
    $(X,Y)$ and which is approximated by
    $(t(P_{i,j}),x(P_{i,j}),U(P_{i,j}))$ for
    $i,j=0,\ldots,N$.}
  \label{fig:paramsurface}
\end{figure}

\bibliographystyle{plain} 

\begin{thebibliography}{10}

\bibitem{Ambrosio}
L. Ambrosio, N. Fusco, and D. Pallara.
\newblock {\em Functions of Bounded Variation and Free Discontinuity Problems}.
\newblock  Oxford University
  Press, New York, 2000.

\bibitem{brenier:09}
Y. Brenier.
\newblock {$L\sp 2$} formulation of multidimensional scalar conservation laws.
\newblock {\em Arch. Ration. Mech. Anal.}, 193(1):1--19, 2009.

\bibitem{BreCons:06}
A. Bressan and A. Constantin.
\newblock Global dissipative solutions of the {C}amassa--{H}olm equation.
\newblock {\em Anal. Appl. (Singap.)}, 5(1):1--27, 2007.

\bibitem{BrHoRa:09}
A. Bressan, H. Holden, and X. Raynaud.
\newblock Lipschitz metric for the Hunter--Saxton equation.
\newblock {\em Preprint, Submitted}, 2009.

\bibitem{BreZhe:06}
A. Bressan and Y. Zheng.
\newblock Conservative solutions to a nonlinear variational wave equation.
\newblock {\em Comm. Math. Phys.}, 266(2):471--497, 2006.

\bibitem{Folland}
G:~B. Folland.
\newblock {\em Real {A}nalysis}.
\newblock  Wiley,  New York, second edition, 1999.

\bibitem{garabedian}
P.~R. Garabedian.
\newblock {\em Partial Differential Equations}.
\newblock AMS Chelsea Publishing, Providence, RI, 1998.

\bibitem{GlaHunZh:96}
R.~T. Glassey, J.~K. Hunter, and Y. Zheng.
\newblock Singularities of a variational wave equation.
\newblock {\em J. Differential Equations}, 129(1):49--78, 1996.

\bibitem{HolRay:07}
H. Holden and X. Raynaud.
\newblock Global conservative solutions of the {C}amassa--{H}olm equation---a
  {L}agrangian point of view.
\newblock {\em Comm. Partial Differential Equations}, 32(10-12):1511--1549,
  2007.

\bibitem{HS:91}
J.~K. Hunter and R. Saxton.
\newblock Dynamics of director fields.
\newblock {\em SIAM J. Appl. Math.}, 51(6):1498--1521, 1991.

\bibitem{Saxt:89}
R.~A. Saxton.
\newblock Dynamic instability of the liquid crystal director.
\newblock In {\em Current Progress in Hyperbolic Systems: {R}iemann Problems
  and Computations ({B}runswick, {ME}, 1988)}, (ed. W.~B.~Lindquist). 
\newblock  {\em Contemp.
  Math.}, Vol. 100,  Amer. Math. Soc., Providence, RI, 1989, pp. 325--330.

\bibitem{ZhaZhen:98}
P. Zhang and Y. Zheng.
\newblock On oscillations of an asymptotic equation of a nonlinear variational
  wave equation.  
\newblock  {\em Asymptot.\ Anal.,}  18:307--327, 1998.

\bibitem{ZhaZhen:00}
P. Zhang and Y. Zheng.
\newblock Existence and uniqueness of solutions of an asymptotic equation
  arising from a variational wave equation with general data.
\newblock {\em Arch. Ration. Mech. Anal.}, 155(1):49--83, 2000.

\bibitem{ZhaZhen:01}
P. Zhang and Y. Zheng.
\newblock Rarefactive solutions to a nonlinear variational wave equation of liquid crystals.
\newblock {\em Comm. Partial Differential Equations}, 26(3\&4):381--419, 2001.

\bibitem{ZhaZhen:01a}
P. Zhang and Y. Zheng.
\newblock   Singular and
  rarefactive solutions to a nonlinear variational wave equation.
\newblock   {\em Chin. Ann. Math.,}  22B:159--170, 2001.
  
  
\bibitem{ZhaZhen:03}
P. Zhang and Y. Zheng.
\newblock Weak solutions to a nonlinear variational wave equation.
\newblock  {\em Arch.\ Rat.\ Mech.\ Anal.,}  166:303--319, 2003.
  
\bibitem{ZhaZhen:05a}
P. Zhang and Y. Zheng.
\newblock  Weak solutions to
  a nonlinear variational wave equation with general data. 
\newblock  {\em  Ann. Inst. H. Poincar\'e Anal. Non Lin{\'e}aire,} 22:207--226, 2005.

\bibitem{ZhaZhen:05}
P. Zhang and Y. Zheng.
\newblock On the global weak solutions to  a variational wave equation.
\newblock In {\em Handbook of Differential Equations. Evolutionary Equations. Volume 2.} (eds. C.\ M.\ Dafermos, E.\ Feireisl).
\newblock Elsevier, Amsterdam, 2005, pp. 561--648.


\end{thebibliography}

\def\ocirc#1{\ifmmode\setbox0=\hbox{$#1$}\dimen0=\ht0 \advance\dimen0
  by1pt\rlap{\hbox to\wd0{\hss\raise\dimen0
  \hbox{\hskip.2em$\scriptscriptstyle\circ$}\hss}}#1\else {\accent" 17
  #1}\fi}\def\cprime{$'$}

\end{document}